\documentclass[mnsc,nonblindrev]{informs3}
\OneAndAHalfSpacedXI % current default line spacing
\usepackage{endnotes}
\let\footnote=\endnote

\makeatletter
\AtBeginDocument{%
  \everypar{\looseness=-1\relax}%
}
\makeatother

\usepackage{natbib}
\usepackage{algorithm}
\usepackage{algpseudocode}
\usepackage{mathtools}
\usepackage{enumitem}
\usepackage{booktabs}
\usepackage{bbm}
\usepackage{float}

 \bibpunct[, ]{(}{)}{,}{a}{}{,}%
 \def\bibfont{\small}%
 \newcommand{\action}{\mathcal{A}}
\newcommand{\Sspace}{\ensuremath{\mathcal{S}}}

\newcommand{\tset}{\mathcal{T}}

\newcommand{\E}{\mathbb{E}}

\newcommand{\Endo}{\mathcal{X}}
\newcommand{\Exo}{\mathcal{W}}
\newcommand{\U}{\mathcal{U}}
\DeclareMathOperator{\diag}{diag}
\newcommand{\R}{\mathbb{R}}

\newcommand{\probP}{\text{I\kern-0.15em P}}
\newcommand{\sx}{w}
\newcommand{\se}{x}

\newcommand{\selva}[2][]{%
  \textcolor{blue}{#2}%                % main text
  \if\relax\detokenize{#1}\relax       % test whether note is empty
  \else
    \marginpar{\raggedright\scriptsize #1}% right-hand note
  \fi
}

\newcommand{\negar}[2][]{%
  \textcolor{red}{#2}%                % main text
  \if\relax\detokenize{#1}\relax       % test whether note is empty
  \else
    \marginpar{\raggedright\scriptsize #1}% right-hand note
  \fi
}

\TheoremsNumberedThrough     % Preferred (Theorem 1, Lemma 1, Theorem 2)
\ECRepeatTheorems
\EquationsNumberedThrough   
\begin{document}
%%%%%%%%%%%%%%%%

% Outcomment only when entries are known. Otherwise leave as is and
%   default values will be used.
%\setcounter{page}{1}
%\VOLUME{00}%
%\NO{0}%
%\MONTH{Xxxxx}% (month or a similar seasonal id)
%\YEAR{0000}% e.g., 2005
%\FIRSTPAGE{000}%
%\LASTPAGE{000}%
%\SHORTYEAR{00}% shortened year (two-digit)
%\ISSUE{0000} %
%\LONGFIRSTPAGE{0001} %
%\DOI{10.1287/xxxx.0000.0000}%

% Author's names for the running heads
% Sample depending on the number of authors;
% \RUNAUTHOR{Jones}
% \RUNAUTHOR{Jones and Wilson}
% \RUNAUTHOR{Jones, Miller, and Wilson}
% \RUNAUTHOR{Jones et al.} % for four or more authors
% Enter authors following the given pattern:
\RUNAUTHOR{Soheili, Nadarajah, Yang}

% Title or shortened title suitable for running heads. Sample:
% \RUNTITLE{Bundling Information Goods of Decreasing Value}
% Enter the (shortened) title:
\RUNTITLE{Weakly Time-Coupled Approximation of MDPs}

% Full title. Sample:
\TITLE{Weakly Time-Coupled Approximation of Markov Decision Processes}
%\TITLE{Revisiting Model Selection in Optimization Based Approximate Dynamic Programming}

% Block of authors and their affiliations starts here:
% NOTE: Authors with same affiliation, if the order of authors allows,
%   should be entered in ONE field, separated by a comma.
%   \EMAIL field can be repeated if more than one author
\ARTICLEAUTHORS{%
\AUTHOR{Negar Soheili}
\AFF{Information and Decision Sciences Department, University of Illinois at Chicago, Chicago, IL, 60607, \EMAIL{nazad@uic.edu}}%, \URL{}}
\AUTHOR{Selvaprabu Nadarajah}
\AFF{Information and Decision Sciences Department, University of Illinois at Chicago, Chicago, IL, 60607, \EMAIL{selvan@uic.edu}} %, \URL{}}
\AUTHOR{Bo Yang}
\AFF{Industrial Engineering and Decision Analytics, Hong Kong University of Science and Technology, Clear Water Bay, Hong Kong SAR, \EMAIL{yangb@ust.hk}}%, \URL{}}
} 

\ABSTRACT{Finite-horizon Markov decision processes (MDPs) with high-dimensional exogenous uncertainty and endogenous states arise in operations and finance, including the valuation and exercise of Bermudan and real options, but face a scalability barrier as computational complexity grows with the horizon. A common approximation represents the value function using basis functions, but methods for fitting weights treat cross-stage optimization differently. Least squares Monte Carlo (LSM) fits weights via backward recursion and regression, avoiding joint optimization but accumulating error over the horizon. Approximate linear programming (ALP) and pathwise optimization (PO) jointly fit weights to produce upper bounds, but temporal coupling causes computational complexity to grow with the horizon. We show this coupling is an artifact of the approximation architecture, and develop a weakly time-coupled approximation (WTCA) where cross-stage dependence is independent of horizon. For any fixed basis function set, the WTCA upper bound is tighter than that of ALP and looser than that of PO, and converges to the optimal policy value as the basis family expands. We extend parallel deterministic block coordinate descent to the stochastic MDP setting exploiting weak temporal coupling. Applied to WTCA, weak coupling yields computational complexity independent of the horizon. Within equal time budget, solving WTCA accommodates more exogenous samples or basis functions than PO, yielding tighter bounds despite PO being tighter for fixed samples and basis functions. On Bermudan option and ethanol production instances, WTCA produces tighter upper bounds than PO and LSM in every instance tested, with near-optimal policies at longer horizons.}

\maketitle

%%%%%%%%%%%%%%%%%%%%%%%%%%%%%%%%%%%%%%%%%%%%%%%%%%%%%%%%%%%%%

\section{Introduction}

Finite-horizon Markov Decision Processes (MDPs) arise in applications where decisions shape the evolution of endogenous system states while exogenous drivers follow stochastic dynamics \citep{hernandez2012discretetimeMDP, puterman2014}. Problems such as investment timing, switching between operating regimes, staged development projects, and commodity storage fit naturally within this framework \citep{dixit1994investment, trigeorgis1996real}. The endogenous state encodes what has been committed or exercised and determines which actions remain feasible, while the exogenous state captures market conditions or environmental drivers that evolve independently of the decision maker's actions. The interaction between controlled endogenous transitions and exogenous uncertainty generates dynamic tradeoffs.

When endogenous transitions are largely reversible, for example, adjusting inventory levels in commodity storage \citep{thompson2009storage,secomandi2010optimal}, a natural heuristic is to periodically solve a deterministic or scenario-based model over a rolling horizon and implement only the first-stage decision \citep{chand2002forecast}. Such reoptimization approaches can perform well in these settings because short-term commitments do not significantly constrain future decisions. In contrast, problems involving irreversible transitions, such as investing, deferring, abandoning, or switching operating regimes, require reasoning about future uncertainty across multiple periods. This intertemporal dependence under evolving exogenous uncertainty can make reoptimization perform poorly and motivates approximations that explicitly account for uncertainty.

A common uncertainty-aware strategy is to approximate the value function at each stage using a linear combination of basis functions. Methods such as least squares Monte Carlo (LSM; \citealp{carriere1996valuation, longstaff2001valuing, tsitsiklis2001regression}), approximate linear programming (ALP; \citealp{schweitzer1985, de_farias_linear_2003}), and pathwise optimization (PO; \citealp{desai2012pathwise}) all follow this approach to construct a value function approximation but differ in how they fit basis function weights across stages. LSM proceeds backward stage by stage, fitting basis function weights to reduce the regression error with respect to continuation value estimates on sample paths of exogenous uncertainty. Because the approximation at each stage serves as input to earlier regressions, errors propagate backward through the horizon, and errors at later stages can distort value estimates at the early stages where accuracy matters most for impending decisions \citep{glasserman2004simulation}. \looseness=-1

In contrast, ALP and PO optimize basis function weights across stages jointly and yield upper bounds on the optimal policy value. ALP enforces constraints derived from Bellman equations across stages, states, and actions. PO directly optimizes an upper bound through the information-relaxation and duality framework \citep{brownsmith2022}, where actions along trajectories of exogenous sample paths are chosen to maximize rewards penalized for knowledge of future information. As the horizon grows, the computational complexity grows nonlinearly in both formulations.\looseness=-1

A central insight of this paper is that when basis function weights are optimized across all stages jointly, the cross-stage dependence among weights, which we term {\emph{temporal coupling}}, need not be dictated by the structure of the MDP but can be controlled through the approximation architecture. Exploiting this observation, we develop an MDP approximation whose temporal coupling is weak, meaning it is not a function of the horizon length. This formulation preserves the upper-bound property of ALP and PO while enabling first-order methods and parallel computation to scale.

Our contributions fall into the following broad categories:

\begin{itemize}

\item \textbf{Weakly Time-Coupled Approximation (WTCA).} We introduce a unifying stochastic optimization formulation to formally define a temporal coupling factor and show that this factor is equal to the horizon length for PO and ALP. That is, both approximations are fully time coupled but for different reasons. We leverage the nature of coupling in ALP and show that it can be relaxed to obtain a ``weakly-time" coupled approximation where the cross-stage dependence among basis function weights is limited to two stages, and thus does not grow with the horizon length. We prove that, for any fixed set of basis functions, the WTCA upper bound is at least as tight as ALP's upper bound, and converges to the optimal value as the richness of the
of basis function set expands. 

\item \textbf{Parallel Stochastic Block Coordinate Descent Algorithm.} Block coordinate descent partitions the basis function weights into blocks and updates each block separately. In its parallel variant, all blocks are updated independently and simultaneously, but the convergence benefit of parallelization depends on the coupling structure among blocks. We extend parallel deterministic block coordinate descent to the stochastic MDP setting in a manner that exploits the weak temporal coupling of WTCA. This parallel stochastic block coordinate descent (PS-BCD) algorithm partitions basis function weights into one block per stage and at each iteration updates all stage blocks in parallel. Two properties make PS-BCD effective for WTCA. First, because WTCA is weakly coupled, the information lost by updating blocks in parallel rather than jointly as in stochastic gradient decent (SGD) is negligible, so the iteration complexity remains independent of the horizon. Second, parallelization removes the per-iteration cost dependence on the horizon. Together, the total computational complexity of using PS-BCD to solve WTCA is independent of horizon length. For PO, full temporal coupling undermines both properties. The information loss from parallel block updates is substantial, worsening iteration complexity compared to SGD. As a result, PS-BCD applied to PO has higher computational complexity than SGD because the information loss under strong coupling outweighs the savings from parallelization.\looseness=-1

\item \textbf{Model Selection Under Time Budgets.} We show for any fixed basis function set, PO yields the tightest bound followed by WTCA and then ALP. This ordering assumes the expectations are computed exactly and each formulation is solved to optimality. In practice, stochastic first-order methods such as SGD and PS-BCD draw one sample path per iteration, so the total number of sample paths processed and the number of basis functions that can be supported both depend on the computational complexity of solving the formulation. Because WTCA is weakly coupled, PS-BCD can solve it at a computational cost independent of the horizon and SGD at cost linear in the horizon, whereas the full coupling in PO forces computational cost to grow nonlinearly with the horizon under both approaches. As the horizon lengthens, more sample paths can be processed within a given budget when solving WTCA than PO. Analogously, we show that WTCA can incorporate more basis functions within a fixed budget, which is a further differentiator. Therefore, although PO gives the tighter upper bound for any fixed set of basis functions and exogenous samples, WTCA can leverage more samples and basis functions per unit of computation, allowing us to establish that the WTCA upper bound for a fixed computational time budget can be tighter than PO. Since ALP is more strongly coupled than WTCA and yields a weaker upper bound, it is dominated on both counts. 

\item \textbf{Empirical Validation.} We evaluate WTCA against PO and LSM on instances of multi-asset Bermudan options from \citet{desai2012pathwise} and merchant ethanol production from \citet{guthrie2009real}, two settings involving (irreversible) sequential timing decisions. We solve WTCA using PS-BCD and PO using stochastic gradient descent, each under equal time budgets. WTCA produces the tightest upper bounds across all instances tested, confirming the model selection reversal noted above. That is, despite PO being theoretically tighter than WTCA in terms of upper bounds for any fixed set of basis functions and exogenous samples, WTCA produces tighter bounds in every instance tested. The policy quality of WTCA and PO is comparable in most instances, with WTCA yielding near-optimal policies throughout and showing the largest advantage at smaller numbers of assets and longer horizons where PO's convergence disadvantage is most pronounced. Both WTCA and PO consistently outperform LSM in both upper bound and policy quality.

\end{itemize}

In summary, these findings demonstrate that computational structure can be as important as approximation tightness in determining practical solution quality, and that progress in optimization-based MDP methods can arise not only from tighter relaxations, but also from rethinking how approximation architecture and computation interact. Under equal computational time budgets and with the ability to parallelize stage updates, WTCA produces tighter bounds than PO in every instance tested, with the advantage widening as the horizon grows. To put this in context, GPU price-performance (measured in floating-point operations per second per dollar) has doubled approximately every 2.5 years over the past two decades \citep{hobbhahn2022gpu}, a trend that makes parallel algorithms increasingly attractive.

Our work is novel with respect to several streams of research on MDPs, first-order methods, and approximate dynamic programming. 

The first stream studies structured MDPs whose tractability follows from special problem structure. Among them, weakly coupled MDPs, in which multiple subsystems interact through linking constraints, have been studied \citep{hawkins2003lagrangian, adelman2008, brown2023strength, nadarajah2025selfadapting}. In these problems, relaxing or reformulating the linking constraints yields tractable subproblems. Our notion of weak temporal coupling is new. The temporal coupling structure resides not in the MDP itself but in the approximation architecture, and the approach applies to finite-horizon MDPs without requiring special problem structure. To the best of our knowledge, WTCA is the first optimization-based MDP approximation that maintains both an upper-bound guarantee and vanishing approximation error while being weakly coupled. \looseness=-1

A second applies first-order methods to MDP approximations. \citet{brown2014information} construct gradient-based dual penalties within the information-relaxation framework for convex stochastic dynamic programs, \citet{lin2019revisitingALP} use mirror descent to solve ALP, and \citet{yang2024least,yang2025improved} apply block coordinate descent and an ADMM-based regression method, respectively, to solve PO. These contributions improve scalability but do not address temporal coupling, so computational complexity still grows with the horizon. Moreover, while \citet{desai2012pathwise}, \citet{yang2024least}, and \citet{yang2025improved} solve a sample average approximation of PO, we consider stochastic approximation, which has not been tested in this context. There is indeed extant literature on block coordinate descent for convex optimization \citep{peng2013parallel,richtarik2014iteration,SBMD,xu2015block,AccCD,richtarik2016parallel,fercoq2017smooth}. Stochastic variants that combine block updates with stochastic gradients have been developed by \citet{SBMD} and \citet{xu2015block}, but these methods do not exploit separability structure in the objective. \citet{richtarik2014iteration}, \cite{AccCD}, and \cite{richtarik2016parallel} are the only works to our knowledge that analyze how the coupling structure among blocks affects the convergence of parallel block coordinate descent, but their analysis is restricted to the deterministic setting. Our work extends their separability-aware parallel block coordinate descent to the stochastic MDP setting and establishes that its convergence guarantees capture temporal coupling. This stochastic extension is of independent interest to the first-order method literature, and its use with WTCA to show horizon-length independent solution complexity is new to the MDP approximation literature. 

Finally, we add to approximate dynamic programming for the application class in our numerical study, namely, financial options and real options. The most common method here is arguably LSM \citep{longstaff2001valuing, tsitsiklis2001regression}, with extensions using alternative regression surrogates such as Gaussian processes \citep{gramacy2015sequential}. Upper bounds have been obtained via the duality approach for American options \citep{rogers2002,haugh2004} and generalized to a broader class of MDPs by the information relaxation framework \citep{brown2010information,brownsmith2022,chen2024information}. These methods have been applied to multi-asset options \citep{broadie1997valuation} and to manage multiple exercise rights in commodity and energy markets \citep{meinshausen2004,jaillet2004valuation,carmona2008,  nadarajah2017comparison}. WTCA and PS-BCD add to approaches that move beyond backward recursion and seek to fit basis function weights simultaneously across stages. Here, in addition to PO and ALP that rely on parametrizing approximations using linear combinations of basis functions, there are approaches that employ nonlinear approximation architectures such as neural networks (see, e.g., \citealp{becker2019deep,bayer2021}). In these methods, computational complexity grows with horizon length, unlike our work. WTCA also adds a new ALP relaxation that provides an upper bound on the optimal policy, expanding limited work on ALP relaxations in the literature \citep{desai2012,nadarajah2015relaxations}.\looseness=-1

In \S\ref{sec:fh-mdp-optimization}, we formalize the MDP setting. In \S\ref{sec:opt-approximations}, we analyze temporal coupling in ALP and PO and introduce WTCA. In \S\ref{sec:stoch-bcd}, we develop PS-BCD. In \S\ref{sec:numerical}, we present numerical experiments. We conclude in \S\ref{sec:Conclusion}. All proofs are in the Electronic Companion.

\section{Finite-Horizon Markov Decision Processes}\label{sec:fh-mdp-optimization}

We consider a finite-horizon MDP with discrete stages $\tset := \{0,1,\ldots,T-1\}$ and discount factor $\gamma\in(0,1)$. At stage $t$, the system is in state $s_t\in\Sspace_t$. After selecting an action $a_t$ from a finite set $\action_t$, the decision-maker receives reward $r_t(s_t,a_t)$ and the system transitions to a new state $s_{t+1}$. The state consists of two components. The endogenous component $\se_t$ belongs to the finite set $\Endo_t$ and transitions deterministically to $\se_{t+1}=h(\se_t,a_t)$ under an action $a_t$. In contrast, the exogenous component $\sx_t$ is in set $\Exo_t$, which could be continuous, and evolves according to a Markovian stochastic process $\mu$ that is independent of the action $a_t$. The initial state at stage $0$ is the singleton $(\se_0,\sx_0)$. The stochastic evolution of the exogenous state induces a probability measure on trajectories $\sx = (\sx_0,\ldots,\sx_{T-1})$. Throughout the paper, $\E[\cdot]$ denotes expectation under $\mu$ unless a subscript specifies a different measure.
\looseness=-1

A policy is a sequence of decision rules
$\pi=\{\pi_t\}_{t=0}^{T-1}$, where each
$\pi_t\!\!:\Sspace_t\to\action_t$ maps the current state to an action.
Given an initial state $s_0$, a policy $\pi$ induces a random trajectory
$(s_0,a_0,s_1,a_1,\ldots,s_{T-1},a_{T-1},s_T),$
where $a_t=\pi_t(s_t)$. The value of policy $\pi$ starting from $s_0$ is
\[
V_0^\pi(s_0)
= \E\!\left[\,\sum_{t=0}^{T-1}\gamma^t
  r_t(s_t,\pi_t(s_t))\,\Big|\,s_0\right],
\]
where the expectation is taken with respect to the exogenous process. An optimal policy solves
\begin{equation*}
\pi^*\in\arg\sup_{\pi\in\Pi}\; V_0^\pi(s_0),
\end{equation*}
where $\Pi$ denotes the set of all admissible state-dependent policies. We denote the optimal policy value by $V_0^*(s_0) := V_0^{\pi^*}(s_0)$. Analogously, the MDP value function $V_t^*(s_t)$ at stage $t\in\tset$ and state $s_t\in\Sspace_t$ satisfies
\[
V^*_t(s_t)
=
\E\!\left[
\sum_{t'=t}^{T-1}
\gamma^{t'-t}
\, r_{t'}(s_{t'},\pi^*_{t'}(s_{t'}))
\;\Big|\; s_t
\right].
\]

In applications such as inventory management or commodity storage, the
endogenous transition is \emph{reversible}; for instance, a storage position built up in one period can be drawn down in the next, so a single action does not permanently restrict which endogenous states remain reachable.
In such settings, a \emph{rolling-horizon} (reoptimization) approach often yields satisfactory performance. At each stage, the decision-maker solves a deterministic optimization over a deterministic forecast of future exogenous states, implements the immediate action, and resolves in the next period with updated information~\citep{chand2002forecast,thompson2009storage,devalkar2011dynamic,powell2011}.\looseness=-1

In contrast, problems such as investment timing, project abandonment, or regime switching involve \emph{irreversible} endogenous transitions, where an action at stage~$t$ permanently alters the set of endogenous states
reachable
thereafter~\citep{pindyck1991irreversibility,dixit1994investment,trigeorgis1996real}.
Because such decisions cannot be undone, a single suboptimal action can be highly consequential, and the decision-maker must weigh immediate rewards against the value of preserving future flexibility. Reoptimization, which does not account for this trade-off, tends to perform poorly in these settings.
Solving such problems exactly via dynamic programming is generally intractable for practically relevant state-space dimensions due to the curse of dimensionality~\citep{puterman2014}, motivating approximation methods that account for uncertainty while keeping the problem computationally tractable.\looseness=-1

A common approximation strategy is to represent the MDP value function at each stage as a linear combination of prespecified basis functions. For each stage $t\in\tset$, let $\phi_t = (\phi_{t,1},\ldots,\phi_{t,B})$ denote a vector of basis functions, where each $\phi_{t,b}$ is defined on $\Sspace_t$.
The approximate value function at stage $t$ is then
\[
\hat V_{t,\beta}(s_t)
:= \sum_{b=1}^B \beta_{t,b}\,\phi_{t,b}(s_t),
\]
with terminal condition $\hat V_{T,\beta}\equiv 0$, and where $\beta_{t,b}$ denotes the $b$-th basis function weight.
For each $t$, let $\beta_t := (\beta_{t,1},\ldots,\beta_{t,B})\in\R^B,$
and define
$\beta=(\beta_0,\ldots,\beta_{T-1})\in\R^{BT}$.

Given such an approximation, a policy can be obtained by acting greedily with respect to the reward plus discounted continuation value under the approximation at each period and state:
\begin{equation}\label{eq:greedy-policy}
\pi_{t}^\beta(s_t)
= \argmax_{a_t\in\action_t}
\{r_t(s_t,a_t)
+ \gamma\E[\hat V_{t+1,\beta}(s_{t+1})\vert s_t,a_t]\}.
 \end{equation}
Simulating the greedy (feasible) policy forward on
independent sample paths provides an estimate of its value and hence a lower bound on~$V^*_0(s_0)$.\looseness=-1

A popular method for fitting value function approximations in problems with timing decisions is LSM~\citep{longstaff2001valuing,tsitsiklis2001regression,glasserman2004monte}. This method proceeds backward in time, starting from the terminal stage and working toward $t=0$. At each stage $t$, it regresses the discounted values assigned by the
already-fitted stage-$(t+1)$ approximation onto the current stage basis functions by solving
\[
\min_{\beta_t\in\R^B}\;
\sum_{i=1}^P
\Big(Y_{t+1}^i
     -\sum_{b=1}^B\beta_{t,b}\phi_{t,b}(s_t^i)\Big)^2,
\]
where $Y_{t+1}^i = \gamma\,\hat V_{t+1,\beta}(s_{t+1}^i)$ is the
discounted approximate value at the next-stage state along the $i$-th sample path and $P$ is the number of paths. The fitted approximation $\hat V_{t,\beta}$ then serves as the continuation-value estimate for stage $t$, used both to construct the regression targets at stage $t-1$ and to define the policy at stage $t$ via~\eqref{eq:greedy-policy}. Because each stage involves only a local regression, LSM is computationally attractive. However, since the regression targets at stage~$t$ depend on the approximation fitted at stage $t+1$, errors accumulate backward through the horizon, degrading accuracy at the early stages where decisions are often most consequential~\citep{clement2002analysis,egloff2005monte,yang2024least}.\looseness=-1

\section{Weakening Temporal Coupling in Optimization-Based Approximations}%
\label{sec:opt-approximations}
Optimization-based methods such as ALP and PO avoid the error accumulation in LSM by optimizing basis function weights across all stages jointly. To understand the implications of this coupled optimization, we cast
ALP and PO within a unified stochastic optimization framework in
\S\ref{sec:global-approx}, show that they are fully coupled across time in
\S\ref{sec:PO-ALP}, and introduce a new approximation with ``weak" temporal coupling in \S\ref{sec:weakly-time-model}. We discuss in \S\ref{subsec:BoundOrderingAndModSelection} why model selection must account for how temporal coupling affects computational cost.

\subsection{A Unified Stochastic Optimization Framework}\label{sec:global-approx}

Both ALP and PO can be analyzed within a unified stochastic optimization framework of the form\looseness =-1
\begin{equation}
\label{eq:general-stochastic-form}
\min_{\beta\in\Omega}\;
F(\beta)
:=\E_\xi\!\left[\,\sum_{j=1}^m f_j(\beta,\sx)\,\right],
\end{equation}
where $\E_\xi$ denotes expectation under a probability measure $\xi$ on $\sx$ that can differ from $\mu$, the feasible set $\Omega\subseteq\R^{BT}$ is a sufficiently large compact convex set, such as a box or Euclidean ball, whose interior contains the optimal solutions of the approximation models and onto which projection is available in closed form, and the component functions $f_j(\beta,\sx)$ are convex in $\beta$ and determine how basis function weights across stages are coupled. 

The key structural feature is how stage-level weight blocks $\beta_t$, $t \in \tset$, appear jointly within the component functions $f_j$. We formalize this as temporal coupling.\looseness=-1
\begin{definition}[Weak and Full Temporal Coupling]\label{def:temporalcoupling} For each component function $f_j$ in~\eqref{eq:general-stochastic-form}, let $C_j \subseteq \{0,1,\ldots,T-1\}$ denote the set of stages whose stage-level weight blocks $\beta_t$ appear in $f_j$, and define the \emph{component coupling degree} $\kappa_j$ as the cardinality of $C_j$. The \emph{temporal coupling width} of $F$ is \looseness=-1
\[
\kappa(F) := \max_{j=1,\ldots,m} \kappa_j.
\]
When $\kappa(F)=1$, the objective decomposes by stage. When $\kappa(F)$ is bounded independently of the horizon length $T$, we call the model \emph{weakly coupled}. When $\kappa(F)=T$, the model is \emph{fully coupled}. 
\end{definition}
This definition is motivated by partial separability in block-coordinate optimization~\citep{richtarik2016parallel,AccCD}, specialized here to stage-indexed weight blocks. As we will establish in \S\ref{sec:stoch-bcd}, full temporal coupling causes overall solution complexity to grow nonlinearly in $T$, whereas weak temporal coupling across $T$ component functions can be exploited via parallelization to keep complexity independent of $T$. To enable this analysis, we first need to characterize the temporal coupling of ALP, PO, and WTCA, which is the focus of the remainder of this section.

\subsection{Fully Time-Coupled Models: Pathwise Optimization and Approximate Linear Programming}\looseness=-1%
\label{sec:PO-ALP}

We formulate PO and ALP within the framework of \eqref{eq:general-stochastic-form} and show that both are fully coupled with $\kappa(F)=T$, though the coupling arises for different reasons. 

\paragraph{Pathwise optimization.}
PO builds on the information-relaxation duality framework, which constructs valid upper bounds on $V^*_0(s_0)$ by allowing a decision-maker to observe future exogenous information $\sx=(\sx_0,\sx_1,\ldots,\sx_{T-1})$ before choosing actions, while imposing penalties that offset the resulting informational advantage~\citep{brown2010information,desai2012pathwise,yang2024least}.

For a fixed weight vector $\beta$, define the stagewise information-relaxation penalty
\begin{equation}
\label{eq:info-relaxpenalty}
\hat{\mathcal P}_t(\se_t,a_t,\sx_t,\sx_{t+1})
:=
\gamma
\Big(
\hat V_{t+1,\beta}(\se_{t+1},\sx_{t+1})
-
\E\!\left[
\hat V_{t+1,\beta}(\se_{t+1},\sx_{t+1})
\mid \sx_t
\right]
\Big),
\end{equation}
where the endogenous state evolves according to $\se_{t+1}=h(\se_t,a_t)$.
The penalty $\hat{\mathcal P}_t$ offsets the informational advantage created by allowing the decision-maker to observe $\sx_{t+1}$ before choosing $a_t$ at state $(\se_t,\sx_t)$. The first term is the continuation value under the realized next-period exogenous state, while the second term subtracts the expected continuation value given the current information $\sx_t$. By construction,
$\E[\hat{\mathcal P}_t \mid \sx_t]=0$, so the penalty does not alter the expected value of any non-anticipative policy.

Given a realized exogenous vector $\sx$, consider the penalized
anticipative problem
\begin{equation}
\label{eq:po-penalized}
f(\beta,\sx)
:=
\max_{a \in \action}
\sum_{t=0}^{T-1}
\gamma^t
\left[
r_t(s_t,a_t)
-
\hat{\mathcal P}_t(\se_t,a_t,\sx_t,\sx_{t+1})
\right],
\end{equation}
where
$\action := \action_0 \times \cdots \times \action_{T-1}$ denotes the set of full-horizon action vectors, and $a=(a_0,\ldots,a_{T-1})$ with $a_t\in\action_t$. A valid upper bound on $V_0^*(s_0)$
\citep{brown2010information,desai2012pathwise} is given by the expectation of $f(\beta,\sx)$: 
\[
F_{\mathrm{PO}}(\beta)
:=
\E\!\left[f(\beta,\sx)\right].
\]
The PO formulation minimizes this upper bound over $\beta$ by solving
\begin{equation}
\label{eq:po-objective}
\min_{\beta\in\Omega}\;
F_{\mathrm{PO}}(\beta),
\end{equation}
yielding the tightest upper bound attainable when the dual penalty is constructed using the chosen value function approximation family. PO fits directly into~\eqref{eq:general-stochastic-form} with random vector $\sx$ distributed according to $\xi = \mu$, and a single component function $f_1(\beta,\sx)=f(\beta,\sx)$ in \eqref{eq:po-penalized}, that is, $m=1$.\looseness=-1

For each realization $\sx$, the maximization in~\eqref{eq:po-penalized} jointly optimizes the entire action sequence
$(a_0,\ldots,a_{T-1})$, and the resulting objective depends on all weight blocks $(\beta_0,\ldots,\beta_{T-1})$: for each action sequence, the information-relaxation penalties collectively involve $\hat V_{t+1,\beta}$ at every stage, so the maximum over all action sequences couples the full weight vector jointly (i.e., $C_1 = \{0,1,\ldots,T-1\}$). This gives $\kappa(F_{\mathrm{PO}})=\kappa_1=T$.

\paragraph{Approximate linear programming.}
ALP computes a value function approximation $\hat V_{t,\beta}$ at each stage $t$ that majorizes the MDP value functions $V^*_t$ and thus provides an upper bound on $V^*_0(s_0)$. For every stage $t$, state $s_t=(\se_t,\sx_t)$, and action $a_t$, it enforces a ``Bellman-type'' inequality requiring that the value function approximation at this state is no less than the immediate reward plus the discounted continuation value at the next stage. For each realization of the exogenous state $\sx_t$, define the maximum
Bellman deviation over the endogenous state and action as
\begin{equation}\label{eq:Delta-Bellmanerror}
\Delta_t(\beta;\sx_t)
:=
\max_{(\se_t,a_t)\in\Endo_t\times\action_t}
\left\{
r_t(s_t,a_t)
+\gamma\,\E\!\left[
\hat V_{t+1,\beta}(s_{t+1})\Big\vert s_t,a_t\right]
-
\hat V_{t,\beta}(s_t)
\right\}.
\end{equation}
The ALP formulation is equivalent to
\begin{equation}
\label{eq:ALP-def}
\min_{\beta\in\Omega}
\hat V_{0,\beta}(s_0)
\quad\text{s.t.}\quad
\Delta_t(\beta;\sx_t)\le 0,
\quad
\forall\,t\in\tset,\sx_t\in\Exo_t.
\end{equation}

\begin{proposition}\label{prop:ALP-saddlepoint}
Let $\mathcal D$ denote the set of sequences
$\alpha = (\alpha_0,\ldots,\alpha_{T-1})$,
where each $\alpha_t$ is a probability measure on $\Exo_t$.
Suppose $\Omega$ is chosen large enough that the optimal
solution of~\eqref{eq:ALP-def} lies in its interior and $\phi_{t,1} \equiv 1$ for all $t \in \tset$. Then the optimal values of~\eqref{eq:ALP-def} and the saddle-point problem
\begin{equation}\label{eq:alp-saddlepoint}
\min_{\beta\in\Omega}\;
\sup_{\alpha\in \mathcal D}\;
\hat V_{0,\beta}(s_0)
+
\sum_{t=0}^{T-1}
\gamma^t\,
\E_{\alpha_t}\!\left[\Delta_t(\beta;\sx_t)\right]
\end{equation}
are equal, and their sets of optimal solutions in $\beta$ coincide.
\end{proposition}

The saddle-point formulation expresses ALP as an unconstrained minimization over the composite objective\looseness=-1
\[
\min_{\beta\in\Omega}
F_{\mathrm{ALP}}(\beta)
:=
\sup_{\alpha\in\mathcal D}
\hat V_{0,\beta}(s_0)
+
\sum_{t=0}^{T-1}
\gamma^t\,
\E_{\alpha_t}\!\left[
\Delta_t(\beta,\sx_t)
\right].
\]
If the supremum is attained at some optimizer $\alpha^*(\beta)$,
then
$F_{\mathrm{ALP}}(\beta)
=
\hat V_{0,\beta}(s_0)
+
\sum_{t=0}^{T-1}
\gamma^t\,
\E_{\alpha_t^*(\beta)}\!\left[
\Delta_t(\beta,\sx_t)
\right],$
where $\alpha^*(\beta)$ denotes the optimizing dual sequence and therefore depends on the entire weight vector $\beta$.

From the perspective of~\eqref{eq:general-stochastic-form}, ALP corresponds to the special case $m=1$. Although each Bellman deviation $\Delta_t(\beta;\sx_t)$ individually involves only adjacent stage-level blocks $(\beta_t,\beta_{t+1})$, the formulation cannot be decomposed into $T$ independent stagewise components because the optimizing measure $\alpha^*(\beta)$ depends on $\beta$. Specifically, $\alpha^*_t(\beta)$ places mass on exogenous realizations $\sx_t$ where the Bellman violation $\Delta_t(\beta;\sx_t)$ is largest, and this worst-case selection couples the measure to the weight vector. As a result, under the $\beta$-dependent measure $\alpha^*(\beta)$, the composite objective $F_{\mathrm{ALP}}$ is a single component function that depends jointly on the full vector $(\beta_0,\ldots,\beta_{T-1})$, giving $\kappa(F_{\mathrm{ALP}})=T$ under Definition~\ref{def:temporalcoupling}. As in PO, full temporal coupling arises through a single component function ($m=1$), but the source differs. In PO it is the trajectory-level maximization over actions, whereas in ALP it is the adversarial selection of the exogenous measure.\looseness = -1

\subsection{A Weakly Time-Coupled Approximation}\label{sec:weakly-time-model}

The source of full temporal coupling in ALP is not the local Bellman structure but the adversarial selection of the exogenous measure. Each Bellman deviation $\Delta_t(\beta;\sx_t)$ depends only on adjacent stage-level blocks $(\beta_t,\beta_{t+1})$, and full coupling arises only through the outer supremum over $\alpha \in \mathcal{D}$. If this supremum is replaced by a fixed expectation under the exogenous distribution, the local stagewise structure is preserved while horizon-wide coupling is substantially reduced. We refer to the resulting formulation as the \emph{Weakly Time-Coupled Approximation (WTCA)}.\looseness=-1

Specifically, WTCA replaces the supremum over dual measures in the saddle-point formulation~\eqref{eq:alp-saddlepoint} of ALP with an expectation under $\mu$, yielding
\begin{equation}\label{eq:WTCA-def}
\min_{\beta\in\Omega}\;
F_{\mathrm{WTCA}}(\beta)
:= \;
\hat V_{0,\beta}(s_0)
+
\sum_{t=0}^{T-1}\gamma^t\,\E\!\left[\Delta_t(\beta;\sx_t) \right].
\end{equation}
The first term represents the approximate value at the initial state, and the second penalizes Bellman deviations in expectation across stages discounted to stage $0$. Unlike ALP, which enforces $\Delta_t(\beta;\sx_t) \leq 0$ over all exogenous realizations, WTCA penalizes violations in expectation under the exogenous distribution $\mu$.
Importantly, the inner maximization over endogenous states and actions within each $\Delta_t$ is unchanged, so the local structure of the Bellman deviations is preserved.

To express WTCA in the unified framework~\eqref{eq:general-stochastic-form}, define the stagewise component functions
\begin{equation}\label{eq:f_WTCA,t}
f_{\mathrm{WTCA},t}(\beta,\sx_t)
:=
\begin{cases}
\hat V_{0,\beta}(s_0)+\Delta_0(\beta;\sx_0), & t=0,\\[7pt]
\gamma^t\,\Delta_t(\beta;\sx_t), & t\in\tset\setminus\{0\}.
\end{cases}
\end{equation}
Because each $f_{\mathrm{WTCA},t}$ depends only on $\sx_t$, linearity of expectation yields
\begin{equation}\label{eq:WTCA-formulation}
F_{\mathrm{WTCA}}(\beta)
=
\E\!\left[\sum_{t=0}^{T-1}
f_{\mathrm{WTCA},t}(\beta,\sx_t)\right].
\end{equation}
To formalize the temporal coupling of WTCA, the representation~\eqref{eq:WTCA-formulation} fits the general stochastic framework~\eqref{eq:general-stochastic-form} with $m=T$ and $\xi = \mu$, each component corresponding to a stagewise term $f_{\mathrm{WTCA},t}$. The component function $f_{\mathrm{WTCA},t}$ involves only adjacent stage-level blocks $(\beta_t,\beta_{t+1})$, and there is no outer supremum to induce full temporal coupling. Accordingly, the WTCA objective couples only consecutive stages (i.e., $\kappa(F_{\mathrm{WTCA}})=2$), which is independent of $T$ and thus constitutes weak coupling by Definition~\ref{def:temporalcoupling}. Unlike ALP and PO, where $m = 1$, we have $m= T$ in WTCA, but each component function is weakly time coupled.

Proposition \ref{prop:WTCA-comparison} compares WTCA to ALP and PO. For a fixed set of basis functions of size $B$, let $U_{\mathrm{PO},B}$,
$U_{\mathrm{WTCA},B}$, and $U_{\mathrm{ALP},B}$ denote the optimal objective function values of PO, WTCA, and ALP, respectively. \looseness=-1
\begin{proposition}\label{prop:WTCA-comparison}
The following hold for any fixed basis function set of size~$B$ and initial state $s_0$:
\begin{itemize}
    \item[(a)] WTCA is equivalent to a linear program that is a relaxation of ALP and provides a valid upper bound on the optimal policy value $V^*_0(s_0)$:
    \[
V^*_0(s_0)
\;\le\;
U_{\mathrm{WTCA},B}
\;\le\;
U_{\mathrm{ALP},B}.
\]
    \item[(b)] For optimal WTCA basis function weights $\beta^{\mathrm{WTCA}}$, we have
\[
V^*_0(s_0)
\;\le\;
U_{\mathrm{PO},B} \; \le\; F_{\mathrm{PO}}(\beta^{\mathrm{WTCA}})
\;\le\;
U_{\mathrm{WTCA},B}.
\]
\end{itemize} 
\end{proposition}

Part (a) shows that WTCA is a relaxation of ALP whose upper bound is at least as tight. Part (b) shows that evaluating the PO objective $F_{\mathrm{PO}}(\beta^{\mathrm{WTCA}})$ at the WTCA optimal weights $\beta^{\mathrm{WTCA}}$ yields an upper bound no worse than that of WTCA. \cite{brown2014information} establish a similar tightening of the ALP upper bound by constructing dual penalties from the ALP value function approximation via information relaxation and duality. That result leverages the fact that the ALP value function approximation is an upper bound on the MDP value function at each stage and state, a property that need not hold for WTCA since it is a relaxation of ALP. The upper bounding property of PO is directly shown in \citet{desai2012pathwise}, which also proves that the PO upper bound is tighter than the ALP upper bound. 

\subsection{Computational Tradeoffs and Model Selection}\label{subsec:BoundOrderingAndModSelection}

Proposition~\ref{prop:WTCA-comparison} establishes that, for a fixed set of basis functions, PO yields the tightest upper bound, followed by WTCA, and then ALP. However, this ranking assumes that expectations are computed exactly and that each model can be solved exactly, neither of which may hold in practical implementation. Under a fixed computational time budget, the effective quality of an upper bound depends not only on the tightness of the formulation but also on how many sample paths and basis functions can be processed within that budget. These two dimensions of model selection -- bound tightness and solution complexity -- can pull in opposite directions, and we discuss each in turn.

\paragraph{Sample paths.} Exogenous uncertainty must be handled approximately, either via sample average approximation, where a fixed set of sample paths is generated ahead of time, or via a stochastic first-order method, where each iteration draws a new sample path. In either case, processing more sample paths requires more computation. Given a fixed time budget, the number of sample paths that can be handled varies inversely with the solution complexity of the formulation. Since WTCA is weakly coupled across stages, if we construct an algorithm that exploits this structure to solve WTCA more efficiently than PO, it can process more sample paths within the same budget. It is therefore plausible that the estimated WTCA upper bound can be tighter than the PO upper bound in practice, despite being looser in theory for a fixed set of basis functions and exact expectations.

\paragraph{Basis functions.} Proposition~\ref{prop:WTCA-comparison} is stated for a fixed set of $B$ basis functions, but in practice, one can improve upper bounds by expanding this set. We next present a construction under which the upper bounds of all three models converge to the optimal policy value at the same rate in $B$. Let
$\mathcal V_B
:=
\{
\hat V_{\beta} : \beta \in \R^{BT}
\}$
denote the class of value function approximations representable using $B$ basis functions. As $B$ increases, we require $\mathcal V_B$ to be dense in the space of bounded continuous value functions under the sup-norm, i.e.,
\[
\inf_{\hat V\in\mathcal V_B}
\| \hat V - V^* \|_\infty
\;\longrightarrow\; 0
\quad \text{as} \quad B\to\infty.
\]
Random basis functions, parametrized by a vector $\theta$, provide a natural mechanism for constructing such $\mathcal V_B$. Given a distribution $\Theta$, one draws $B$ i.i.d.\ samples $\{\theta_b\}_{b=1}^B$ from $\Theta$ and defines $B$ basis functions as\looseness=-1
\[
\phi_{t,b}(s_t) := \psi_t(s_t;\theta_b), \quad\forall b = 1,\ldots,B.
\]
For example, $\psi_t(\cdot;\theta)$ may be a random Fourier basis function $\psi_t(s;\theta)=\cos(\theta^\top s+\theta_0)$, where $\theta$ is drawn from the Gaussian distribution $\mathcal N(0,\varrho)$ and $\theta_0$ is drawn independently from the uniform distribution on $[-\pi,\pi]$.

Under standard regularity conditions for random basis functions \citep{barron1993,rahimi2008weighted,bach2017,pakiman2024}, Assumption~\ref{assump:random-basis functions} formalizes the requirements and Proposition~\ref{prop:asymptotic-consistency-random} establishes that ALP, WTCA, and PO share the same convergence rate.

\begin{assumption}
\label{assump:random-basis functions}
Suppose that:
\begin{enumerate}
    \item The state space $\Sspace_t$ is compact and the optimal value function $V^*$ is continuous.
    \item Rewards are uniformly bounded.
    \item For each stage $t$, the basis function map $s_t \mapsto \psi_t(s_t;\theta)$ is uniformly bounded and uniformly Lipschitz in $s_t$ for all $\theta$.
    \item The random basis functions $\{\phi_{t,b}\}_{b=1}^B$ are generated i.i.d.\ from a distribution $\Theta$ such that the induced basis class is dense in $C(\Sspace_t)$.
\end{enumerate}
\end{assumption}

\begin{proposition}
\label{prop:asymptotic-consistency-random}
Let $\delta \in (0,1)$. Under Assumption~\ref{assump:random-basis functions}, for each of ALP, WTCA, and PO, the upper bound gap $U_{\cdot,B} - V^*_0(s_0)$ is
\[
\mathcal O\!\left(
\frac{1}{(1-\gamma)\sqrt{B}}
\sqrt{\log\!\frac{1}{\delta}}
\right)
\]
with probability at least $1-\delta$.
\end{proposition}
Since all three formulations converge at the same rate in $B$, the ability to incorporate more basis functions within a fixed computational time budget becomes a key differentiator. As with sample paths, if we construct an algorithm that exploits WTCA's weaker temporal coupling, it can process more basis functions than PO within the same budget, making it plausible to reverse the upper bound ordering established in Proposition~\ref{prop:WTCA-comparison} for fixed basis functions and exact expectations.

In summary, model selection must account for both bound tightness under exact solution and solution complexity, which determines how many sample paths and basis functions can be leveraged within a fixed computational time budget. Since ALP is more strongly coupled than WTCA and yields a weaker upper bound, it is dominated by WTCA on both dimensions. The key question is, therefore, how WTCA compares to PO in practice, which motivates the development of an algorithm that exploits WTCA's weak temporal coupling, without which we cannot realize its computational advantages or assess how they translate into performance relative to PO under a fixed time budget.\looseness=-1

\section{Exploiting Weak Time Coupling via Parallel Computation}\label{sec:stoch-bcd}
This section develops a parallel stochastic block-coordinate descent (PS-BCD) algorithm that exploits weak temporal coupling and uses it to formalize the computational benefits of WTCA. In \S\ref{sec:seqBCD}, we describe a proximal stochastic gradient descent algorithm on a smoothed version of \eqref{eq:general-stochastic-form} to set up the major drivers of computational complexity in first-order methods. In \S\ref{sec:sbcd-algorithm}, we present PS-BCD whose computational complexity for solving WTCA is independent of horizon length, an independence that does not extend to PO owing to its strong temporal coupling. We discuss the resulting model selection implications in \S\ref{sec:model-selection}.

\subsection{Smoothing and Stochastic Gradient Descent}
\label{sec:seqBCD}
We solve WTCA and PO, both cast in the form \eqref{eq:general-stochastic-form}, using first-order methods that update the basis function weights $\beta$ with gradient information. The current state-of-the-art for solving PO is \citet{yang2025improved}, a first-order method based on sample average approximation; in contrast, our approach uses stochastic approximation, which avoids storing a fixed sample and scales more favorably with the horizon. At a given iterate $\beta$, the gradient indicates a direction along which the objective locally decreases. The objective function terms \eqref{eq:f_WTCA,t} and \eqref{eq:po-penalized} corresponding to WTCA and PO, respectively, are nonsmooth because they involve maximizations over collections of affine functions. We begin by discussing a standard smoothing technique to obtain a differentiable objective and then explain the gradient update.

To describe the smoothing transformation in a manner applicable to both WTCA and PO, consider a generic component function
\looseness=-1
\begin{equation}\label{eq:max-of-affine-form}
f_j(\beta,\sx)
=
c_j\max_{u \in \U_j}
\left\{
a_{j,\sx}(u)^\top \beta
+
b_{j,\sx}(u)
\right\},
\end{equation}
where $c_j >0$ is a component scaling constant and $\U_j$ is a finite set indexing the affine terms in the maximization. For fixed $j$ and $\sx$, the slope vector and intercept are defined through the mappings $u \mapsto a_{j,\sx}(u)$ and $u \mapsto b_{j,\sx}(u)$ from $\U_j$ to $\mathbb{R}^{BT}$ and $\mathbb{R}$, respectively. Thus, each element $u\in\U_j$ specifies one affine function of $\beta$ of the form  $a_{j,\sx}(u)^\top\beta + b_{j,\sx}(u)$.\looseness=-1

We replace each max operator with a smooth approximation using the log-sum-exp (LSE) smoothing technique \citep{smoothingfirstordermethods}. The quality of the approximation is controlled by a smoothing parameter $\sigma > 0$. Specifically, the log-sum-exp approximation of $f_j$ is 
\begin{equation}\label{eq:general-smooth-component}
f_{j,\sigma}(\beta,\sx)
=
\sigma\; c_j
\log
\sum_{u\in\U_j}
\exp\!\left(
\frac{a_{j,\sx}(u)^\top\beta + b_{j,\sx}(u)}{\sigma}
\right).
\end{equation}
It is known (see Example 4.5 in \cite{smoothingfirstordermethods}) that $f_{j,\sigma}(\cdot,\sx)$ is convex and continuously differentiable for every fixed $\sx$, and the approximation error is bounded above by a term linear in $\sigma$, i.e.,\looseness=-1
\[
0
\le
f_{j,\sigma}(\beta,\sx)
-
f_j(\beta,\sx)
\le
\sigma\; c_j \log |\U_j|,
\qquad \forall \beta.
\]
Applying this operation componentwise gives the smoothed analogue of \eqref{eq:general-stochastic-form} as
\begin{equation}\label{eq:general-formulation-smooth}
\min_{\beta\in\Omega}\;F_\sigma(\beta)
:=
\E_\xi\!\left[
\sum_{j=1}^m
f_{j,\sigma}(\beta,\sx)
\right].
\end{equation}

We next specify the corresponding affine coefficients when using WTCA and PO.

{\bf WTCA:} From \eqref{eq:WTCA-formulation}, the component index $j$ corresponds to the stage index $t$, so the summation in \eqref{eq:general-formulation-smooth} becomes a sum over stages $t=0,\ldots,T-1$ with $m=T$. 
For each stage $t$ and realization $\sx_t$, the Bellman deviation $\Delta_t(\beta;\sx_t)$ in \eqref{eq:Delta-Bellmanerror} is a maximum over a finite index set $\U_t$, where $\U_0 = \action_0$ and $\U_t := \Endo_t \times \action_t$ for $t\in\tset\setminus\{0\}$; when $t \geq 1$, an element $u_t = (\se_t,a_t) \in \U_t$ pairs an endogenous state with an action. Because $\hat V_{t,\beta}(s)=\phi_t(s)^\top\beta_t$ is linear in $\beta_t$, every term inside this maximization is affine in the full coefficient vector $\beta$ and can be written as $a_{t,\sx_t}(u_t)^\top\beta + b_{t,\sx_t}(u_t)$, where, using $s_t=(\se_t,\sx_t)$ and $\se_{t+1}=h(\se_t,a_t)$, the coefficients are
\begin{equation}\label{eq:WTCA-row0}
a_{0,\sx_0}(a_0)
=
(0,\ \gamma\,\E[\phi_1(s_1)\mid s_0,a_0],\ 0,\ldots,0),
\qquad
\text{ and } \qquad b_{0,\sx_0}(a_0)=r_0(s_0,a_0),
\end{equation}
for the initial stage, and for $t\ge 1$,
\begin{equation}\label{eq:WTCA-rows}
a_{t,\sx_t}(\se_t,a_t)
=
(0,\ldots,0,
\underbrace{-\phi_t(\se_t,\sx_t)}_{t\text{-th block}},\;
\underbrace{\gamma\,\E[\phi_{t+1}(h(\se_t,a_t),\sx_{t+1})\mid\sx_t]}_{(t+1)\text{-th block}},
0,\ldots,0), 
\end{equation}
and
\[
b_{t,\sx_t}(\se_t,a_t)
=
r_t((\se_t,\sx_t),a_t).
\]
Here, $a_{t,\sx_t}(\se_t,a_t)\in\R^{BT}$ is a vector whose entries are indexed by stage.
At $t=0$, the initial endogenous state $\se_0$ is fixed and the terms $\hat V_{0,\beta}(s_0)$ and $-\hat V_{0,\beta}(s_0)$ inside $\Delta_0$ cancel, which is why $a_{0,\sx_0}(a_0)$ has no entry in the $\beta_0$ block. More generally, $a_{t,\sx_t}$ is sparse and contains nonzero entries in at most two adjacent stage blocks corresponding to $\beta_t$ and $\beta_{t+1}$. Hence, setting the component scaling factors to $c_t = \gamma^t$ and defining $f_{t,\sigma}(\beta,\sx)$ via \eqref{eq:general-smooth-component} gives $F_\sigma(\beta) \equiv F_{\mathrm{WTCA},\sigma}(\beta)$.\looseness=-1

\smallskip
{\bf PO:}
The PO trajectory-level objective in \eqref{eq:po-penalized} contains a single component ($m=1$) defined as a maximum over the finite set $\U:=\action$ of full-horizon action sequences $u=(a_0,\ldots,a_{T-1})$.
By the same linearity argument as for WTCA, the term inside this maximum corresponding to a given $u$ and realization $\sx = (\sx_0,\ldots,\sx_{T-1})$ is affine in $\beta$ and can be written as $a_{\sx}(u)^\top\beta + b_{\sx}(u)$, where, using $s_t=(\se_t,\sx_t)$ and $\se_{t+1}=h(\se_t,a_t)$,
\begin{equation}\label{eq:PO-rows}
b_{\sx}(u)
=
\sum_{t=0}^{T-1}\gamma^t r_t((\se_t,\sx_t),a_t),
\qquad \text{ and }\qquad
a_{\sx}(u)
=
(0,
c_1(\sx,a_0),
\ldots,
c_T(\sx,a_{T-1})),
\end{equation}
with stagewise coefficients
\[
c_{t+1}(\sx,a_t)
=
\gamma^{t+1}
\Big(
-\phi_{t+1}(h(\se_t,a_t),\sx_{t+1})
+
\E[\phi_{t+1}(h(\se_t,a_t),\sx_{t+1})\mid\sx_t]
\Big)\in\R^B.
\]
Unlike WTCA, where each affine term involves only two adjacent stage blocks, the coefficient vector $a_{\sx}(u)$ in PO contains nonzero entries across nearly all stage blocks. Substituting these coefficients into \eqref{eq:general-smooth-component} with $c_1=1$ yields the smoothed PO objective $F_\sigma(\beta)\equiv F_{\mathrm{PO},\sigma}(\beta)$. \looseness=-1

Having specified the smoothed objective \eqref{eq:general-formulation-smooth} for both WTCA and PO, we now describe how to solve it using gradient-based updates. At iteration $k$ and iterate $\beta^k$, constructing the next iterate by minimizing a local quadratic model of the objective around the current iterate is a common strategy: \looseness=-1
\begin{equation}
\label{eq:proximal-update-global}
\beta^{k+1}
=
\arg\min_{z\in\Omega}
\left\{
\langle \nabla_\beta F_\sigma(\beta^k), z-\beta^k\rangle
+
\frac{L}{2\alpha_k}\|z-\beta^k\|_2^2
\right\},
\end{equation}
where the quadratic term $\|z-\beta^k\|_2^2$ measures the deviation of the new iterate from $\beta^k$, $\alpha_k>0$ is a scaling parameter that controls the effect of this quadratic term, and $L$ is the Lipschitz constant of the gradient of $F_\sigma$. Intuitively, $L$ captures the curvature of $F_\sigma$ (i.e., how quickly the gradient changes as $\beta$ varies). Formally, the constant 
$L$ satisfies
\begin{equation}
\label{eq:global-smoothness}
\|\nabla_\beta F_\sigma(\beta+d)-\nabla_\beta F_\sigma(\beta)\|_2
\le
L\|d\|_2.
\end{equation}
Because the curvature is bounded, the objective satisfies the quadratic bound\looseness=-1
\begin{equation}
\label{eq:quadratic-upper-bound}
F_\sigma(\beta+d)
\le
F_\sigma(\beta)
+
\langle\nabla_\beta F_\sigma(\beta),d\rangle
+
\frac{L}{2}\|d\|_2^2,
\end{equation}
which is leveraged in the update \eqref{eq:proximal-update-global}. 

The solution of \eqref{eq:proximal-update-global} simplifies to the update $\beta^{k+1}
=
\beta^k
-
\frac{\alpha_k}{L}\nabla_\beta F_\sigma(\beta^k),
$
if $\beta^{k+1} \in \Omega$. Otherwise, an additional projection onto $\Omega$ is needed, i.e., $\beta^{k+1}
=\mathrm{proj}_\Omega(\beta^k - \frac{\alpha_k}{L}\nabla_\beta F_\sigma(\beta^k))$. In either case, the update requires the gradient
$\nabla_\beta F_\sigma(\beta)
=
\E_\xi\!\left[\sum_{j=1}^m \nabla_\beta f_{j,\sigma}(\beta,\sx)\right],$
where differentiation and expectation are interchanged because the log-sum-exp components are smooth with bounded gradients on $\Omega$. We approximate the expectation using a stochastic gradient $G(\beta,\sx)$. Denoting by $\bar\sx$ an independent sample from $\xi$, we define
\[\nabla_\beta F_\sigma(\beta) \approx G(\beta,\bar\sx):= \sum_{j = 1}^m \nabla_\beta f_{j,\sigma}(\beta,\bar\sx).\]
The stochastic gradient update without projection becomes $\beta^{k+1}
=
\beta^k
-
\frac{\alpha_k}{L}G(\beta^k,\bar\sx)$. The curvature in this update plays an important role because the step size $\frac{\alpha_k}{L}$ is inversely proportional to $L$, and larger curvature constants result in smaller step sizes. This in turn worsens the iteration complexity, defined as the number of iterations needed to find a solution whose objective function value is within $\epsilon$ of $\min_{\beta\in\Omega}\;F_\sigma(\beta)$. The computational complexity is the product of bounds on the iteration complexity and the per-iteration cost of computing the stochastic gradient.

\subsection{A Parallel Stochastic Block Coordinate Descent}
\label{sec:sbcd-algorithm}

BCD partitions the basis function weights into blocks. We partition by stage, $\beta=(\beta_0,\ldots,\beta_{T-1})$ with $\beta_t\in\mathbb{R}^B$, aligning each block with the unit of temporal coupling. At iteration $k$, block-coordinate updates modify one
stage block $\beta^k_t$ to $\beta^{k+1}_t$ while keeping the other blocks fixed. As in gradient descent, $\beta^{k+1}_t$ is determined by a proximal update 
\begin{equation}\label{eq:seqProxUpdate}
\beta_t^{k+1}
=
\arg\min_{z\in\Omega_t}
\left\{
\langle\nabla_{\beta_t}F_\sigma(\beta^k),z-\beta_t^k\rangle
+
\frac{L_t}{2\alpha_k}\|z-\beta_t^k\|_2^2
\right\},
\end{equation}
where $\Omega_t$ denotes the projection of $\Omega$ onto block $t$, and $L_t$ is the Lipschitz constant of the partial gradient of $F_\sigma$ with respect to block $t$. Analogous to the role of $L$ in gradient descent, $L_t$ captures the curvature of $F_\sigma$ as only $\beta_t$ is changed, and similarly governs the step size and iteration complexity.

Rather than updating a single block per iteration, the parallel BCD we propose updates all $T$ blocks simultaneously. As in \eqref{eq:seqProxUpdate}, each block update uses the partial gradient with respect to that block, but because all blocks change at once, the block-wise Lipschitz constant $L_t$ does not account for the interaction among blocks updated simultaneously. To capture this interaction, we leverage Expected Separable Overapproximation (ESO) theory for parallel block-coordinate methods
\citep{richtarik2014iteration,AccCD}. Let $d=(d_0,\ldots,d_{T-1})$ denote increments applied
to all stage blocks. The ESO condition guarantees
\begin{equation}\label{eq:ESOequation}
F_\sigma(\beta+d)
\le
F_\sigma(\beta)
+
\langle \nabla_\beta F_\sigma(\beta), d\rangle
+
\frac12\sum_{t=0}^{T-1}\nu_t\|d_t\|_2^2 ,
\end{equation}
for appropriately chosen curvature weights $\nu_t$, $t\in\tset$. This inequality shows that even though the objective may couple multiple stage blocks, the increase $F_\sigma(\beta+d) - F_\sigma(\beta)$ from a perturbation $d$ can be bounded by a separable quadratic model. As a result, we can perform a parallel block update with $L_t$ in \eqref{eq:seqProxUpdate} replaced by $\nu_t$:
\begin{equation}\label{eq:parallel-block-update}
\beta_t^{k+1}
=
\arg\min_{z\in\Omega_t}
\left\{
\langle\nabla_{\beta_t}F_\sigma(\beta^k),z-\beta_t^k\rangle
+
\frac{\nu_t}{2\alpha_k}\|z-\beta_t^k\|_2^2
\right\}.
\end{equation}
This optimization simplifies to
$\beta_t^{k+1}
=
\beta_t^k
-
\frac{\alpha_k}{\nu_t}
\nabla_{\beta_t}F_\sigma(\beta^k)$ without projection. Let $L_{jt}$ denote the Lipschitz constant of the block-$t$ partial gradient of $\E_\xi[f_{j,\sigma}(\beta,\sx)]$. Since component $j$ involves stage blocks with indices in $C_j$, the curvature weight $\nu_t$ is
\begin{equation}\label{eq:nu-def}
\nu_t
=
\sum_{j:\,t\in C_j}
\kappa_j\,L_{jt}.
\end{equation}
This expression makes the role of temporal coupling explicit. Each component contributes curvature to every block it couples, and that contribution is amplified by $\kappa_j$, the number of blocks jointly involved in that component. When $\kappa_j$ is bounded independently of the horizon length $T$ (weak temporal coupling), the curvature weights $\nu_t$ remain bounded as $T$ grows. When $\kappa_j=T$ (full temporal coupling), each block inherits curvature contributions that scale linearly with $T$. Larger curvature weights force smaller steps in first-order updates, thereby increasing the number of iterations required for convergence. Thus, higher temporal coupling translates to smaller step sizes and worsens iteration complexity.

\begin{algorithm}[ht!]
\caption{Parallel Stochastic Block Coordinate Descent (PS-BCD)}
\label{alg:SBCD}
\begin{algorithmic}[1]
\Require Initial point $\beta^1 \in \Omega$, iteration limit $K$,
stepsize schedule $\alpha_k = 1/(2\sqrt{k+1})$, and curvature weights $\{\nu_t\}_{t\in\tset}$
\For{$k=1,2,\ldots,K$}
    \State Independently draw an exogenous trajectory sample $\sx^k$ from  $\mu$
    \For{each stage $t \in \tset$ \textbf{in parallel}}
        \State Compute stochastic gradient $G_t(\beta^k,\sx^k)$
        \State Update $\beta^{k+1}_t =\mathrm{proj}_{\Omega}\left(\beta_t^k - \dfrac{\alpha_k}{\nu_t}\, G_t(\beta^k,\sx^k)\right)$
    \EndFor
\EndFor
\State \textbf{Output:}
$\displaystyle
\bar\beta^K =
\frac{\sum_{k=2}^{K} \alpha_{k-1}\,\beta^k}
{\sum_{k=1}^{K-1} \alpha_k}
$
\end{algorithmic}
\end{algorithm}

 Algorithm \ref{alg:SBCD} summarizes PS-BCD. The inputs are an initial weight vector $\beta^1$, a decreasing step-size schedule $\alpha_k$ (a standard requirement in first-order methods), and the curvature weights $\nu_t$. These weights are available in closed form under our smoothing approach; as the expressions require additional notation, we relegate them to \S\ref{sec:EC-log-sum-exp-smoothing-curvature}.

At each iteration, the algorithm samples an exogenous trajectory $\sx^k$ from $\mu$ and then, in parallel for each $t \in \tset$, computes the block-$t$ stochastic gradient
\[
G_t(\beta^k,\sx^k)
:=
\sum_{j:t \in C_j} \nabla_{\beta_t} f_{j,\sigma}(\beta^k,\sx^k),
\]
and updates the block-$t$ weights. Given the log-sum-exp form of $F_\sigma(\beta)$, we can compute $G_t(\beta^k,\sx^k)$ in closed form as
\begin{equation}\label{eq:blocktstochGrad}
G_t(\beta^k,\sx^k)
=
\sum_{j:t \in C_j}
\sum_{u\in\U_j}
p_{j,u}(\beta^k,\sx^k)\,
a_{j,\sx^k}^{(t)}(u),    
\end{equation}
where $a_{j,\sx}^{(t)}(u)$ denotes the block-$t$ portion of the affine
vector $a_{j,\sx}(u)$, and
\[
p_{j,u}(\beta^k,\sx^k)
=
\frac{
\exp\!\left((a_{j,\sx^k}(u)^\top\beta^k+b_{j,\sx^k}(u))/\sigma\right)
}{
\sum_{v\in\U_j}
\exp\!\left((a_{j,\sx^k}(v)^\top\beta^k+b_{j,\sx^k}(v))/\sigma\right)
}.
\]
After $K$ iterations, the algorithm outputs a weighted average of the iterates $\bar\beta^K$, a standard scheme in stochastic approximation that ensures convergence at the rate stated below.

Theorem \ref{thm:psbcd-convergencerate} presents the convergence rate of Algorithm \ref{alg:SBCD} for solving \eqref{eq:general-formulation-smooth}. The bound is expressed in terms of a weighted norm that encodes the curvature weights $\nu_t$, making the influence of temporal coupling on convergence explicit. Specifically, define
\[
\|\beta\|_\nu^2
:=
\sum_{t=0}^{T-1}\nu_t\|\beta_t\|_2^2,
\qquad
\|\beta\|_{\nu,*}^2
:=
\sum_{t=0}^{T-1}\nu_t^{-1}\|\beta_t\|_2^2,
\]
and let $\beta^*$ denote an optimal solution of $\min_{\beta\in\Omega}F_\sigma(\beta)$. We use the common notation $\tilde{\mathcal O}(\cdot)$ to suppress logarithmic factors.
\begin{theorem}\label{thm:psbcd-convergencerate}
Suppose the stochastic gradient used in Algorithm~\ref{alg:SBCD} is an unbiased estimator of $\nabla_\beta F_\sigma(\beta)$ and satisfies the bounded-variance condition
\[
\E\!\left[
\left\|
G(\beta^k,\sx^k)
-
\nabla_\beta F_\sigma(\beta^k)
\right\|_{\nu,*}^2
\;\middle|\;
\beta^k
\right]
\le
\chi^2
\qquad
\text{for all }k .
\]
Then the averaged iterate $\bar\beta^K$ produced by Algorithm~\ref{alg:SBCD} satisfies
\begin{equation}\label{eq:complexitybound}
F_{\sigma}(\bar\beta^K)-F_{\sigma}(\beta^*)
=
\tilde{\mathcal O}\!\left(
\frac{1}{\sqrt K}
\left(
\sum_{t=0}^{T-1} \nu_t\|\beta_t^1-\beta_t^*\|_2^2
+
\chi^2
\right)
\right).
\end{equation}

\end{theorem}
The bounded-variance assumption is standard in stochastic approximation and bounds the conditional second moment of the stochastic gradient noise, measured in the weighted dual norm $\|\cdot\|_{\nu,*}$, uniformly by $\chi^2$. Because the algorithm outputs an average of the iterates, the effect of this noise diminishes at rate $1/\sqrt{K}$ and does not accumulate over iterations. Corollary \ref{cor:iterComplexPSBCD} translates the convergence rate of Theorem \ref{thm:psbcd-convergencerate} into iteration complexity and then specializes it to WTCA and PO.
\begin{corollary}\label{cor:iterComplexPSBCD}
    Suppose the assumptions of Theorem \ref{thm:psbcd-convergencerate} and $\beta^1 = 0$ hold. For any $\epsilon > 0$, PS-BCD achieves $F_\sigma(\bar\beta^K)-F_\sigma(\beta^*)\le \epsilon$ in $\tilde{\mathcal O}\Big(\frac{1}{\epsilon^2}\big(\sum_{t=0}^{T-1}\nu_t\big)^2\Big)$ iterations, which is bounded above by $\tilde{\mathcal O}\Big(1/\big(\sigma(1-\gamma)\epsilon\big)^2\Big)$ and $\tilde{\mathcal O}\Big(|\action|^2T^2/\big(\sigma(1-\gamma)\epsilon\big)^2\Big)$ for WTCA and PO, respectively, as shown in \S\ref{sec:EC-K},  where $\chi^2$ is a constant independent of $T$ for both formulations (Corollary \eqref{cor:dual-noise-specializations}).\looseness=-1
\end{corollary}
The main takeaway is that the iteration complexity of PS-BCD when solving WTCA is independent of $T$, while for PO it grows nonlinearly with $T$. The factor $|\action|=|\action_0|\times\ldots\times|\action_{T-1}|$ in the PO bound also embeds dependence on $T$, since it is a product over $T$ per-stage action sets and therefore grows with the horizon even when each per-stage action set is small. Its presence in the iteration complexity is a consequence of the smoothing technique we employ.\looseness=-1

Understanding the overall computational complexity of PS-BCD requires incorporating the per-iteration cost of computing the stochastic gradient \eqref{eq:blocktstochGrad}. Computing $p_{j,u}(\beta^k,\sx^k)$ requires $\mathcal O(|\U_j|\kappa_j B)$ operations. Therefore, computing $G_t(\beta^k,\sx^k)$ takes $\mathcal O(\sum_{j:t \in C_j}|\U_j|\kappa_j B)$. For WTCA, the per-iteration complexity is $\mathcal O(B)$ because $m = T$, $|\{j:t \in C_j\}| = 2$, $\kappa_j = 2$, and $\U_t=\Endo_t\times\action_t$ is finite and independent of the horizon length. In contrast, for PO, this computation requires $\mathcal O(|\action|\,TB)$ since $m= 1 = |\{j:t \in C_j\}|$, $\kappa_1 = T$, and $\U_1=\action$. The factor of $|\action|$ is intrinsic to PO's trajectory-level structure: evaluating the inner maximization in \eqref{eq:po-penalized} requires enumerating all full-horizon action sequences regardless of how that maximization is handled. Moreover, since $|\action|$ grows with $T$, this per-iteration cost grows nonlinearly with the horizon. Based on these observations, Proposition \ref{prop:complexity-comparison} states the computational complexity of using PS-BCD to solve WTCA and PO, and also compares the analogous complexities when using stochastic gradient descent (SGD) instead.

\begin{proposition}\label{prop:complexity-comparison}
Let $\epsilon>0$ denote the target optimality tolerance and suppose the assumptions underlying Theorem \ref{thm:psbcd-convergencerate} hold.
The computational complexities of finding $\epsilon$-optimal solutions to WTCA and PO using SGD and PS-BCD are as follows:

\begin{table}[H]
\centering
\begin{tabular}{l|cc}
\toprule
 & SGD & PS-BCD \\
\midrule
WTCA 
&
$\displaystyle
\tilde{\mathcal O}\!\left(
\frac{TB}{\epsilon^2\sigma^2(1-\gamma)^2}
\right)
$
&
$\displaystyle
\tilde{\mathcal O}\!\left(
\frac{B}{\epsilon^2\sigma^2(1-\gamma)^2}
\right)
$
\\[20pt]

PO
&
$\displaystyle
\tilde{\mathcal O}\!\left(
\frac{|\action|^3\,TB}{\epsilon^2\sigma^2(1-\gamma)^2}
\right)
$
&
$\displaystyle
\tilde{\mathcal O}\!\left(
\frac{|\action|^3\,T^3 B}{\epsilon^2\sigma^2(1-\gamma)^2}
\right)
$
\\
\bottomrule
\end{tabular}
\end{table}
\end{proposition}
PS-BCD exhibits horizon-length independent computational complexity when used to solve WTCA. Using SGD instead to solve WTCA introduces a linear dependence on $T$. It is tempting to think that this difference is purely due to parallelization in PS-BCD removing the per-iteration cost dependence on $T$. However, there is a more subtle effect on the iteration complexity. Because WTCA is weakly coupled, the information loss from using block-$t$ updates in parallel is not significant compared to the full gradient update used in SGD, and thus the iteration complexity also remains independent of $T$. The PO model exhibits a markedly different dependence. First, the computational complexity of applying PS-BCD to solve PO shows nonlinear growth with $T$ due to strong temporal coupling. Second, SGD achieves lower computational complexity than PS-BCD when solving PO. This somewhat striking behavior happens because PS-BCD parallelizes block updates of the PO objective function, which is fully coupled over time. Employing block updates leads to significant information loss compared to a full gradient update, and this causes the PS-BCD iteration complexity to depend more steeply on $T$ than that of SGD.

\subsection{Model Selection Under Computational Time Budgets}
\label{sec:model-selection}

In \S\ref{subsec:BoundOrderingAndModSelection}, we argued that model selection should account for computational time budgets because the number of exogenous samples and basis functions that can be processed within a fixed budget depends on solution complexity. That discussion was necessarily conditional on the existence of an algorithm that exploits weak temporal coupling. With Proposition~\ref{prop:complexity-comparison} in hand, we can now make these arguments concrete. We focus on WTCA and PO and compare the best algorithm for each: PS-BCD for WTCA and SGD for PO, whose computational complexities are
\[
\mathrm{WTCA~(PS\text{-}BCD)}:
\quad
\tilde{\mathcal O}\!\left(
\frac{B}{\epsilon^2\sigma^2(1-\gamma)^2}
\right),
\qquad
\mathrm{PO~(SGD)}:
\quad
\tilde{\mathcal O}\!\left(
\frac{|\action|^3\,T\,B}{\epsilon^2\sigma^2(1-\gamma)^2}
\right).
\]

\paragraph{Exogenous samples.}
PS-BCD employs stochastic approximation, where each iteration processes one realization of the exogenous trajectory $\sx$. Executing $K$ iterations therefore incorporates $K$ independent sample paths. The per-sample cost for WTCA is independent of $T$, whereas for PO it grows with $T$ both directly through the explicit factor of $T$ and indirectly through $|\action| = |\action_0| \times \cdots \times |\action_{T-1}|$, which itself grows with the horizon. Therefore, under a fixed computational time budget, the number of sample paths that WTCA can process relative to PO increases with $T$.

\paragraph{Basis functions.}
As established in Proposition~\ref{prop:asymptotic-consistency-random}, increasing the number of random basis functions $B$ improves approximation accuracy at the rate $1/\sqrt{B}$, and WTCA and PO share this rate. Let $\mathcal C$ denote the available time budget and fix a target optimization accuracy $\epsilon$. 
Equating each approximation's computational complexity with the budget $\mathcal C$
determines the largest admissible basis dimension under each formulation:
\[
B_{\mathrm{WTCA}}
\asymp
\mathcal C\,\epsilon^2\sigma^2(1-\gamma)^2,
\qquad
B_{\mathrm{PO}}
\asymp
\frac{\mathcal C\,\epsilon^2\sigma^2(1-\gamma)^2}
{|\action|^3\,T}.
\]
Thus, under identical computational resources, WTCA can support basis families larger by a factor proportional to $|\action|^3 T$.
Substituting these admissible basis dimensions into the $\mathcal O(1/\sqrt{B})$ approximation rate of Proposition~\ref{prop:asymptotic-consistency-random} yields
\[
\|U_{\mathrm{WTCA},B}-V^*\|_\infty
=
\tilde{\mathcal O}\!\left(\frac{1}{\sqrt{\mathcal C}}\right),
\qquad
\|U_{\mathrm{PO},B}-V^*\|_\infty
=
\tilde{\mathcal O}\!\left(
\frac{\sqrt{|\action|^3 T}}{\sqrt{\mathcal C}}
\right).
\]

In summary, although PO yields the tightest upper bound for any fixed set of basis functions and exact expectations, its full temporal coupling increases computational cost and restricts both the number of exogenous samples and the number of basis functions that can be processed. In contrast, WTCA's weak temporal coupling enables PS-BCD to achieve horizon-independent complexity, allowing more sample paths and larger basis function sets within the same budget. The resulting approximation rate for WTCA is $\tilde{\mathcal O}(1/\sqrt{\mathcal C})$, free of horizon dependence, whereas the rate for PO degrades as $\tilde{\mathcal O}(\sqrt{|\action|^3 T}/\sqrt{\mathcal C})$. This advantage grows with the horizon length. That is, although PO provides a tighter bound than WTCA for any fixed set of basis functions and exact expectations, under a fixed computational time budget WTCA can produce tighter bounds than PO, and this reversal widens as the horizon grows, as anticipated in \S\ref{subsec:BoundOrderingAndModSelection}.\looseness=-1

\section{Numerical Experiments}
\label{sec:numerical}
We evaluate the budget--accuracy tradeoff predicted by our theoretical analysis in two settings, Bermudan option pricing~\citep{desai2012pathwise} and merchant ethanol production~\citep{guthrie2009real,yang2024least,yang2025improved}. Both settings admit MDP formulations with a high-dimensional exogenous state and involve irreversible endogenous transitions, such as option exercise and facility abandonment, but differ in two respects that affect computational tradeoffs. 
The Bermudan option has a binary action space (stop or continue) and a fixed-dimensional exogenous state, whereas ethanol production has up to four feasible actions per state and an exogenous state whose dimension grows with the horizon. 
In these problems, myopic rolling reoptimization based on deterministic forecasts performs poorly because it neglects the option value of preserving future flexibility.
Indeed, in our Bermudan option instances, the well known intrinsic policy stops whenever the immediate payoff is positive, and generates lower bounds that are negligible relative to the known upper bounds, with optimality gaps exceeding 90\%. In our ethanol instances, the (rolling) intrinsic policy performs even worse. It abandons the facility at the initial stage on every Monte Carlo sample path and generates zero policy value, because the initial forward curve has negative intrinsic value. We present Bermudan option pricing results in \S\ref{scc:Insts-BermudanOptions}. Due to space limitations, the ethanol production treatment in \S\ref{scc:ethanol-main} is more condensed, with details provided in \S\ref{sec:EC-numerical}.

We compare WTCA, PO, and LSM implemented in \texttt{C++}, with least-squares regressions solved via Gurobi~12.0.  Experiments were conducted on a workstation with a 24-core Intel(R) Core(TM) i9-14900K CPU (3.20~GHz) and 128~GB of RAM.

\subsection{Bermudan Option Pricing}
\label{scc:Insts-BermudanOptions}
We consider the valuation of a multi-asset Bermudan call option with a knock-out barrier, formulated as a finite-horizon optimal stopping problem following \citet{desai2012pathwise}. The option is written on $N$ underlying assets with exercise dates indexed by $\tset:=\{0,1,\ldots,T-1\}$. At time~$t\in\tset$, the exogenous state is the vector of asset prices $\sx_t=(\sx_{t,1},\ldots,\sx_{t,N})\in\mathbb{R}_+^N$. Under the risk-neutral measure, each asset price follows a geometric Brownian motion. Specifically, for each $q=1,\ldots,N$, the price process $\{W_{q}(t'), t'\in\mathbb{R}_{+}\}$ satisfies
\[
\frac{dW_{q}(t')}{W_{q}(t')} = r_f dt' + \eta_q dZ_{q}(t'),
\]
where $r_f$ is the risk-free rate, $\eta_q$ is the volatility of asset~$q$, and $\{Z_{q}(t')\}_{q=1}^N$ are independent standard Brownian motions \citep{desai2012pathwise}. Then $\{\sx_{t,q},t\in\tset\}$ is obtained by sampling $W_q(t')$ at each possible exercise date $t$.
The per-period discount factor is $\gamma=e^{-r_f\Delta t}$, with $\Delta t$ denoting the fixed time step between consecutive exercise dates.

The endogenous state $\se_t\in\{0,1\}$ indicates whether the option is active ($\se_t=0$) or inactive ($\se_t=1$) due to exercise or knock-out. The feasible action set depends only on $\se_t$:
\[
\mathcal{A}_t(\se_t)=
\begin{cases}
\{\text{Stop},\,\text{Continue}\}, & \text{if } \se_t=0,\\[2pt]
\varnothing, & \text{if } \se_t=1.
\end{cases}
\]
The transition of $\se_t$ is driven by the chosen action and the barrier condition. Let $w^B$ denote the barrier price. If the maximum of the underlying prices ever hits or exceeds $w^B$, or if the holder exercises, the option becomes inactive thereafter:
\[
\se_t=
\begin{cases}
\mathbbm{1}_{\{\max_q\sx_{t,q}\ge w^B\}}, & t=0,\\[3pt]
\max\Big\{\se_{t-1},\,\mathbbm{1}_{\{\max_q\sx_{t,q}\ge w^B\}},\,\mathbbm{1}_{\{a_{t-1}=\mathrm{Stop}\}}\Big\}, & t\in\tset\setminus\{0\},
\end{cases}
\]
where $\mathbbm{1}$ is the indicator function. Let $w^S$ be the strike price. The immediate payoff is
\begin{equation}\label{eq:payoff_Bermudan}
r_t((\se_t,\sx_t),a_t)=
\begin{cases}
(\max_q\sx_{t,q}-w^S)^+, & \text{if } \se_t=0 \text{ and } a_t=\mathrm{Stop},\\[3pt]
0, & \text{otherwise},
\end{cases}
\end{equation}
where $(\cdot)^+=\max\{\,\cdot,0\}$. Thus, the option yields a positive payoff only when it is active and exercised; otherwise, the payoff is zero.

\subsubsection{Implementation Details}
\label{scc:Details-BermudanOptions}
We evaluate three methods on two sets of instances that differ in the time horizon. 
The first set uses $T=36$ exercise dates, which corresponds to a 36-month horizon, and the second set uses $T=100$ exercise dates, which corresponds to a 100-month horizon. Following \citet{desai2012pathwise}, we vary the initial price $w^{I}\in\{90,100,110\}$ and the number of underlying assets $N\in\{4,8,16\}$, which yields nine instances for each horizon. With fixed $T$, $N$, and $w^{I}$, the initial asset prices are identical across assets, i.e., $\sx_{0,q}=w^I$ for $q=1,\ldots, N$. The remaining parameters follow \citet{desai2012pathwise} and are reported in Table~\ref{tab:parameter_settings}.
\begin{table}[ht!]
\centering
\caption{Parameter settings for Bermudan option instances}
\label{tab:parameter_settings}
\renewcommand{\arraystretch}{1.15}
\begin{tabular}{lc}
\hline
Parameter & Value \\
\hline
Strike price $w^S$ & \$100 \\
Barrier price $w^B$ & \$170 \\
Time step $\Delta t$ & $1/12$ (monthly) \\
Volatility $\eta_q$ (annualized) & 0.20 \\
Risk-free rate $r_f$ (annualized) & 0.05 \\
\hline
\end{tabular}
\end{table}
\paragraph{Basis functions.}
Our basis function set consists of a constant term, the payoff \eqref{eq:payoff_Bermudan}, and $B-2$ random Fourier basis functions.
Including the payoff as a basis term is standard in option pricing ~\citep{kohler2010review,chen2021deep,desai2012pathwise}.
Random Fourier bases are also widely used in machine learning and more recently in option pricing because they provide flexible high-dimensional approximations of smooth value functions without manual feature design
~\citep{rahimi2007random,rahimi2008,pakiman2024,TothOberhauserSzabo2025}. 

Each random Fourier basis function has the form 
$\phi(\sx_t;\theta)=\cos\big(\theta_0+\sum_{q=1}^N\theta_q\sx_{t,q}\big),$ 
where the coefficient vector $\theta=(\theta_0,\theta_1,\ldots,\theta_N)$ is generated randomly. The intercept $\theta_0$ is drawn from a uniform distribution $\text{Unif}([-\pi,\pi])$, and $\theta_1, \ldots,\theta_N$ are sampled independently from a normal distribution $N(0,\varrho)$, with $\varrho$ controlling the frequency bandwidth of the basis functions.
The value function approximation is
\begin{equation}\label{eq:VFA_Bermudan}
\hat{V}_t(\se_t,\sx_t)
=\beta_{t,0}
+\beta_{t,1}(\max_q\sx_{t,q}-w^S)^+(1-\se_t)
+\sum_{b=2}^{B-1}\beta_{t,b}\phi_b(\sx_t;\theta^b),
\end{equation}
where $\theta^b:=(\theta_0^b,\theta_1^b,\ldots,\theta_N^b)$ is the $b$-th set of random draws, and $(\beta_{t,0}, \beta_{t,1},\ldots, \beta_{t,B-1})$ are the weights associated with the basis functions. 
We set $B=17$, which corresponds to 15 random Fourier basis functions plus the constant and payoff terms. This is the smallest number of bases that yield stable numerical results. Increasing this number further does not significantly improve the results in our experiments.

\paragraph{Training and parameter tuning.}
Both WTCA and PO use a log-sum-exp smoothing of the inner maximization to obtain a differentiable objective. We tune the bandwidth $\varrho$ over $\{10^5,10^4,\ldots,10^{-5}\}$ and the smoothing parameter $\sigma$ over $\{1,10,20,30,\ldots,5000\}$. Larger values of $\sigma$ typically speed up optimization but introduce more bias relative to the original nonsmooth formulation, while smaller values reduce this bias at higher computational cost. For numerical stability and more reliable step-size selection, we scale the payoff basis term by a factor $\lambda\in\{1,10,20,\ldots,w^B\}$ so that its magnitude is comparable to that of the random Fourier features. We select all hyperparameters by cross-validation to balance speed and accuracy.\looseness =-1

We solve the smoothed WTCA and PO problems using Algorithm~\ref{alg:SBCD} and SGD, respectively, with identical basis functions and hyperparameter grids. After SGD converges, PO requires an additional regression step to recompute value function approximations that yield unbiased upper bounds \citep[\S 3.3]{desai2012pathwise}; no analogous post-processing step is needed for WTCA. We allocate equal CPU-time limits to WTCA and PO to ensure a fair comparison. These limits were two hours for $T = 36$ and four hours for $T = 100$. Within these budgets, WTCA has already achieved stable performance, whereas PO remains far away from convergence and exhibits a slow rate of improvement.

\paragraph{Benchmarks and evaluation.}
As a benchmark, we include the LSM method
~\citep{carriere1996valuation,longstaff2001valuing,tsitsiklis2001regression,glasserman2004monte}, which estimates continuation values through backward regressions.
Each regression uses 50,000 simulated trajectories and the same basis set as WTCA and PO.

We estimate lower bounds by simulating the greedy policy in \eqref{eq:greedy-policy} on 100,000 Monte Carlo paths. We estimate upper bounds using the information-relaxation and duality approach \citep{brown2010information} on the same sample paths. For each state-action pair, we compute the dual penalty in \eqref{eq:info-relaxpenalty} and estimate the required conditional expectations using one-step inner sampling with 500 draws. \looseness =-1

\subsubsection{Results}
\label{scc:Results-BermudalOptions}
Table~\ref{T:Results1} reports the performance of LSM, PO, and the proposed WTCA on nine Bermudan option instances with horizon $T=36$. Each instance is characterized by the number of assets $N$ and the initial asset price $w^I$.
The top panel reports upper and lower bound estimates, denoted UB and LB, with standard errors in parentheses.  
The bottom panel reports, for each method, the percentage deviation of its UB from the tightest observed bound, labeled ``Best UB'', and the relative optimality gap computed as $(\text{Best UB} - \text{LB})/(\text{Best UB})$. This gap measures the normalized distance between the best available upper bound and the method's policy value. Smaller values thus indicate better near-optimal performance.
\begingroup
\renewcommand{\arraystretch}{1.0}
\begin{table}[ht!]
\caption{Upper and lower bound estimates (with standard errors) and relative percentage differences for LSM, PO, and WTCA in instances with $\mathbf{T=36}$}
\label{T:Results1}
\centering
\begin{tabular}{cccccccc}
\hline\hline
    &         & \multicolumn{3}{c}{Upper Bound Estimate}         & \multicolumn{3}{c}{Lower Bound Estimate}         \\ \cline{3-8} 
$N$ & $w^I$ & LSM            & PO             & WTCA           & LSM            & PO             & WTCA           \\ \hline
4   & 90      & 34.09 (0.06)   & 39.94 (0.06)   & 33.94 (0.04)   & 32.54 (0.08)   & 31.26 (0.08)   & 33.13 (0.08)   \\
4   & 100     & 42.87 (0.07)   & 47.97 (0.05)   & 42.45 (0.04)   & 40.86 (0.06)   & 40.87 (0.08)   & 41.78 (0.07)   \\
4   & 110     & 49.39 (0.07)   & 53.37 (0.03)   & 48.71 (0.05)   & 46.34 (0.05)   & 47.63 (0.07)   & 47.89 (0.06)   \\
8   & 90      & 45.04 (0.08)   & 49.20 (0.04)   & 44.59 (0.04)   & 42.83 (0.05)   & 42.97 (0.07)   & 43.71 (0.06)   \\
8   & 100     & 51.36 (0.08)   & 54.47 (0.03)   & 50.50 (0.05)   & 47.91 (0.04)   & 49.50 (0.06)   & 49.59 (0.05)   \\
8   & 110     & 54.60 (0.08)   & 57.05 (0.02)   & 53.61 (0.04)   & 50.88 (0.03)   & 52.77 (0.05)   & 52.71 (0.05)   \\
16  & 90      & 51.74 (0.09)   & 54.87 (0.02)   & 50.91 (0.05)   & 48.37 (0.03)   & 50.03 (0.05)   & 50.09 (0.05)   \\
16  & 100     & 54.52 (0.08)   & 57.22 (0.02)   & 53.59 (0.04)   & 50.99 (0.03)   & 52.87 (0.05)   & 52.79 (0.05)   \\
16  & 110     & 55.58 (0.07)   & 58.06 (0.02)   & 54.83 (0.04)   & 52.88 (0.03)   & 54.17 (0.05)   & 54.06 (0.05)   \\ \hline
    &         & \multicolumn{3}{c}{\% (UB - Best UB)/(Best UB)} & \multicolumn{3}{c}{\% (Best UB - LB)/(Best UB)} \\ \cline{3-8} 
$N$ & $w^I$ & LSM            & PO             & WTCA           & LSM            & PO             & WTCA           \\ \hline
4   & 90      & 0.44           & 17.69          & 0.00           & 4.13           & 7.88           & 2.39           \\
4   & 100     & 0.98           & 13.00          & 0.00           & 3.75           & 3.73           & 1.59           \\
4   & 110     & 1.40           & 9.58           & 0.00           & 4.87           & 2.22           & 1.68           \\
8   & 90      & 1.02           & 10.34          & 0.00           & 3.95           & 3.63           & 1.98           \\
8   & 100     & 1.70           & 7.87           & 0.00           & 5.14           & 1.99           & 1.80           \\
8   & 110     & 1.84           & 6.40           & 0.00           & 5.11           & 1.56           & 1.68           \\
16  & 90      & 1.63           & 7.78           & 0.00           & 5.00           & 1.73           & 1.61           \\
16  & 100     & 1.75           & 6.78           & 0.00           & 4.84           & 1.33           & 1.50           \\
16  & 110     & 1.37           & 5.90           & 0.00           & 3.55           & 1.21           & 1.40           \\ \hline\hline
\end{tabular}
\end{table}
\endgroup

WTCA consistently generates the tightest upper bounds across all instances.
Relative to WTCA, LSM produces upper bounds that are 0.4\%--1.8\% higher, averaging a 1.3\% difference. PO performs markedly worse, generating upper bounds that are 5.9\% to 17.7\% looser than those of WTCA, with an average shortfall of 9.5\%.
While these results may seem at odds with Proposition~\ref{prop:WTCA-comparison}, which shows that PO yields tighter bounds for any fixed basis set, the discrepancy underscores the practical benefits of model selection discussed in \S\ref{sec:stoch-bcd}.
Under equal CPU-time budget, WTCA benefits from a horizon-independent iteration cost that is substantially lower than that of PO. This efficiency enables WTCA to process more exogenous samples, which ultimately outweighs PO's theoretical upper bounding advantage.\looseness = -1

Regarding lower bounds, WTCA and PO yield comparable policies in most instances. At $N=8$ and $N=16$, the optimality gaps of the two methods differ by at most 0.5\% in five of six instances, with both falling between 1.2\%--2.0\%. The main exception occurs at $(N,w^I)=(4,90)$, where PO's gap reaches 7.9\%, while WTCA's remains tight at 2.4\%. LSM yields the largest gaps in eight of the nine instances, with gaps ranging from 3.6\% to 5.1\%, which is consistent with the backward error accumulation discussed in \S\ref{sec:fh-mdp-optimization}.

Using the WTCA upper bounds as yardsticks, the lower bound results further indicate that the WTCA policy is near-optimal across all instances. PO achieves similarly near-optimal performance in six of the nine instances, with larger gaps only at $(N,w^I)=(4,90)$, $(N,w^I)=(4,100)$, and $(N,w^I)=(8,90)$.
These conclusions are not apparent when optimality gaps are computed using upper bounds from PO or LSM.

WTCA achieves more significant improvements in upper bounds than in lower bounds because upper bounds depend directly on the absolute accuracy of the value function approximation. In contrast, the induced greedy policy relies primarily on the relative ranking of actions at each state, making it less sensitive to approximation errors.
\begin{figure}[ht!]
    \centering
    \includegraphics[width=0.45\linewidth]{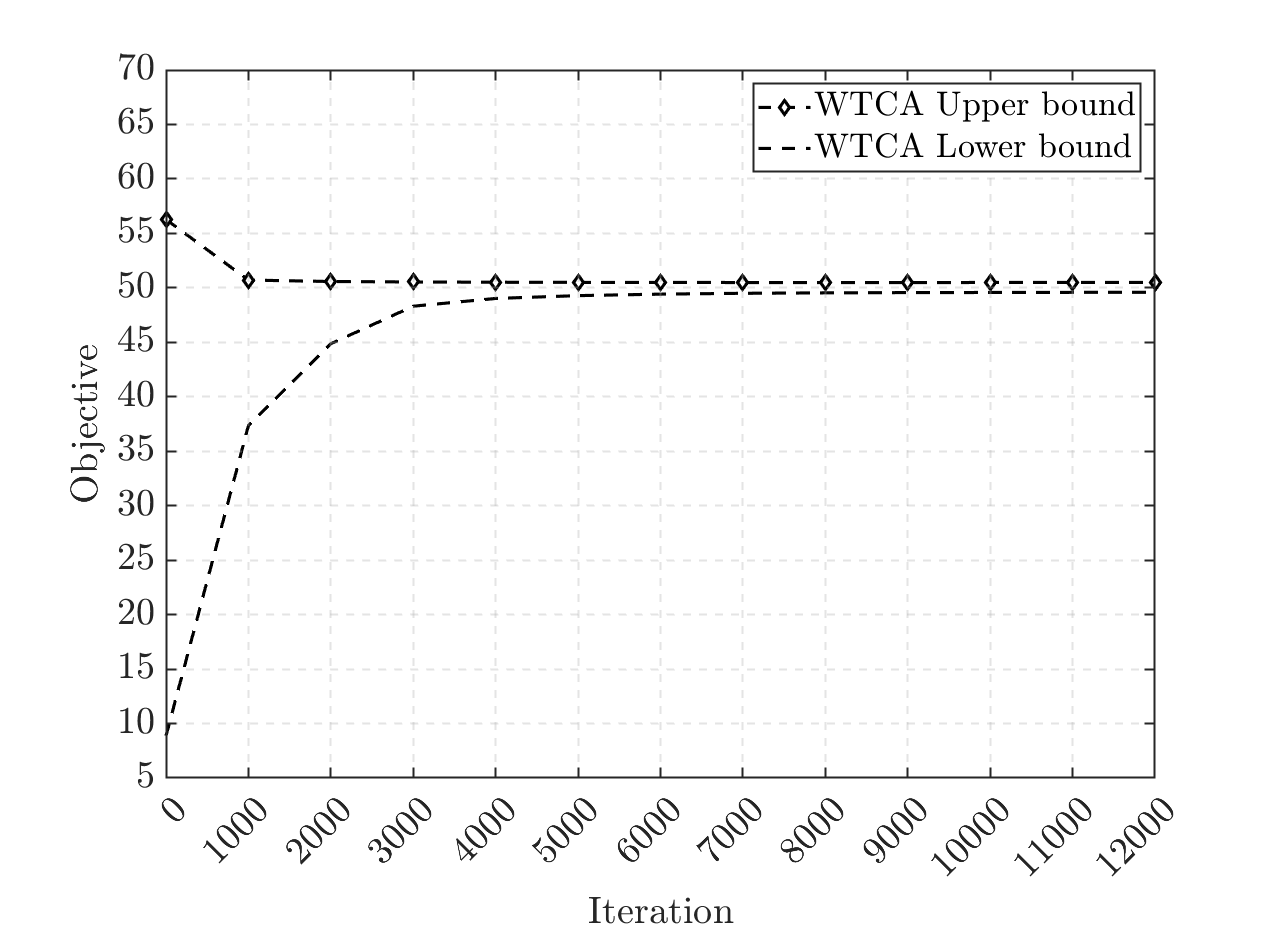}
    \includegraphics[width=0.45\linewidth]{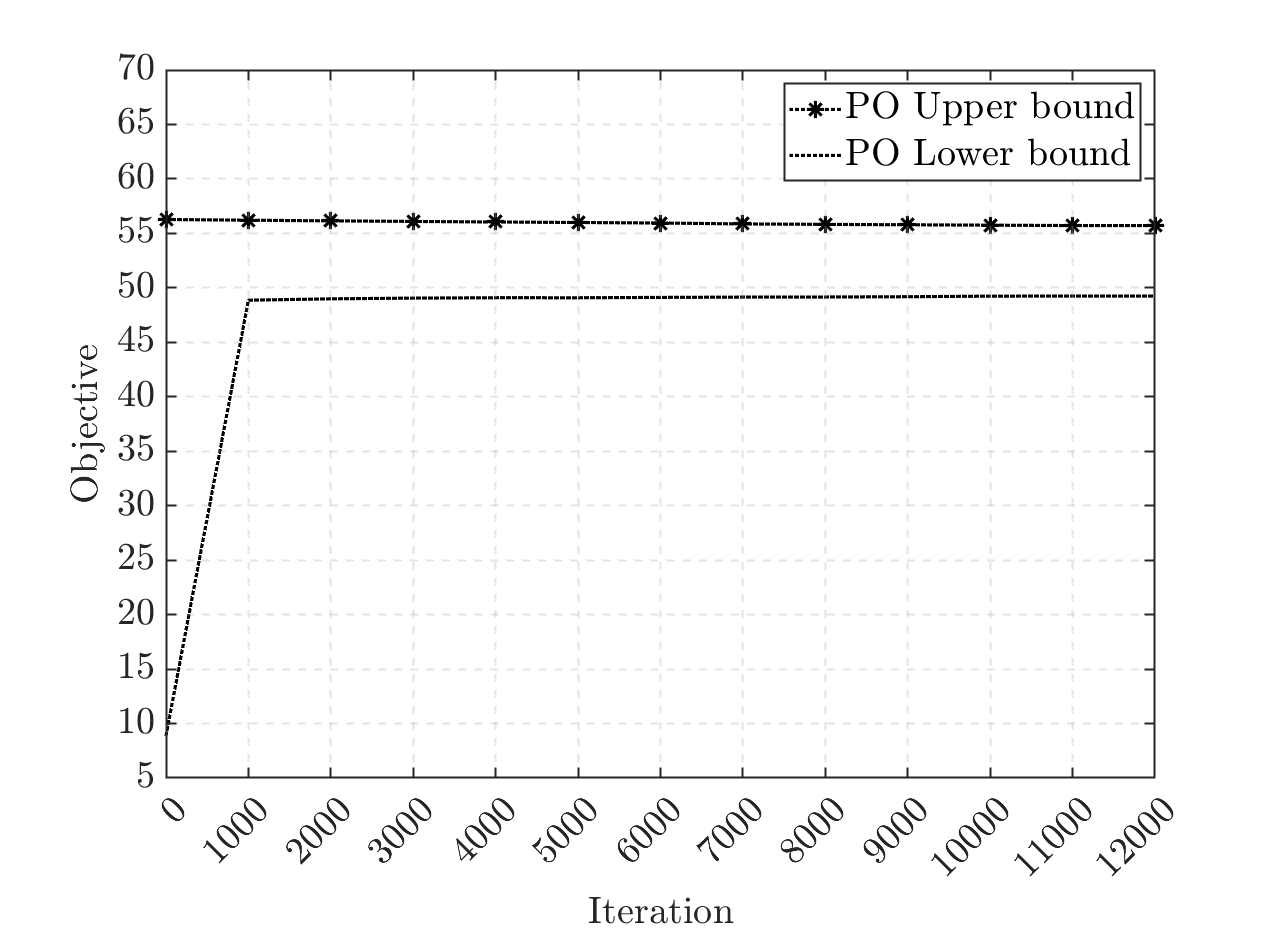}
    \caption{Convergence of upper and lower bounds for WTCA (left) and PO (right) in the instance with $\mathbf{T=36}$, $\mathbf{N=8}$, and $\mathbf{w^I=100}$.}
    \label{fig:Comparison}
\end{figure}

Figure~\ref{fig:Comparison} illustrates the convergence behavior predicted by Theorem~\ref{thm:psbcd-convergencerate}.
In a representative instance with $(T,N,w^I)=(36,8, 100)$, the WTCA upper bound falls within 0.5\% of its final value by 12{,}000 iterations. At the same iteration count, PO remains roughly 7\% above its final estimate.
This disparity is consistent with the theoretical rates of $\mathcal{O}(1/\sqrt{K})$ for WTCA and $\mathcal{O}(\sqrt{T}/\sqrt{K})$ for PO.
Both lower bounds increase monotonically. WTCA stabilizes after approximately 5,000 iterations, whereas PO rises sharply during the first 1,000 iterations before gradually approaching its final asymptotic value.
\begingroup 
\renewcommand{\arraystretch}{1.0}
\begin{table}[b!]
\caption{Upper and lower bound estimates (with standard errors) and relative percentage differences for LSM, PO, and WTCA in 100-stage $\mathbf{(T=100)}$ instances}
\label{T:Results2}
\centering
\begin{tabular}{cccccccc}
\hline\hline
    &         & \multicolumn{3}{c}{Upper Bound Estimate}         & \multicolumn{3}{c}{Lower Bound Estimate}          \\ \cline{3-8} 
$N$ & $w^I$   & LSM            & PO             & WTCA           & LSM             & PO             & WTCA           \\ \hline
4   & 90      & 47.24 (0.11)   & 51.30 (0.04)   & 46.13 (0.07)   & 41.27 (0.05)    & 41.94 (0.08)   & 44.80 (0.06)   \\
4   & 100     & 51.08 (0.11)   & 54.45 (0.03)   & 49.81 (0.07)   & 45.73 (0.04)    & 47.33 (0.07)   & 48.26 (0.05)   \\
4   & 110     & 53.92 (0.10)   & 56.87 (0.02)   & 52.57 (0.07)   & 48.35 (0.03)    & 50.74 (0.06)   & 50.85 (0.04)   \\
8   & 90      & 52.20 (0.11)   & 55.22 (0.02)   & 50.77 (0.07)   & 46.88 (0.03)    & 49.34 (0.06)   & 49.53 (0.04)   \\
8   & 100     & 54.21 (0.10)   & 57.28 (0.02)   & 52.83 (0.06)   & 49.04 (0.03)    & 51.47 (0.06)   & 51.35 (0.04)   \\
8   & 110     & 55.48 (0.09)   & 58.87 (0.02)   & 54.31 (0.05)   & 51.04 (0.03)    & 53.22 (0.05)   & 53.00 (0.04)   \\
16  & 90      & 53.75 (0.10)   & 57.18 (0.02)   & 52.50 (0.06)   & 49.09 (0.03)    & 51.44 (0.05)   & 51.28 (0.04)   \\
16  & 100     & 54.92 (0.09)   & 58.74 (0.02)   & 53.87 (0.05)   & 51.08 (0.03)    & 52.97 (0.05)   & 52.84 (0.05)   \\
16  & 110     & 55.60 (0.07)   & 59.81 (0.02)   & 54.40 (0.05)   & 52.87 (0.03)    & 54.23 (0.05)   & 53.39 (0.05)   \\ \hline
    &         & \multicolumn{3}{c}{\% (UB - Best UB)/ (Best UB)} & \multicolumn{3}{c}{\% (Best UB - LB) / (Best UB)} \\ \cline{3-8} 
$N$ & $w^I$ & LSM            & PO             & WTCA           & LSM             & PO             & WTCA           \\ \hline
4   & 90      & 2.40           & 11.20          & 0.00           & 10.54           & 9.09           & 2.87         \\
4   & 100     & 2.53           & 9.30           & 0.00           & 8.20            & 4.99           & 3.12         \\
4   & 110     & 2.56           & 8.17           & 0.00           & 8.02            & 3.49           & 3.28         \\
8   & 90      & 2.80           & 8.75           & 0.00           & 7.67            & 2.83           & 2.45         \\
8   & 100     & 2.60           & 8.42           & 0.00           & 7.19            & 2.58           & 2.81         \\
8   & 110     & 2.15           & 8.41           & 0.00           & 6.01            & 2.01           & 2.41         \\
16  & 90      & 2.38           & 8.91           & 0.00           & 6.51            & 2.03           & 2.33         \\
16  & 100     & 1.95           & 9.04           & 0.00           & 5.18            & 1.67           & 1.92         \\
16  & 110     & 2.20           & 9.95           & 0.00           & 2.80            & 0.30           & 1.84         \\ \hline\hline
\end{tabular}
\end{table}
\endgroup

\paragraph{Scaling to $T=100$.}
Table~\ref{T:Results2} presents results for long-horizon instances with one hundred monthly stages ($T=100$). Consistent with the first theoretical prediction, WTCA again produces the tightest upper bound across all instances. PO's upper bound is, on average, 9.1\% looser than WTCA's, closely matching its 9.5\% average deviation at $T=36$.
The LSM upper bound exceeds the WTCA upper bound by 1.9\%--2.8\%, nearly doubling the deviation observed in the 36-stage instances.
PO's persistently large deviation highlights the limitations of its full temporal coupling. As $T$ grows, PO's per-iteration workload increases sharply, limiting the number of iterations it can perform within a fixed budget to achieve its theoretically tight bound. 
In contrast, WTCA benefits from a horizon-independent per-iteration cost, and continues to generate the strongest bounds.

For lower bounds, WTCA and PO yield similar policies in most instances, with optimality gaps differing by less than 0.5\% in six of nine cases. The two methods diverge primarily at $N=4$, where WTCA's gaps of 2.9\%--3.3\% are substantially tighter than the 3.5\%--9.1\% gaps produced by PO. LSM's gaps range from 2.8\% to 10.5\%.

\paragraph{Cross-horizon comparison.}
WTCA delivers the tightest upper bounds at both $T=36$ and $T=100$, and its average optimality gap increases only modestly from 1.7\% to 2.6\%.
Near-optimal policies are achieved with only 15 random Fourier basis functions, far fewer than sample average approximation approaches, which typically generate much larger optimality gaps with the same number of random Fourier features. For example, in our experiments, solving PO via SAA with 15 random Fourier bases in this context yields an optimality gap of about 50\%.
PO is relatively stable at large values of $N$; specifically, for $N=16$, the average gap is 1.4\% at both $T=36$ and $T=100$. However, PO's performance worsens at $N=4$, where the average gap increases from 4.6\% to 5.9\% because the long horizon amplifies its convergence disadvantage.
LSM deteriorates substantially. Its average gap rises from 4.5\% to 6.9\%, which aligns with backward error accumulation over long horizons. 

Overall, these results indicate that WTCA scales well in both horizon length and dimensionality while maintaining the tightest bounds and near-optimal policies at problem sizes beyond those previously reported.

\subsection{Merchant Ethanol Production}
\label{scc:ethanol-main}
We next evaluate these three approaches on merchant ethanol production~\citep{guthrie2009real,yang2024least,yang2025improved}, which provides a complementary test with richer structure than the Bermudan option setting. An ethanol production plant converts corn and natural gas into ethanol. The operator maximizes expected total profit over a finite planning horizon by dynamically switching the facility among three operating modes: operational ($\mathsf{O}$), mothballed ($\mathsf{M}$), or abandoned ($\mathsf{A}$). When operational, the plant can produce at full capacity, suspend production temporarily, transit to the mothballed state, or permanently shut down, yielding up to four feasible actions per state. This richer action space contrasts with the binary stop/continue decision in Bermudan options. The exogenous state comprises the spot and futures prices of corn, natural gas, and ethanol, represented as three forward curves generated by a multi-factor term-structure model. As a result, its dimension increases with the horizon $T$, unlike the fixed-dimensional asset prices in the Bermudan setting. The full problem formulation, price dynamics, operational parameters, and basis function architecture are provided in \S\ref{sec:EC-numerical}.\looseness=-1

We evaluate WTCA, PO, and LSM on 12 instances corresponding to the months of 2011, each using an initial forward curve from the first trading day of that month. The first set uses a 24-month horizon ($T=24$) and the second extends to 36 months ($T=36$). Following the same protocol as in the Bermudan experiments, all three methods use identical basis functions, and WTCA and PO are compared under equal computational time budgets within each instance. 
The time limits are five hours for the 24-month instances and nine hours for the 36-month instances. Under these limits, WTCA has already stabilized, whereas PO remains far from convergence.
Each instance uses 15 random Fourier basis functions plus commodity-specific terms; full details are given in \S\ref{sec:EC-numerical}. Table~\ref{T:Results3-main} reports results for the $T=24$ instances. 
\begingroup
\renewcommand{\arraystretch}{1.0}
\begin{table}[h!]
\centering
\caption{Upper and lower bound estimates (with standard errors) and relative percentage differences for LSM, PO, and WTCA in 24-stage $\mathbf{(T=24)}$ ethanol production instances}
\begin{tabular}{ccccccc}
\hline\hline
         & \multicolumn{3}{c}{Upper Bound Estimate}         & \multicolumn{3}{c}{Lower Bound Estimate}          \\ \cline{2-7} 
Instance & LSM              & PO               & WTCA            & LSM              & PO               & WTCA             \\ \hline
Jan      & 21.61 (0.05)     & 22.11 (0.05)     & 21.21 (0.05)    & 19.25 (0.07)     & 19.02 (0.07)     & 19.52 (0.07)     \\
Feb      & 21.92 (0.05)     & 21.51 (0.05)     & 21.11 (0.05)    & 19.08 (0.07)     & 19.00 (0.07)     & 19.38 (0.07)     \\
Mar      & 26.05 (0.05)     & 26.61 (0.05)     & 24.99 (0.05)    & 23.44 (0.07)     & 23.09 (0.07)     & 23.49 (0.07)     \\
Apr      & 27.57 (0.05)     & 27.53 (0.05)     & 26.83 (0.05)    & 24.92 (0.07)     & 24.91 (0.07)     & 25.42 (0.07)     \\
May      & 23.64 (0.05)     & 24.29 (0.05)     & 23.09 (0.05)    & 20.95 (0.07)     & 20.87 (0.07)     & 21.07 (0.07)     \\
Jun      & 20.27 (0.05)     & 20.59 (0.05)     & 19.59 (0.05)    & 17.52 (0.07)     & 17.89 (0.07)     & 18.01 (0.07)     \\
Jul      & 17.35 (0.05)     & 17.63 (0.05)     & 16.63 (0.05)    & 14.74 (0.07)     & 15.01 (0.07)     & 15.14 (0.07)     \\
Aug      & 24.49 (0.05)     & 24.44 (0.05)     & 23.74 (0.05)    & 20.76 (0.07)     & 21.36 (0.07)     & 21.76 (0.07)     \\
Sep      & 24.65 (0.05)     & 24.18 (0.05)     & 23.38 (0.05)    & 21.72 (0.07)     & 21.23 (0.07)     & 21.83 (0.07)     \\
Oct      & 22.06 (0.05)     & 21.61 (0.05)     & 21.30 (0.05)    & 19.25 (0.07)     & 19.26 (0.07)     & 19.56 (0.07)     \\
Nov      & 20.50 (0.05)     & 20.31 (0.05)     & 20.00 (0.05)    & 18.04 (0.07)     & 18.13 (0.07)     & 18.43 (0.07)     \\
Dec      & 16.18 (0.05)     & 15.35 (0.05)     & 15.22 (0.05)    & 13.84 (0.07)     & 13.99 (0.07)     & 14.34 (0.07)     \\ \hline
         & \multicolumn{3}{c}{\%(UB - Best UB)/ (Best UB)} & \multicolumn{3}{c}{\%(Best UB - LB) / (Best UB)} \\ \cline{2-7} 
Instance & LSM              & PO               & WTCA            & LSM              & PO               & WTCA             \\ \hline
Jan      & 4.24             & 1.89             & 0.00            & 9.24             & 10.33            & 7.97             \\ 
Feb      & 3.84             & 1.89             & 0.00            & 9.62             & 10.00            & 8.20             \\ 
Mar      & 6.48             & 4.24             & 0.00            & 6.20             & 7.60             & 6.00             \\ 
Apr      & 2.76             & 2.61             & 0.00            & 7.12             & 7.16             & 5.26             \\ 
May      & 2.38             & 5.20             & 0.00            & 9.27             & 9.61             & 8.75             \\ 
Jun      & 3.47             & 5.10             & 0.00            & 10.57            & 8.68             & 8.07             \\ 
Jul      & 7.94             & 6.01             & 0.00            & 11.37            & 9.74             & 8.96             \\ 
Aug      & 3.16             & 2.95             & 0.00            & 12.55            & 10.03            & 8.34             \\ 
Sep      & 5.43             & 3.42             & 0.00            & 7.10             & 9.20             & 6.63             \\ 
Oct      & 3.57             & 1.46             & 0.00            & 9.62             & 9.58             & 8.17             \\ 
Nov      & 2.50             & 1.55             & 0.00            & 9.80             & 9.35             & 7.85             \\ 
Dec      & 6.31             & 0.85             & 0.00            & 9.07             & 8.08             & 5.78             \\ \hline\hline
\end{tabular}
\label{T:Results3-main}
\end{table}
\endgroup

\paragraph{Upper bounds.} WTCA produces the tightest upper bound in all 12 instances. LSM's upper bounds are 2.4\% to 7.9\% higher than WTCA's (average 4.3\%), while PO's are 0.9\% to 6.0\% higher (average 3.1\%). The gap between PO and WTCA is smaller here than in Bermudan options because of the shorter horizon of $T=24$, which reduces PO's per-iteration cost disadvantage under equal budgets. Nevertheless, WTCA's weak coupling structure still enables it to process sufficient samples to overcome its weaker per-basis bound.\looseness=-1

\paragraph{Policy quality.} WTCA yields the strongest operating policy in all 12 instances. Its optimality gaps range from 5.3\% to 9.0\% (average 7.5\%), compared with 9.1\% for PO and 9.3\% for LSM. Unlike the Bermudan option experiments, where WTCA and PO deliver comparable policies at larger $N$, WTCA's policy value exceeds that of PO in every ethanol instance by an average of 1.6\%. This pattern is consistent with the richer action space in ethanol production: with up to four feasible actions per state, action selection depends more on the absolute accuracy of the value function approximation than in the binary stop/continue setting.

\paragraph{Scaling to $T=36$.} Extending the horizon from 24 to 36 months, WTCA continues to deliver the tightest upper bound in every instance. PO's average upper-bound deviation nearly triples from 3.1\% to 9.5\%, reflecting the growing per-iteration cost of SGD for PO as $T$ increases, which limits the number of iterations that can be performed within a fixed budget, while LSM's rises from 4.3\% to 5.7\%. Policy performance is similar across methods at $T=36$, with average optimality gaps of 6.2\% for WTCA, 6.5\% for PO, and 6.7\% for LSM. Full results for the 36-stage instances are reported in \S\ref{sec:EC-numerical}.\looseness=-1

\section{Conclusion}\label{sec:Conclusion}
This paper develops a weakly time-coupled approximation (WTCA) for finite-horizon Markov decision processes, built on the observation that the temporal coupling in optimization-based methods such as ALP and PO is not an intrinsic feature of the underlying MDP but an artifact of the approximation architecture. By enforcing Bellman constraints in expectation over the exogenous distribution rather than pointwise, the resulting formulation limits cross-stage dependence to neighboring periods even when the MDP itself is fully coupled across time. For any fixed basis family, the WTCA bound lies between PO and ALP in tightness and converges to the optimal value as the basis expands. A parallel stochastic block coordinate descent algorithm exploits this structure so that stages update concurrently, achieving computational complexity independent of horizon length for the first time in an optimization-based MDP approximation. Although PO is tighter for any fixed basis, WTCA's horizon-independent complexity allows more samples and basis functions to be processed within a given budget, creating a model selection reversal that widens with the horizon. Numerical experiments on Bermudan options and merchant ethanol production confirm these predictions. Under equal budgets, WTCA produces tighter bounds than PO in every instance tested, and near-optimal policies are obtained with a small number of random basis functions at problem sizes beyond those previously reported. These findings suggest that progress in optimization-based MDP methods can arise from rethinking how approximation architecture and computation interact, and that computational structure can be as important as approximation tightness in determining practical solution quality.

\bibliographystyle{informs2014} 
\bibliography{references}

%% Here starts the e-companion (EC)
%%%%%%%%%%%%%%%%%%%%%%%%%%%%%%%%%%%%%%%%%%%%%%%%%%%%%%%%%%
\ECSwitch

%\ECDisclaimer
%%%%%%%%%%%%%%%%%%%%%%%%%%%%%%%%%%%%%%%%%%%%%%%%%%%%%%%%%%

%%% Main head for the e-companion
\ECHead{Electronic Companion}

This Electronic Companion (EC) contains proofs of all theoretical results stated in the paper and the detailed complexity calculations
supporting the comparison between WTCA and PO.
\S\ref{sec:EC-notations} summarizes notation used throughout the EC. 
\S\ref{sec:EC-ALP-saddle} proves Proposition~\ref{prop:ALP-saddlepoint}, establishing the saddle-point representation of ALP. 
\S\ref{sec:EC-WTCA-ordering} proves Proposition~\ref{prop:WTCA-comparison}, which characterizes the ordering of the WTCA, PO, and ALP upper bounds. 
\S\ref{sec:EC-log-sum-exp-smoothing-curvature} derives properties of the log-sum-exp smoothing that yield explicit block-Lipschitz and curvature
bounds for the smoothed WTCA and PO objectives.
\S\ref{sec:EC-PSBCD} presents the convergence analysis of PS-BCD and proves Theorem~\ref{thm:psbcd-convergencerate}. Although the main paper
requires only the stagewise partition, we develop the analysis for arbitrary block partitions and $\tau$-nice sampling, as this general stochastic block coordinate descent framework may be of independent interest to the optimization community.
\S\ref{sec:EC-complexity} combines the convergence bound with the curvature and per-iteration cost estimates to prove Proposition~\ref{prop:complexity-comparison}.
\S\ref{sec:EC-asymptotics} establishes the asymptotic approximation properties under random basis functions and proves Proposition~\ref{prop:asymptotic-consistency-random}. 
Finally, \S\ref{sec:EC-numerical} provides the full formulation, implementation details, and additional results for the merchant ethanol production experiments summarized in \S\ref{scc:ethanol-main}.\looseness=-1

\section{Notation}\label{sec:EC-notations}

This section collects notation and conventions used throughout the Electronic Companion that are not introduced in the main paper. 

\begin{itemize} 
\item In this EC, the PS-BCD algorithm is presented in a general form in which the vector of basis function weights $\beta \in \R^{BT}$ is partitioned into $n$ blocks, denoted by $\beta = (\beta_{(1)}, \ldots, \beta_{(n)})$, with $\beta_{(i)}\in\R^{n_i}$ and $\sum_{i=1}^n n_i = BT$. The stagewise partition of the main paper corresponds to $n = T$ with $\beta_{(i)} = \beta_t$.

\item For each block $i$, we equip $\mathbb{R}^{n_i}$ with the standard
Euclidean inner product $\langle u,v\rangle = u^\top v$ and define the block norm and its dual as
\[
\|\beta_{(i)}\|_i = \|\beta_{(i)}\|_{i,*} :=\|\beta_{(i)}\|_2.
\]

\item Given curvature weights $\nu_i > 0$, define the weighted norm on
$\mathbb{R}^{BT}$ and its dual by
\[
\|\beta\|_\nu^2
:=
\sum_{i=1}^n \nu_i \|\beta_{(i)}\|_i^2,
\qquad
\|\beta\|_{\nu,*}^2
:=
\sum_{i=1}^n \nu_i^{-1} \|\beta_{(i)}\|_{i,*}^2.
\]
Since each block norm is Euclidean, $\|\cdot\|_\nu$ is a weighted
Euclidean norm on $\mathbb{R}^{BT}$.

\item 
For each block $i$, let $P_i : \mathbb{R}^{BT} \to \mathbb{R}^{n_i}$ denote the canonical projection onto block $i$, so that
\[
P_i d = d_{(i)}, \qquad \forall d \in \mathbb{R}^{BT}.
\]
We denote by  $P_i^\top : \mathbb{R}^{n_i} \to \mathbb{R}^{BT}$, the transpose of $P_i$, which transfers a block vector into the full-dimensional space.
That is, for $d \in \mathbb{R}^{n_i}$,
\begin{equation}\label{eq:lifted-perturbation}
d_{[i]} := P_i^\top d
\end{equation}
is the vector in $\mathbb{R}^{BT}$ whose $i$-th block equals $d$ and whose remaining blocks are zero. Accordingly, $\beta + d_{[i]}$ represents a perturbation of $\beta$ restricted to block $i$.

\item More generally, for a subset $S \subseteq \{1,\ldots,n\}$
and a vector $d = (d_{(1)},\ldots,d_{(n)}) \in \mathbb{R}^{BT}$,
define the block-restricted vector
\begin{equation}\label{eq:d[S]}
d_{[S]}
:=
\sum_{i \in S} P_i^\top d_{(i)}.
\end{equation}
Equivalently, $d_{[S]}$ is obtained from $d$ by retaining the blocks indexed by $S$ and setting all other blocks to zero.
\item A subset $S \subseteq \{1,\ldots,n\}$ follows $\tau$-nice sampling \citep{richtarik2016parallel, AccCD} if it is drawn uniformly at random from all subsets of $\{1,\ldots,n\}$ with expected cardinality $E[|S|] = \tau$.
\item Let $\bar{\mathcal S}$ denote the distribution over block selections induced by $\tau$-nice sampling. Expectation with respect to this sampling scheme is denoted by $\E_{\bar{\mathcal S}}[\cdot]$.
When both block sampling and exogenous randomness are present, we consider the product measure $\bar{\mathcal S}\times\mu$ and write $\E_{\bar{\mathcal S}\times\mu}[\cdot]$ for expectation with respect to the joint distribution. 
\item For each component $f_{j,\sigma}$ in the smoothed objective~\eqref{eq:general-smooth-component},
let $\mathcal{I}_j \subseteq \{1,\ldots,n\}$ denote the set of block indices
whose blocks $\beta_{(i)}$ appear in $f_{j,\sigma}$, and define the
\emph{component coupling degree} $\kappa_j := |\mathcal{I}_j|$.
When $n = T$ and $\beta_{(i)} = \beta_t$, $\mathcal{I}_j$ coincides with
$C_j$ and $\kappa_j$ coincide with the temporal coupling degree defined in the main paper.

\item We say $f_{j,\sigma}$ has \emph{block-Lipschitz gradients} with constant
$L_{ji} > 0$ for block $i$ if
\begin{equation}\label{eq:ESO:block-lipschitz}
\bigl\|\nabla_{\beta_{(i)}} f_{j,\sigma}(\beta + d_{[i]})
-
\nabla_{\beta_{(i)}} f_{j,\sigma}(\beta)\bigr\|_{i,*}
\;\le\;
L_{ji}\,\|d\|_i,
\qquad
\forall\,\beta \in \mathbb{R}^{BT},\; d \in \mathbb{R}^{n_i}.
\end{equation}

\item For the smoothed objective $F_\sigma$ in~\eqref{eq:general-formulation-smooth},
suppose $\beta \in \mathbb{R}^{BT}$ is partitioned into $n$ blocks and
blocks are sampled via $\tau$-nice sampling.
The \emph{curvature weights}
\begin{equation}\label{eq:nu-def-general}
\nu_i
\;=\;
\sum_{j:\,i\in\mathcal{I}_j}
\left(
1
+
\frac{(\kappa_j-1)\,(\tau - 1)}{\max\{1,\,n - 1\}}
\right)
L_{ji},
\qquad i = 1,\ldots,n,
\end{equation}
satisfy the \emph{Expected Separable Overapproximation (ESO)} condition
\citep[see][]{richtarik2014iteration, AccCD}:
\begin{equation}\label{eq:ESO}
\E_{\bar{\mathcal{S}}}\!\left[
F_\sigma(\beta + d_{[S]})
\right]
\;\le\;
F_\sigma(\beta)
+
\frac{\tau}{n}
\left(
\langle \nabla F_\sigma(\beta),\, d \rangle
+
\frac{1}{2}
\sum_{i=1}^n \nu_i \|d_{(i)}\|_i^2
\right),
\end{equation}
where the expectation is taken over the random subset $S$ drawn under $\tau$-nice sampling, 
$d=(d_{(1)},\ldots,d_{(n)})$ is an arbitrary direction, and $d_{[S]}$ denotes the block–restricted perturbation in~\eqref{eq:d[S]}. 
When $n=\tau=T$, the blocks are indexed by stage (so $\beta_{(i)}=\beta_t$), 
the weights in~\eqref{eq:nu-def-general} simplify to those in~\eqref{eq:nu-def}, 
and \eqref{eq:ESO} reduces to \eqref{eq:ESOequation} in the main paper with $\mathcal{I}_j$ replaced by $C_j$.\looseness=-1

\item For any matrix $A \in \mathbb{R}^{|\U| \times BT}$, we use the spectral norm
\begin{equation}\label{eq:matrix-norm}
\|A\|
:=
\sup_{\|z\|_2=1} \|Az\|_2.
\end{equation}
Equivalently, $\|A\|$ is the largest singular value of $A$. The spectral norm is bounded by the Frobenius norm:
\[
\|A\| \le \|A\|_F
:=
\left(\sum_{r,c} a_{rc}^2\right)^{1/2},
\]
where $a_{rc}$ indicates the element of $A$ that is in row $r$ and column $c$.
If the Euclidean norm of each row of $A$, denoted by $\|a_{r,\cdot}\|_2$, is bounded by $R$, then we get\looseness=-1
\begin{equation}\label{eq:F-norm-bound}
\|A\|_F^2
=
\sum_{r=1}^{|\U|} \|a_{r,\cdot}\|_2^2
\le
|\U| R^2,
\qquad\text{and hence}\qquad
\|A\|
\le
\sqrt{|\U|}\,R,
\end{equation}
Thus, uniform $\ell_2$ bounds on the rows immediately yield a bound on the spectral norm.

We denote by $A^{(i)}\in\R^{|\U| \times n_i}$ a submatrix of $A$ whose columns are restricted to block $i$ components. The spectral norm of $A^{(i)}$ is defined analogously.

\item We use the diagonal operator $\diag(p)$ to denote the diagonal matrix whose diagonal entries are given by the vector $p$.

\item We write
\[
D(u,v) := \frac{1}{2}\|u-v\|_\nu^2
\]
for the quadratic distance induced by the weighted norm.

\item By convention, we set $\gamma^{-1}\equiv 0$ in summation formulas involving stage-indexed discount factors, so that boundary terms at $t=0$ vanish without requiring separate treatment. This is a notational shorthand and does not refer to the quantity $1/\gamma$.

\end{itemize}

\section{Saddle-Point Representation of ALP}
\label{sec:EC-ALP-saddle}

This section proves Proposition~\ref{prop:ALP-saddlepoint}. We form the Lagrangian dual of ALP, show that the optimal dual variables decompose as discount-weighted probability measures on the exogenous state space, and use this structure to establish equality of optimal values and solution sets between ALP and the saddle-point problem~\eqref{eq:alp-saddlepoint}.

\begin{proof}{Proof of Proposition~\ref{prop:ALP-saddlepoint}.}

Fix $B>0$ and recall $U_{\mathrm{ALP},B}$ denotes the ALP optimal value. The constraint $\Delta_t(\beta;\sx_t) \le 0$ in~\eqref{eq:ALP-def} is equivalent to the following constraints:
\begin{equation}\label{eq:ALP-disaggregated}
\phi_t(s_t)^\top \beta_t
- \gamma\,\E\!\left[\phi_{t+1}(s_{t+1})^\top \beta_{t+1}
  \,\middle|\, s_t, a_t\right]
\ge r_t(s_t,a_t),
\quad
\forall\,t \in \tset,\; \se_t \in \Endo_t,\;
a_t \in \action_t,\; \sx_t \in \Exo_t,
\end{equation}
where $s_t = (\se_t,\sx_t)$,
$s_{t+1} = (h(\se_t,a_t),\sx_{t+1})$, and
$\hat V_{T,\beta} \equiv 0$.
Each constraint is linear in $\beta$, so~\eqref{eq:ALP-def} is a semi-infinite linear program with finitely many variables $\beta \in \Omega \subseteq \R^{BT}$ and constraints indexed by $(t, \se_t, a_t, \sx_t)$.\looseness=-1

\smallskip
\paragraph{Lagrangian and upper bound.}
For each constraint indexed by $(t,\se_t,a_t)$, introduce a nonnegative measure $\lambda_{t,\se_t,a_t}$ on $\Exo_t$. The Lagrangian is
\begin{equation}\label{eq:ALP-lagrangian}
\mathcal{L}(\beta,\lambda)
=
\hat V_{0,\beta}(s_0)
+
\sum_{t=0}^{T-1}
\sum_{\substack{\se_t \in \Endo_t \\ a_t \in \action_t}}
\int_{\Exo_t}
\Big[
r_t(s_t,a_t)
+ \gamma\,\E\!\left[\phi_{t+1}(s_{t+1})^\top\beta_{t+1}
  \mid s_t,a_t\right]
- \phi_t(s_t)^\top\beta_t
\Big]
\lambda_{t,\se_t,a_t}(d\sx_t),
\end{equation}
where $\lambda = \{\lambda_{t,\se_t,a_t}\}$ is the collection of nonnegative measures. For any fixed $\beta$, if some constraint is violated, concentrating mass at the violating point and scaling arbitrarily gives $\sup_{\lambda \ge 0}\mathcal{L}(\beta,\lambda)=+\infty$.
If all constraints hold, the integrands are nonpositive and the supremum is attained at $\lambda=0$, giving $\hat V_{0,\beta}(s_0)$. Therefore
\begin{equation}\label{eq:ALP-minmax-nonneg}
\min_{\beta \in \Omega}\;
\sup_{\lambda \ge 0}\;
\mathcal{L}(\beta,\lambda)
\;=\;
U_{\mathrm{ALP},B}.
\end{equation}
Any $\alpha \in \mathcal{D}$ induces a nonnegative measure sequence by setting $\lambda_{t,\se_t,a_t}(d\sx_t) = \gamma^t\,\alpha_t(d\sx_t)$ for the $(\se_t,a_t)$ achieving the maximum in $\Delta_t$ (ties broken arbitrarily) and zero otherwise. For this choice,
\[
\mathcal{L}(\beta,\lambda)
= \hat V_{0,\beta}(s_0)
+ \sum_{t=0}^{T-1}\gamma^t\,
\E_{\alpha_t}\!\left[\Delta_t(\beta;\sx_t)\right].
\]
Since this is one particular nonnegative $\lambda$, the supremum over $\alpha \in \mathcal{D}$ is bounded above by the supremum over all nonnegative measures. Minimizing over $\beta$ gives
\begin{equation}\label{eq:upper-bound}
\min_{\beta\in\Omega}\;
\sup_{\alpha \in \mathcal{D}}
\left\{
\hat V_{0,\beta}(s_0)
+\sum_{t=0}^{T-1}\gamma^t\,
\E_{\alpha_t}\!\left[\Delta_t(\beta;\sx_t)\right]
\right\}
\;\le\;
U_{\mathrm{ALP},B}.
\end{equation}

\medskip
\paragraph{Strong duality, dual decomposition, and lower bound.}
Under compactness of $\Omega$, continuity of $\phi_t$ and $r_t$, and the Slater condition (existence of $\hat\beta \in \Omega$ satisfying all constraints strictly), strong duality holds for the semi-infinite linear program~\eqref{eq:ALP-def} \citep[see, e.g.,][]{goberna1998linear}, and the dual optimum is attained at some $\lambda^*$:
\begin{equation}\label{eq:strong-duality}
\sup_{\lambda \ge 0}\;
\inf_{\beta \in \Omega}\;
\mathcal{L}(\beta,\lambda)
\;=\;
U_{\mathrm{ALP},B}.
\end{equation}
We now show that $\lambda^*$ decomposes as $\gamma^t\,\alpha^*_t$
with $\alpha^*_t \in \mathcal{P}(\Exo_t)$, where
$\mathcal{P}(\Exo_t)$ denotes the set of probability measures on
$\Exo_t$.

By the assumption in Proposition~\ref{prop:ALP-saddlepoint} that the ALP optimal solution lies in the interior of $\Omega$, the stationarity conditions at $\lambda^*$ hold as equalities. Collecting the coefficient of $\beta_0$: the objective contributes $\phi_0(s_0)$ and the $t=0$ constraints contribute $-\sum_{a_0}\int_{\Exo_0}\phi_0(s_0)\, \lambda^*_{0,\se_0,a_0}(d\sx_0)$ (recall that $\se_0$ is a singleton).
Applying the stationarity condition to the constant component $\phi_{0,1}\equiv 1$ (which holds by assumption) gives
\begin{equation}\label{eq:stationarity-beta0}
\sum_{a_0 \in \action_0}
\lambda^*_{0,\se_0,a_0}(\Exo_0) = 1,
\end{equation}
so the total mass of the stage-$0$ dual variables equals $1$.

For $t \ge 1$, the coefficient of $\beta_t$ receives contributions from stage-$t$ constraints (containing $-\phi_t(s_t)^\top\beta_t$)
and stage-$(t{-}1)$ constraints (containing $\gamma\,\E[\phi_t(s_t)^\top\beta_t\mid s_{t-1},a_{t-1}]$). Setting this to zero yields the flow-balance equation
\begin{equation}\label{eq:flow-balance}
\sum_{\se_t,\, a_t}
\int_{\Exo_t}\phi_t(\se_t,\sx_t)\,
\lambda^*_{t,\se_t,a_t}(d\sx_t)
=
\gamma
\sum_{\se_{t-1},\, a_{t-1}}
\int_{\Exo_{t-1}}
\E\!\left[\phi_t(h(\se_{t-1},a_{t-1}),\sx_t)
\,\middle|\,\sx_{t-1}\right]
\lambda^*_{t-1,\se_{t-1},a_{t-1}}(d\sx_{t-1}).
\end{equation}
Define the total mass at stage $t$ as $M_t:=\sum_{\se_t\in\Endo_t,\,a_t\in\action_t} \lambda^*_{t,\se_t,a_t}(\Exo_t)$.
Applying~\eqref{eq:flow-balance} to the constant component $\phi_{t,1}\equiv 1$ (by assumption) gives $M_t = \gamma\,M_{t-1}$.
Iterating from $M_0=1$ yields \begin{equation}\label{eq:total-mass}
M_t = \gamma^t, \qquad t \in \tset.
\end{equation}

Define the aggregate measure $\bar\lambda^*_t :=\sum_{\se_t\in\Endo_t,\,a_t\in\action_t} \lambda^*_{t,\se_t,a_t}$, which has total mass $\gamma^t$, and write $\bar\lambda^*_t = \gamma^t\,\alpha^*_t$ with $\alpha^*_t\in\mathcal{P}(\Exo_t)$. Because $\Delta_t$ maximizes over $(\se_t,a_t)$, concentrating each $\lambda^*_{t,\se_t,a_t}$ on the maximizing pair does not decrease the Lagrangian, so the stage-$t$ term becomes $\gamma^t\,\E_{\alpha^*_t}[\Delta_t(\beta;\sx_t)]$. Since $\alpha^*\in\mathcal{D}$,
\[
\sup_{\alpha\in\mathcal{D}}\;
\inf_{\beta\in\Omega}
\left\{
\hat V_{0,\beta}(s_0)
+\sum_{t=0}^{T-1}\gamma^t\,
\E_{\alpha_t}\!\left[\Delta_t(\beta;\sx_t)\right]
\right\}
\;\ge\;
\inf_{\beta\in\Omega}\,
\mathcal{L}(\beta,\lambda^*)
\;=\;
U_{\mathrm{ALP},B},
\]
where the last equality is strong duality~\eqref{eq:strong-duality}. Since $\Omega$ is compact and the objective is linear (hence continuous) in $\beta$, the infimum over $\Omega$ is attained and may be replaced by a minimum throughout. The minimax inequality then gives
$\min_{\beta\in\Omega}\sup_{\alpha\in\mathcal{D}}\{\cdot\} \ge \sup_{\alpha\in\mathcal{D}}\min_{\beta\in\Omega}\{\cdot\} \ge U_{\mathrm{ALP},B}$,
which yields
\begin{equation}\label{eq:lower-bound}
\min_{\beta\in\Omega}\;
\sup_{\alpha\in\mathcal{D}}
\left\{
\hat V_{0,\beta}(s_0)
+\sum_{t=0}^{T-1}\gamma^t\,
\E_{\alpha_t}\!\left[\Delta_t(\beta;\sx_t)\right]
\right\}
\;\ge\;
U_{\mathrm{ALP},B}.
\end{equation}

Combining~\eqref{eq:upper-bound} and~\eqref{eq:lower-bound} establishes equality of optimal values:
\begin{equation}\label{eq:value-equivalence}
\min_{\beta\in\Omega}\;
\sup_{\alpha\in\mathcal D}\;
\left\{
\hat V_{0,\beta}(s_0)
+
\sum_{t=0}^{T-1}\gamma^t\,
\E_{\alpha_t}\!\left[\Delta_t(\beta;\sx_t)\right]
\right\}
\;=\;
U_{\mathrm{ALP},B}.
\end{equation}

\medskip
\paragraph{Equivalence of optimal solution sets.}
It remains to show that the two problems share the same optimal solution set in $\beta$. Suppose $\beta^{\mathrm{ALP}}$ is optimal for ALP. Then $\hat V_{0,\beta^{\mathrm{ALP}}}(s_0) = U_{\mathrm{ALP},B}$, since $\hat V_{0,\beta}(s_0)$ is the ALP objective. Primal feasibility gives $\Delta_t(\beta^{\mathrm{ALP}};\sx_t)\le 0$ for all $t$ and $\sx_t$, so for any $\alpha\in\mathcal{D}$ the saddle-point objective satisfies
\[
\hat V_{0,\beta^{\mathrm{ALP}}}(s_0)
+\sum_{t=0}^{T-1}\gamma^t\,
\E_{\alpha_t}\!\left[\Delta_t(\beta^{\mathrm{ALP}};\sx_t)\right]
\;\le\;
U_{\mathrm{ALP},B}.
\]
Combined with~\eqref{eq:value-equivalence}, $\beta^{\mathrm{ALP}}$ achieves the minimum in~\eqref{eq:value-equivalence} and is therefore, saddle-point optimal. Conversely, suppose $\bar\beta$ achieves the minimum in~\eqref{eq:value-equivalence}. If $\bar\beta$ were primal infeasible for ALP, there would exist $t$, $\sx_t$, and $(\se_t,a_t)$ with $\Delta_t(\bar\beta;\sx_t)>0$; concentrating $\alpha_t$ on that $\sx_t$ would drive the saddle-point objective above $U_{\mathrm{ALP},B}$, contradicting~\eqref{eq:value-equivalence}. Hence $\bar\beta$ is primal feasible for ALP and achieves value $U_{\mathrm{ALP},B}$, so it is ALP optimal.
\hfill\Halmos
\end{proof}

\section{Ordering of ALP, WTCA, and PO Bounds}
\label{sec:EC-WTCA-ordering}

This section proves Proposition~\ref{prop:WTCA-comparison}. For a fixed basis family of size $B$, recall that $U_{\mathrm{PO},B}$, $U_{\mathrm{WTCA},B}$, and $U_{\mathrm{ALP},B}$ denote the optimal objective values obtained by solving the corresponding formulations over $\beta \in \Omega$. We show that WTCA defines a valid upper bound on the optimal value $V_0^*(s_0)$ and that the bounds satisfy
\[
V_0^*(s_0)
\le
U_{\mathrm{PO},B}
\le
F_\mathrm{PO}(\beta^\mathrm{WTCA})
\le
U_{\mathrm{WTCA},B}
\le
U_{\mathrm{ALP},B},
\]
where $\beta^\mathrm{WTCA}$ is an optimal solution to WTCA. We also show that WTCA is a relaxation of ALP.

\begin{proof}{Proof of Proposition~\ref{prop:WTCA-comparison}.}

Fix $B$ and the associated value function approximations $\hat V_{t,\beta}$ for $t\in\tset$. We first prove part (a) and then part (b).

\medskip
\paragraph{Part (a): WTCA is a relaxation of ALP.}

To facilitate comparison with the ALP formulation, we rewrite WTCA as a constrained problem by introducing slack variables.

Introduce auxiliary variables $u_t\ge0$ for $t\in\tset$. Then \eqref{eq:WTCA-def} is equivalent to
\begin{equation}
\label{eq:WTCA-constraints}
\min_{\beta\in\Omega,\;u\in\mathbb R^T}
\hat V_{0,\beta}(s_0)
+
\sum_{t=0}^{T-1}\gamma^t u_t
\quad
\text{s.t.}\quad
\E[\Delta_t(\beta;\sx_t)]\le u_t,
\;\forall t\in\tset .
\end{equation}

Define the feasible sets
\[
\mathcal F_{\mathrm{ALP}}
:=
\{\beta:\; \beta\in\Omega,\;
\Delta_t(\beta;w)\le0,\;
\;\forall w\in\Exo_t,\;\forall t\in\tset\},
\]
and
\[
\mathcal F_{\mathrm{WTCA}}
:=
\{(\beta,u):\;
\beta\in\Omega,\;
u\in\mathbb R_+^T,\;
\E[\Delta_t(\beta;\sx_t)]\le u_t\;
\;\forall t\in\tset\}.
\]

If $\beta\in\mathcal F_{\mathrm{ALP}}$, then $\Delta_t(\beta;w)\le0$ for all $w\in\Exo_t$. Taking expectations yields $\E[\Delta_t(\beta;\sx_t)]\le0$. Hence $(\beta,u)$ with $u_t=0$ is feasible for WTCA. Evaluating the WTCA objective at this point gives $\hat V_{0,\beta}(s_0).$ Therefore
\[
\min_{(\beta,u)\in\mathcal F_{\mathrm{WTCA}}}
\left\{
\hat V_{0,\beta}(s_0)
+
\sum_{t=0}^{T-1}\gamma^t u_t
\right\}
\le
\hat V_{0,\beta}(s_0)
\qquad
\forall\beta\in\mathcal F_{\mathrm{ALP}}.
\]

Taking the minimum over $\beta\in\mathcal F_{\mathrm{ALP}}$ gives
\[
U_{\mathrm{WTCA},B}
\le
U_{\mathrm{ALP},B}.
\]

Thus, WTCA enlarges the feasible region by permitting positive Bellman violations, while penalizing their expected values in the objective. Consequently, WTCA forms a relaxation of ALP and yields an improved upper bound. The validity of $U_{\mathrm{WTCA},B}$ as an upper bound on $ V_0^*(s_0)$ is proven in Part (b).

\medskip
\paragraph{Part (b): Ordering of PO and WTCA bounds.}

For any $\beta\in\Omega$, recall that
\[
F_{\mathrm{PO}}(\beta)
=
\E\!\left[
\max_{a\in\action}
\sum_{t=0}^{T-1}\gamma^t
\Big(
r_t(s_t,a_t)
+
\gamma\,\E\!\left[
\hat V_{t+1,\beta}(s_{t+1})
\mid\sx_t
\right]
-
\gamma\,\hat V_{t+1,\beta}(s_{t+1})
\Big)
\right],
\]
and
\[
F_{\mathrm{WTCA}}(\beta)
=
\hat V_{0,\beta}(s_0)
+
\sum_{t=0}^{T-1}\gamma^t
\E\!\left[\Delta_t(\beta;\sx_t)\right],
\]
where the Bellman deviations $\Delta_t(\beta;\sx_t)$ are defined in \eqref{eq:Delta-Bellmanerror}.

\medskip

To relate the two objectives, we add and subtract $\hat V_{t,\beta}(s_t)$ inside the summand of $F_{\mathrm{PO}}$. In particular, for each stage $t$,\looseness=-1
\begin{align}
r_t(s_t,a_t)
+
\gamma\,\E\!\left[
\hat V_{t+1,\beta}(s_{t+1})
\mid\sx_t
\right]
-
\gamma\,\hat V_{t+1,\beta}(s_{t+1})
&=
\Big(
r_t(s_t,a_t)
+
\gamma\,\E\!\left[
\hat V_{t+1,\beta}(s_{t+1})
\mid\sx_t
\right]
-
\hat V_{t,\beta}(s_t)
\Big)
\nonumber
\\
&\quad+
\Big(
\hat V_{t,\beta}(s_t)
-
\gamma\,\hat V_{t+1,\beta}(s_{t+1})
\Big).
\label{eq:PO-terms}
\end{align}

Along any realized trajectory and any (possibly anticipative) action vector $a=(a_0,\ldots,a_{T-1})$, the endogenous state sequence $\{\se_t\}$ is determined by $(\se_0,a_0,\ldots,a_{t-1})$. Because the exogenous process evolves independently of actions, the conditional distribution of $\sx_{t+1}$ given $(s_t,a_t)$ depends only on $\sx_t$. Hence for each fixed pair $(\se_t,a_t)$,
\[
\E\!\left[
\hat V_{t+1,\beta}(s_{t+1})\,\Big\vert\, s_t,a_t
\right]
=
\E\!\left[
\hat V_{t+1,\beta}(s_{t+1})\,\Big\vert\, \sx_t
\right].
\]

Using this equivalence, the first bracket in \eqref{eq:PO-terms} is bounded above by $\Delta_t(\beta;\sx_t)$ for the realized $\sx_t$ by \eqref{eq:Delta-Bellmanerror}. Thus, for every trajectory $\sx$ and every action vector $a$,
\begin{align*}
\sum_{t=0}^{T-1}\gamma^t
\Big(
r_t(s_t,a_t)
+
\gamma\,\E\!\left[
\hat V_{t+1,\beta}(s_{t+1})
\mid\sx_t
\right]
-
\gamma\,\hat V_{t+1,\beta}(s_{t+1})
\Big)
&\le
\sum_{t=0}^{T-1}\gamma^t
\Delta_t(\beta;\sx_t)
\\
&\quad+
\sum_{t=0}^{T-1}\gamma^t
\Big(
\hat V_{t,\beta}(s_t)
-
\gamma\,\hat V_{t+1,\beta}(s_{t+1})
\Big).
\end{align*}

The second sum telescopes and gives
\[
\sum_{t=0}^{T-1}\gamma^t
\Big(
\hat V_{t,\beta}(s_t)
-
\gamma\,\hat V_{t+1,\beta}(s_{t+1})
\Big)
=
\hat V_{0,\beta}(s_0)-\gamma^T\hat V_{T,\beta}(s_T).
\]
Because $\hat V_{T,\beta}\equiv 0$, this simplifies to $\hat V_{0,\beta}(s_0)$. Consequently,
\[
\sum_{t=0}^{T-1}\gamma^t
\Big(
r_t(s_t,a_t)
+
\gamma\,\E\!\left[
\hat V_{t+1,\beta}(s_{t+1})
\mid\sx_t
\right]
-
\gamma\,\hat V_{t+1,\beta}(s_{t+1})
\Big)
\le
\hat V_{0,\beta}(s_0)
+
\sum_{t=0}^{T-1}\gamma^t
\Delta_t(\beta;\sx_t).
\]

Taking the maximum over $a\in\action$ preserves the inequality, and taking expectations with respect to $\mu$ yields\looseness=-1
\[
F_{\mathrm{PO}}(\beta)
\le
\hat V_{0,\beta}(s_0)
+
\sum_{t=0}^{T-1}\gamma^t
\E\!\left[\Delta_t(\beta;\sx_t)\right]
=
F_{\mathrm{WTCA}}(\beta).
\]

Since this inequality holds for every $\beta\in\Omega$,
it holds in particular at the WTCA optimal solution
$\beta^{\mathrm{WTCA}}$.
Hence
\[
U_{\mathrm{PO},B} = \min_{\beta\in\Omega} F_\mathrm{PO}(\beta) 
\le
F_{\mathrm{PO}}(\beta^{\mathrm{WTCA}})
\le
F_{\mathrm{WTCA}}(\beta^{\mathrm{WTCA}})
=
U_{\mathrm{WTCA},B}.
\]

Because $U_{\mathrm{PO},B}$ is known to be an upper bound on $V^*_0(s_0)$ (see \cite{brown2010information,desai2012pathwise}), the inequality above also implies that WTCA produces a valid upper bound on the optimal value.\hfill\Halmos
\end{proof}

\section{Log-Sum-Exp Properties and Curvature Bounds}
\label{sec:EC-log-sum-exp-smoothing-curvature}

Section~\ref{sec:seqBCD} introduces the log-sum-exp (LSE) smoothing used to construct the differentiable objectives
$F_{\mathrm{WTCA},\sigma}$ and $F_{\mathrm{PO},\sigma}$. This section collects analytical properties of the LSE function and
derives curvature bounds for these smoothed formulations. We first derive gradient and Hessian representations for the log-sum-exp of affine functions, which yield Lipschitz constants for the smoothed gradient. We then specialize these bounds to WTCA and PO and establish explicit
curvature weights $\nu_i$ satisfying the ESO condition~\eqref{eq:ESO:block-lipschitz} for general block partition size $n$ and $\tau$-nice sampling.\looseness=-1

\begin{assumption}\label{ass:normalized-basis functions}
The vectors of basis functions
$\phi_t := (\phi_{t,1},\ldots,\phi_{t,B})$ are uniformly bounded in
$\ell_2$-norm:\looseness=-1
\[
\|\phi_t(s_t)\|_2 \le 1
\qquad
\forall s_t \in \Sspace_t, t\in\tset.
\]
\end{assumption}

This normalization is a standard assumption in linear reinforcement learning and bandit theory \citep{abbasi2011improved, jin2020provably, lattimore2020bandit} and can always be enforced by rescaling the basis functions, which is absorbed into the weights $\beta$ without loss of generality.\looseness=-1

\paragraph{Properties of log-sum-exp of affine functions.}

Recall the log-sum-exp function of the form~\eqref{eq:general-smooth-component}. To simplify notation, we drop the component index $j$ and write
\[
f_\sigma(\beta,\sx)
=
c\cdot\sigma
\log
\sum_{u\in\U}
\exp\!\left(
\frac{a_{\sx}(u)^\top\beta + b_{\sx}(u)}{\sigma}
\right),
\]
where $c>0$ is a scaling constant, $\U$ is a finite index set, and $a_{\sx}(u)\in\R^{BT}$ is the affine coefficient vector. All results in this section hold for each component $f_{j,\sigma}$ by restoring the index $j$. Define the matrix $A_{\sx} \in \R^{|\U|\times BT}$ whose rows are
$a_{\sx}(u)^\top$, $u\in\U$, and let $p(\beta,\sx)\in\R^{|\U|}$ denote the weights
\begin{equation}
\label{eq:p}
p_u(\beta,\sx)
=
\dfrac{
\exp\!\left(
(a_{\sx}(u)^\top\beta + b_{\sx}(u))/\sigma
\right)
}{
\sum_{v\in\U}
\exp\!\left(
(a_{\sx}(v)^\top\beta + b_{\sx}(v))/\sigma
\right)
}.
\end{equation}
It is well known (see, e.g., \citealp{boyd2004convex}) that
\begin{equation*}
\nabla_\beta f_\sigma(\beta,\sx)
=
c\cdot \sum_{u\in\U}
p_u(\beta,\sx)\, a_{\sx}(u),
\end{equation*}
and
\begin{equation}
\label{eq:hessian}
\nabla^2_\beta f_\sigma(\beta,\sx)
=
\frac{c}{\sigma}
A_{\sx}^\top
\big(
\diag(p(\beta,\sx)) - p(\beta,\sx)p(\beta,\sx)^\top
\big)
A_{\sx}.
\end{equation}
Because $p(\beta,\sx)$ is a probability vector, $p_u\in[0,1]$ for all $u$, so $\diag(p)\preceq I$. Since $pp^\top\succeq 0$, it follows that $0\preceq\diag(p)-pp^\top\preceq I$. Substituting into~\eqref{eq:hessian} gives
\begin{equation}
\label{eq:Hessian-bound}
\nabla^2_\beta f_\sigma(\beta,\sx)
\preceq
\frac{c}{\sigma}
A_{\sx}^\top A_{\sx},
\end{equation}
and taking spectral norms yields
\begin{equation}
\label{eq:Hessian-norm-bound}
\|\nabla^2_\beta f_\sigma(\beta,\sx)\|
\le
\frac{c}{\sigma}
\|A_{\sx}\|^2.
\end{equation}

\begin{lemma}[Theorem~5.12 in \citet{beck2017}]
\label{lem:Hessian-and-Lipschitz}
Let $g$ be twice continuously differentiable on $\Omega$. If $\sup_{\beta\in\Omega}\|\nabla^2 g(\beta)\| \le L$, then $\nabla g$ is Lipschitz continuous with constant $L$.
\end{lemma}

Combining \eqref{eq:Hessian-norm-bound} with Lemma~\ref{lem:Hessian-and-Lipschitz}, for each fixed $\sx$ the gradient $\nabla_\beta f_\sigma(\beta,\sx)$ is Lipschitz continuous with constant $L = (c/\sigma)\|A_{\sx}\|^2$.

\medskip
\paragraph{Curvature bounds for WTCA.}

The affine coefficients of the WTCA components $f_{\mathrm{WTCA},t}(\beta,\sx_t)$ are given in~\eqref{eq:WTCA-row0} at $t=0$ and in~\eqref{eq:WTCA-rows} for $t\ge1$. Let $A_{\mathrm{WTCA},t,\sx_t}$ denote the matrix whose rows are $a_{t,\sx_t}(u)^\top$, $u\in\U_t = \Endo_t\times\action_t$. For block $i$, let $A_{\mathrm{WTCA},t,\sx_t}^{(i)}$ denote the submatrix formed by retaining the columns of $A_{\mathrm{WTCA},t,\sx_t}$ corresponding to block $\beta_{(i)}$.

\begin{lemma}
\label{lem:WTCA-smooth}
Fix $\sigma>0$. The smoothed WTCA objective $F_{\mathrm{WTCA},\sigma}$ is convex and continuously differentiable. Moreover:

\begin{enumerate}
\item[(i)] (\emph{Uniform approximation})
\[
0
\le
F_{\mathrm{WTCA},\sigma}(\beta)-F_{\mathrm{WTCA}}(\beta)
\le
\sigma\sum_{t=0}^{T-1}\gamma^t\log(|\Endo_t||\action_t|).
\]

\item[(ii)] (\emph{Block Lipschitz continuity})
For each stage $t$ and block $i$,
\begin{equation}\label{eq:WTCA-Lipschitz-it}
L_{ti}^{\mathrm{WTCA}}
:=
\frac{\gamma^t}{\sigma}
\E\!\left[
\|A_{\mathrm{WTCA},t,\sx_t}^{(i)}\|^2
\right]
\end{equation}
is a valid Lipschitz constant of $\nabla_{\beta_{(i)}}\E[f_{\mathrm{WTCA},t,\sigma}(\beta,\sx_t)]$. Consequently,
\begin{equation}\label{eq:WTCA-block-Lipschitz}
L_i^{\mathrm{WTCA}}
:=
\sum_{t:\,i\in C_t}
L_{ti}^{\mathrm{WTCA}}
=
\sum_{t:\,i\in C_t}
\frac{\gamma^t}{\sigma}
\E\!\left[
\|A_{\mathrm{WTCA},t,\sx_t}^{(i)}\|^2
\right]
\end{equation}
is a valid block Lipschitz constant of $\nabla_\beta F_{\mathrm{WTCA},\sigma}$.

\item[(iii)] (\emph{ESO curvature weights})
For general block partition size $n$ and $\tau$-nice sampling, the curvature weights
\begin{equation}\label{eq:WTCA-nu}
\nu_i
\leq
\sum_{t:\,i\in C_t}
\left(
1+\frac{(\kappa_t-1)(\tau-1)}{\max\{1,n-1\}}
\right)
\frac{\gamma^t(1+\gamma^2)}{\sigma}
|\Endo_t||\action_t|,
\qquad i=1,\ldots,n,
\end{equation}
satisfy the ESO condition~\eqref{eq:nu-def-general}. For $n=\tau=T$ this reduces to
$\nu_i \leq \sum_{t:\,i\in C_t} \frac{2\gamma^t(1+\gamma^2)}{\sigma}|\Endo_t||\action_t|$, consistent with~\eqref{eq:nu-def} of the main paper.
\end{enumerate}
\end{lemma}

\begin{proof}{Proof.}
Part (i) follows from the standard log-sum-exp bound (Example~4.5 in \citealp{smoothingfirstordermethods}) applied to each stage component.
Part (ii) follows from applying inequality~\eqref{eq:Hessian-bound} to $f_{\mathrm{WTCA},t,\sigma}$ with $c_t=\gamma^t$ which yields
\[
\nabla^2_\beta
f_{\mathrm{WTCA},t,\sigma}(\beta,\sx_t)
\preceq
\frac{\gamma^t}{\sigma}\,
A_{\mathrm{WTCA},t,\sx_t}^\top
A_{\mathrm{WTCA},t,\sx_t}.
\]
Restricting this bound to block $i$, taking expectation with respect to~$\mu$, and invoking Lemma~\ref{lem:Hessian-and-Lipschitz} produces $L_{ti}^{\mathrm{WTCA}}$ as stated. Summing over all component functions that involve block $i$ then yields $L_i^{\mathrm{WTCA}}$. 
We prove part (iii) by using the ESO formula~\eqref{eq:nu-def-general} applied with component indices $j = t$ and block-Lipschitz constants $L_{ti}^{\mathrm{WTCA}}$, which provides
\[
\nu_i
=
\sum_{t:\,i\in C_t}
\left(
1+\frac{(\kappa_t-1)(\tau-1)}{\max\{1,n-1\}}
\right)
L_{ti}^{\mathrm{WTCA}}.
\]
Using \eqref{eq:WTCA-Lipschitz-it}, it remains to bound 
$\E[\|A_{\mathrm{WTCA},t,\sx_t}^{(i)}\|^2]$.  For any row $a_{t,\sx_t}(u)\in\R^{BT}$, the block-$i$ component satisfies  $\|a_{t,\sx_t}^{(i)}(u)\|_2 \le \|a_{t,\sx_t}(u)\|_2$. For $t=0$, from~\eqref{eq:WTCA-row0}, the row has one nonzero block $\gamma\,\E[\phi_1(s_1)\mid s_0,a_0]$, so $\|a_{0,\sx_0}(u)\|_2^2\le\gamma^2\le 1+\gamma^2$. For $t\ge1$, from~\eqref{eq:WTCA-rows}, the row has at most two nonzero blocks: $-\phi_t(\se_t,\sx_t)$ and $\gamma\,\E[\phi_{t+1}(h(\se_t,a_t),\sx_{t+1})\mid\sx_t]$. By Jensen's inequality and Assumption~\ref{ass:normalized-basis functions}, $\|a_{t,\sx_t}(u)\|_2^2\le 1+\gamma^2$. Hence in both cases $\|a_{t,\sx_t}^{(i)}(u)\|_2^2\le 1+\gamma^2$. Applying~\eqref{eq:F-norm-bound} with $|\U_t|=|\Endo_t||\action_t|$ rows gives $\|A_{\mathrm{WTCA},t,\sx_t}^{(i)}\|^2 \le(1+\gamma^2)|\Endo_t||\action_t|$, which is deterministic and therefore equals its expectation. Substituting into $L_{ti}^{\mathrm{WTCA}}$ and then into the ESO formula yields~\eqref{eq:WTCA-nu}.
\hfill\Halmos
\end{proof}

\medskip
\paragraph{Curvature bounds for PO.}

Using the affine representation in~\eqref{eq:PO-rows}, let $A_{\mathrm{PO},\sx}$ denote the matrix whose rows are $a_{\sx}(u)^\top$, $u\in\U=\action$. For block $i$, let $A_{\mathrm{PO},\sx}^{(i)}$ denote the corresponding block-restricted submatrix formed by retaining the columns of $A_{\mathrm{PO},\sx}$ corresponding to block $\beta_{(i)}$.

\begin{lemma}
\label{lem:PO-smooth}
Fix $\sigma>0$. The smoothed PO objective $F_{\mathrm{PO},\sigma}$ is convex and continuously differentiable. Moreover:

\begin{enumerate}
\item[(i)] (\emph{Uniform approximation})
\[
0
\le
F_{\mathrm{PO},\sigma}(\beta)-F_{\mathrm{PO}}(\beta)
\le
\sigma\log(|\action|).
\]

\item[(ii)] (\emph{Block Lipschitz continuity})
For each block $i$,
\begin{equation}\label{eq:PO-Lipschitz}
L_i^{\mathrm{PO}}
:=
\frac{1}{\sigma}
\E\!\left[
\|A_{\mathrm{PO},\sx}^{(i)}\|^2
\right]
\end{equation}
is a valid block Lipschitz constant of $\nabla_\beta F_{\mathrm{PO},\sigma}$.

\item[(iii)] (\emph{ESO curvature weights})
For general block partition size $n\ge1$ and $\tau$-nice sampling, the curvature weights
\begin{equation}\label{eq:PO-nu-general}
\nu_i
\;\le\;
\left(
1+\frac{(n-1)(\tau-1)}{\max\{1,n-1\}}
\right)
\frac{4\gamma^2\,|\action|}{\sigma(1-\gamma^2)}
\;=\;
\frac{4\tau\,\gamma^2\,|\action|}{\sigma(1-\gamma^2)},
\qquad i=1,\ldots,n,
\end{equation}
satisfy the ESO condition~\eqref{eq:nu-def-general}.
For $n=\tau=T$ with block $i$ coinciding with stage $t$,
the tighter stagewise bound
\begin{equation}\label{eq:PO-nu}
\nu_i
\;\le\;
\frac{4T\,\gamma^{2i}\,|\action|}{\sigma}
\end{equation}
holds, consistent with~\eqref{eq:nu-def} of the main paper.
\end{enumerate}
\end{lemma}

\begin{proof}{Proof.}
Part (i) follows from the log-sum-exp properties stated in Example~4.5 of \citealp{smoothingfirstordermethods}. Part (ii) follows from \eqref{eq:Hessian-bound} with $c=1$ which implies
\[
\nabla^2_\beta
f_{\mathrm{PO},\sigma}(\beta,\sx)
\preceq
\frac{1}{\sigma}
A_{\mathrm{PO},\sx}^\top
A_{\mathrm{PO},\sx}.
\]
Restricting to block $i$, taking expectation over $\mu$, and applying Lemma~\ref{lem:Hessian-and-Lipschitz} yields $L_i^{\mathrm{PO}}$ as stated.
For part (iii), since PO has a single component ($m=1$) with $\kappa_1=n$, the ESO formula~\eqref{eq:nu-def-general} gives
\[
\nu_i
=
\left(
1+\frac{(n-1)(\tau-1)}{\max\{1,n-1\}}
\right)
L_i^{\mathrm{PO}}
=
\tau\,L_i^{\mathrm{PO}},
\qquad i=1,\ldots,n,
\]
for all $n\ge1$, where the last equality holds since $(1+(n-1)(\tau-1)/\max\{1,n-1\})=\tau$ for $n\ge2$ and $\nu_1=L_1^{\mathrm{PO}}=\tau L_1^{\mathrm{PO}}$ trivially for $n=\tau=1$. Since $L_i^{\mathrm{PO}}
:=
\dfrac{1}{\sigma}
\E\!\big[
\|A_{\mathrm{PO},\sx}^{(i)}\|^2\big]$ by \eqref{eq:PO-Lipschitz}, it remains to bound $\E[\|A_{\mathrm{PO},\sx}^{(i)}\|^2]$.

Since $A_{\mathrm{PO},\sx}^{(i)}$ is the column restriction of $A_{\mathrm{PO},\sx}$ to block $i$, each row satisfies $\|a_\sx^{(i)}(u)\|_2\le\|a_\sx(u)\|_2$ for every $u\in\action$. From~\eqref{eq:PO-rows} and the triangle inequality, each stagewise coefficient satisfies
$\|c_{t+1}(\sx,a_t)\|_2\le 2\gamma^{t+1}$ under Assumption~\ref{ass:normalized-basis functions}. Since the coefficient vector $a_\sx(u)$ decomposes stagewise as $a_\sx(u)=(0,c_1(\sx,a_0),\ldots,c_T(\sx,a_{T-1}))$,
\[
\|a_\sx(u)\|_2^2
=
\sum_{t=0}^{T-1}\|c_{t+1}(\sx,a_t)\|_2^2
\le
\sum_{t=0}^{T-1}4\gamma^{2(t+1)}
\le
\frac{4\gamma^2}{1-\gamma^2}.
\]
Hence $\|a_\sx^{(i)}(u)\|_2^2\le 4\gamma^2/(1-\gamma^2)$ for every block $i$ and every partition. Applying~\eqref{eq:F-norm-bound} with $|\action|$ rows gives
\[
\|A_{\mathrm{PO},\sx}^{(i)}\|^2
\le
\|A_{\mathrm{PO},\sx}^{(i)}\|_F^2
\le
\frac{4\gamma^2\,|\action|}{1-\gamma^2},
\]
and therefore $L_i^{\mathrm{PO}}\le 4\gamma^2|\action|/(\sigma(1-\gamma^2))$. Substituting into $\nu_i=\tau\,L_i^{\mathrm{PO}}$ yields~\eqref{eq:PO-nu-general}.

When $n=\tau=T$ and block $i$ coincides with stage $t$, i.e., $\beta_{(i)}=\beta_t$, the block-$i$ ($i=t$) portion of each row $a_\sx(u)$ is exactly the stagewise coefficient $c_t(\sx,a_{t-1})$ from~\eqref{eq:PO-rows}. For $t=0$, this coefficient is zero since $a_\sx(u)$ has
no $\beta_0$ component. For $t\ge1$, by the triangle inequality and Assumption~\ref{ass:normalized-basis functions},
$\|c_t(\sx,a_{t-1})\|_2\le 2\gamma^t$. Hence $\|a_\sx^{(i)}(u)\|_2^2\le 4\gamma^{2t}$ for all $u\in\action$. Applying~\eqref{eq:F-norm-bound} with $|\action|$ rows gives
\[
\|A_{\mathrm{PO},\sx}^{(i)}\|^2
\le
\|A_{\mathrm{PO},\sx}^{(i)}\|_F^2
\le
4\gamma^{2i}\,|\action|,
\]
so $L_i^{\mathrm{PO}}\le 4\gamma^{2i}|\action|/\sigma$. Substituting into $\nu_i=T\,L_i^{\mathrm{PO}}$ yields~\eqref{eq:PO-nu}.
\hfill\Halmos
\end{proof}

\section{General PS-BCD and Convergence Rate}
\label{sec:EC-PSBCD}

This section presents a generalization of the PS-BCD (GPS-BCD) algorithm of~\S\ref{sec:sbcd-algorithm} to arbitrary block partitions of $\beta\in\mathbb{R}^{BT}$ and $\tau$-nice block sampling applied to the general smoothed stochastic optimization problem~\eqref{eq:general-formulation-smooth}. We develop this general framework because stochastic block coordinate descent with $\tau$-nice sampling and ESO-based curvature weights may be of independent interest to the optimization community beyond the MDP setting studied here.
The stagewise PS-BCD algorithm of~\S\ref{sec:sbcd-algorithm} and the classical stochastic gradient descent (SGD) method are
recovered as special cases corresponding to $n=\tau=T$ and $n=\tau=1$, respectively. We then establish the convergence rate of this general algorithm and derive Theorem~\ref{thm:psbcd-convergencerate} as a consequence.

Let the coefficient vector be partitioned into $n$ blocks
\[
\beta=(\beta_{(1)},\ldots,\beta_{(n)}),
\qquad
\beta_{(i)}\in\mathbb{R}^{n_i},
\qquad
\sum_{i=1}^n n_i = BT.
\]
At each iteration, GPS-BCD selects a random subset $S_k \subseteq \{1,\ldots,n\}$ and updates only those blocks, keeping the remaining blocks fixed. The subset $S_k$ follows the \emph{$\tau$-nice sampling} scheme of~\citet{richtarik2014iteration}, in which every subset of size $\tau$ is selected with equal probability, so $\E_{\bar{\mathcal S}}[|S_k|]=\tau$. At each iteration, $S_k$ and the exogenous trajectory sample $\sx^k$ are drawn independently from $\xi$.

For a given iterate $\beta$ and exogenous sample $\sx$, the stochastic gradient of the smoothed objective $F_\sigma(\beta) = \E_\xi[\sum_{j=1}^m f_{j,\sigma}(\beta,\sx)]$ is
\begin{equation}\label{eq:stochastic-gradient}
G(\beta,\sx)
:=
\sum_{j=1}^m \nabla_\beta f_{j,\sigma}(\beta,\sx)
=
\sum_{j=1}^m \sum_{u\in\U_j}
p_{j,u}(\beta,\sx)\,a_{j,\sx}(u),
\end{equation}
where $p_{j,u}(\beta,\sx)$ is the  weight defined in~\eqref{eq:p}. This is an unbiased estimator of $\nabla_\beta F_\sigma(\beta)$, since the bounded gradients of $f_{j,\sigma}$ on $\Omega$ justify interchanging differentiation and expectation.

The block-$i$ component of $G(\beta,\sx)$ is
\begin{equation}\label{eq:stochastic-gradient-block}
G_{(i)}(\beta,\sx)
:=
\sum_{j:\,i\in\mathcal{I}_j}
\sum_{u\in\U_j}
p_{j,u}(\beta,\sx)\,a_{j,\sx}^{(i)}(u),
\end{equation}
where $a_{j,\sx}^{(i)}(u)$ denotes the block-$i$ portion of the affine coefficient vector $a_{j,\sx}(u)$, and the outer sum runs only over components $j$ whose index set $\mathcal{I}_j$ contains block $i$. Note that $G_{(i)}$ is the $i$-th block of $G$, i.e., $G_{(i)}(\beta,\sx) = P_i G(\beta,\sx)$. Each active block $i\in S_k$ is then updated by 
\begin{equation}\label{eq:beta-update-general}
\beta_{(i)}^{k+1}
=
\arg\min_{z\in\Omega_i}
\left\{
\langle G_{(i)}(\beta^k,\sx^k), z-\beta_{(i)}^k\rangle
+
\frac{\nu_i}{2\alpha_k}\|z-\beta_{(i)}^k\|_i^2
\right\},
\qquad i\in S_k,
\end{equation}
while $\beta_{(i)}^{k+1}=\beta_{(i)}^k$ for $i\notin S_k$. Here $\nu_i$ denotes the curvature weight in~\eqref{eq:nu-def-general}, $\alpha_k$ is the stepsize schedule, and $\Omega_i$ is the projection of the feasible set onto block $i$. This update is the general block-coordinate version  of~\eqref{eq:parallel-block-update}.

Algorithm~\ref{alg:psbcd-general} summarizes the general PS-BCD scheme.

\begin{algorithm}[H]
\caption{General Parallel Stochastic Block Coordinate Descent (GPS-BCD)}
\label{alg:psbcd-general}
\begin{algorithmic}[1]
\Require Initial point $\beta^1\in\Omega$, stepsize schedule
$\alpha_k = 1/(2\sqrt{k+1})$ starting from $\alpha_1 = 1/2\sqrt{2}$, and curvature weights $\{\nu_i\}_{i=1}^n$
\For{$k=1,2,\ldots, K$}
\State Sample block subset $S_k$ from $\bar{\mathcal S}$ with
  $\E[|S_k|]=\tau$
\State Sample exogenous trajectory $\sx^k$ from $\xi$
  independently of $S_k$
\For{each block $i\in S_k$ \textbf{in parallel}}
\State Compute block stochastic gradient
          $G_{(i)}(\beta^k,\sx^k)$ via~\eqref{eq:stochastic-gradient-block}
\State Update $\beta_{(i)}^{k+1}$ via \eqref{eq:beta-update-general}
\EndFor
\For{each block $i\notin S_k$}
\State $\beta_{(i)}^{k+1}=\beta_{(i)}^k$
\EndFor
\EndFor
\State \textbf{Output}: $\bar\beta^K = \frac{\sum_{k=2}^{K}\vartheta_{k}\beta^k} {\sum_{k=1}^{K-1}\vartheta_k},$ where $\vartheta_k := n^2\alpha_{k-1} -(n^2 -\tau^2)\alpha_k$. 
\end{algorithmic}
\end{algorithm}

The special case analyzed in the main paper corresponds to the
stagewise partition with $n=T$ and $\tau=T$, where all stage
blocks are updated simultaneously.

\medskip

\begin{theorem}\label{thm:psbcd-general}
Suppose the stochastic gradient $G(\beta,\sx)$
in~\eqref{eq:stochastic-gradient} is an unbiased estimator
of $\nabla_\beta F_\sigma(\beta)$ and satisfies the bounded
variance condition
\[
\E_\xi\!\left[
\left\|
G(\beta,\sx) - \nabla_\beta F_\sigma(\beta)
\right\|_{\nu,*}^2
\;\middle|\;
\beta
\right]
\le \chi^2
\qquad \text{for all } \beta\in\Omega.
\]
Then the averaged iterate $\bar\beta^K$ produced by
Algorithm~\ref{alg:psbcd-general} satisfies
\begin{equation}\label{eq:psbcd-general-bound}
\E_{\bar{\mathcal{S}}\times\xi}[F_\sigma(\bar\beta^K)]-F_\sigma(\beta^*)
\le
\mathcal O\left(\frac{
(n^2-\tau^2)(F_\sigma(\beta^1)-F_\sigma(\beta^*))
+
n\tau \|\beta^1-\beta^*\|^2_\nu
+
\tau^2\log(K+1)\chi^2
}{
\tau^2(\sqrt{K+2}-2)
}\right).
\end{equation}
where $\|\cdot\|_\nu$ is the weighted norm defined in \eqref{eq:nu-def-general}.
\end{theorem}

\begin{proof}{Proof.}
Consider iteration $k$ and define the update increment $d^k := \beta^{k+1}-\beta^k$. By construction, $d^k_{(i)}=0$ for all $i\notin S_k$, so $d^k=d^k_{[S_k]}$.
To apply the ESO to a fixed direction, define for each block $i$ the deterministic candidate update
\[
\tilde d_{(i)}
:=
\arg\min_{z-\beta_{(i)}^k\in\R^{n_i}}
\left\{
\langle G_{(i)}(\beta^k,\sx^k), z-\beta_{(i)}^k\rangle
+
\frac{\nu_i}{2\alpha_k}\|z-\beta_{(i)}^k\|_i^2
\right\},
\]
so that $\tilde{d}_{(i)}$ is fixed given $\sx^k$.
Since $\Omega$ is chosen large enough to contain the unconstrained minimizers in its interior, the constrained and unconstrained updates coincide, giving $d^k = \tilde{d}_{[S_k]}$.
Applying the ESO \eqref{eq:ESO:block-lipschitz} to the
fixed direction $\tilde d$ gives
\begin{align}
\E_{\bar{\mathcal{S}}}\!\left[F_\sigma(\beta^k+d^k)\right]
&\le
F_\sigma(\beta^k)
+\frac{\tau}{n}
\Big[
\langle \nabla_\beta F_\sigma(\beta^k), \tilde d\rangle
+\tfrac{1}{2}\|\tilde d\|_\nu^2
\Big]
\nonumber\\
&=
F_\sigma(\beta^k)
+\frac{\tau}{n}
\Big[
\langle G(\beta^k,\sx^k), \tilde d\rangle
-
\langle \delta^k, \tilde d\rangle
+\tfrac{1}{2}\left\|\tilde d\right\|_\nu^2
\Big],
\label{eq:ESO-ineq-updated}
\end{align}
where $\delta^k := G(\beta^k,\sx^k) - \nabla_\beta F_\sigma(\beta^k)$ is the stochastic noise, satisfying $\E[\delta^k \mid \beta^k]=0$ by our unbiased stochastic gradient assumption. 

From the first-order optimality of subproblem~\eqref{eq:beta-update-general}, for each $i\in S_k$ and any optimal solution $\beta^*$, the standard three-point proximal inequality gives
\begin{align}
\left\langle G_{(i)}(\beta^k,\sx^k),
\beta_{(i)}^{k+1}-\beta_{(i)}^k\right\rangle
+\frac{\nu_i}{2\alpha_k}
\left\|\beta_{(i)}^{k+1}-\beta_{(i)}^k\right\|_i^2
&\le
\left\langle G_{(i)}(\beta^k,\sx^k),
\beta_{(i)}^*-\beta_{(i)}^k\right\rangle
\nonumber\\
&\quad
+\frac{\nu_i}{2\alpha_k}
\Big(
\left\|\beta_{(i)}^*-\beta_{(i)}^k\right\|_i^2
-
\left\|\beta_{(i)}^*-\beta_{(i)}^{k+1}\right\|_i^2
\Big).
\label{eq:opt-ineq-updated}
\end{align}
See, e.g., Lemma~1 of \citet{lan2012optimal} for a detailed
derivation.

Adding
$\tfrac{1}{2}\nu_i\|\beta_{(i)}^{k+1}-\beta_{(i)}^k\|_i^2
-
\langle \delta_{(i)}^k,\beta_{(i)}^{k+1}-\beta_{(i)}^k\rangle$
to both sides of \eqref{eq:opt-ineq-updated} and applying the
Young inequality
\[
-\langle \delta^k_{(i)}, \beta_{(i)}^{k+1}-\beta_{(i)}^k\rangle
-
\frac{c}{2}\|\beta_{(i)}^{k+1}-\beta_{(i)}^k\|_i^2
\le
\frac{1}{2c}\| \delta^k_{(i)}\|_{i,*}^2,
\qquad c>0,
\]
with $c=\nu_i(\dfrac{1}{\alpha_k}-1)>0$ yields
\[
-\frac{\nu_i}{2\alpha_k}\|\beta_{(i)}^{k+1}-\beta_{(i)}^k\|_i^2
+\frac{\nu_i}{2}\|\beta_{(i)}^{k+1}-\beta_{(i)}^k\|_i^2
-\langle\delta_{(i)}^k,\beta_{(i)}^{k+1}-\beta_{(i)}^k\rangle
\le
\frac{\alpha_k}{2\nu_i(1-\alpha_k)}\left\|\delta_{(i)}^k\right\|_{i,*}^2
\le
\frac{\alpha_k}{\nu_i}\,\left\|\delta_{(i)}^k\right\|_{i,*}^2,
\]
where the last inequality uses $\alpha_k\le \dfrac12$ for all $k\ge1$ since $\alpha_k = 1/(2\sqrt{k+1})$.
Multiplying \eqref{eq:opt-ineq-updated} by
$\mathbbm{1}_{\{i\in S_k\}}$ and summing over $i$ gives
\begin{align}
\langle G(\beta^k,\sx^k),d^k\rangle
+\tfrac{1}{2}\|d^k\|_\nu^2
-\langle\delta^k,d^k\rangle
&\le
\sum_{i=1}^n
\mathbbm{1}_{\{i\in S_k\}}
\langle G_{(i)}(\beta^k,\sx^k),
\beta_{(i)}^*-\beta_{(i)}^k\rangle
\nonumber\\
&\quad
+\frac{1}{2\alpha_k}
\sum_{i=1}^n
\mathbbm{1}_{\{i\in S_k\}}
\nu_i
\Big(
\|\beta_{(i)}^*-\beta_{(i)}^k\|_i^2
-
\|\beta_{(i)}^*-\beta_{(i)}^{k+1}\|_i^2
\Big)
\nonumber\\
&\quad
+\sum_{i=1}^n
\mathbbm{1}_{\{i\in S_k\}}
\alpha_k\nu_i^{-1}\|\delta_{(i)}^k\|_{i,*}^2.
\label{eq:block-sum-updated}
\end{align}

We now take conditional expectations of each term on the
right-hand side of~\eqref{eq:block-sum-updated} over
$\bar{\mathcal{S}}\times\xi$ given $\beta^k$,
then multiply by $\dfrac{\tau}{n}$ to combine
with~\eqref{eq:ESO-ineq-updated}.

\emph{Gradient term.}
Using $\E[\delta^k\mid\beta^k]=0$ and $\Pr(i\in S_k)=\tau/n$,
\begin{align}
\E_{\bar{\mathcal{S}}\times\xi}
\!\Big[
\sum_{i\in S_k}
\langle G_{(i)}(\beta^k,\sx^k),
\beta_{(i)}^*-\beta_{(i)}^k\rangle
\ \Big|\ \beta^k
\Big]
&=
\frac{\tau}{n}
\langle \nabla_\beta F_\sigma(\beta^k),
\beta^*-\beta^k\rangle
\le
\frac{\tau}{n}
\big(
F_\sigma(\beta^*)
-
F_\sigma(\beta^k)
\big),
\label{eq:convexity-updated}
\end{align}
where the inequality uses convexity of $F_\sigma$. After multiplying by $\dfrac{\tau}{n}$, this contributes $\dfrac{\tau^2}{n^2}(F_\sigma(\beta^*)-F_\sigma(\beta^k))$. \looseness=-1

\emph{Distance term.}
For each $i$, block $i$ is updated only when $i\in S_k$,
and when $i\notin S_k$ we have
$\beta_{(i)}^{k+1}=\beta_{(i)}^k$. Therefore,
$\E[\|\beta_{(i)}^*-\beta_{(i)}^{k+1}\|_i^2]
=
\dfrac{\tau}{n}\E[\|\beta_{(i)}^*-\beta_{(i)}^{k+1}\|_i^2\mid i\in S_k]
+
(1-\dfrac{\tau}{n})\|\beta_{(i)}^*-\beta_{(i)}^k\|_i^2,$
so
$\E[\mathbbm{1}_{\{i\in S_k\}}(\|\beta_{(i)}^*-\beta_{(i)}^k\|_i^2
-\|\beta_{(i)}^*-\beta_{(i)}^{k+1}\|_i^2)]
=
\|\beta_{(i)}^*-\beta_{(i)}^k\|_i^2
-\E[\|\beta_{(i)}^*-\beta_{(i)}^{k+1}\|_i^2].$
Summing over $i$ with weights $\nu_i$ gives
\begin{align}
\E_{\bar{\mathcal{S}}\times\xi}
\!\Big[
\frac{1}{2\alpha_k}
\sum_{i\in S_k}
\nu_i
\big(
\|\beta_{(i)}^*-\beta_{(i)}^k\|_i^2
-
\|\beta_{(i)}^*-\beta_{(i)}^{k+1}\|_i^2
\big)
\ \Big|\ \beta^k
\Big]
&=
\frac{1}{\alpha_k}
\Big(
D(\beta^*,\beta^k)
-
\E_{\bar{\mathcal{S}}\times\xi}\!\left[D(\beta^*,\beta^{k+1})
\mid\beta^k\right]
\Big).
\label{eq:distance-updated}
\end{align}
After multiplying by $\dfrac{\tau}{n}$, this contributes
$\dfrac{\tau}{n\alpha_k}(D(\beta^*,\beta^k)
-\E[D(\beta^*,\beta^{k+1})\mid\beta^k])$.

\emph{Variance term.}
Since $S_k$ is drawn independently of $\sx^k$,
conditioning on $(\beta^k,\sx^k)$ gives
\[
\E_{\bar{\mathcal{S}}}\!\left[
\sum_{i\in S_k}\nu_i^{-1}\|\delta_{(i)}^k\|_{i,*}^2
\,\middle|\, \beta^k,\sx^k\right]
=
\sum_{i=1}^n \Pr(i\in S_k)\,\nu_i^{-1}\|\delta_{(i)}^k\|_{i,*}^2
=
\frac{\tau}{n}\,\|\delta^k\|_{\nu,*}^2.
\]
Taking $\E[\cdot\mid\beta^k]$ and applying
our bounded variance condition:
\begin{equation}
\label{eq:variance-updated}
\E_{\bar{\mathcal{S}}\times\xi}\!\left[
\sum_{i\in S_k}\nu_i^{-1}\|\delta_{(i)}^k\|_{i,*}^2
\,\middle|\, \beta^k\right]
= \frac{\tau}{n}\,\E\!\left[\|\delta^k\|_{\nu,*}^2 \mid \beta^k\right]
\le
\frac{\tau}{n}\,\chi^2.
\end{equation}
After multiplying by $\dfrac{\tau}{n}$, this contributes
$\dfrac{\tau^2\alpha_k}{n^2}\chi^2$.

Substituting the three expectations
into~\eqref{eq:block-sum-updated},
multiplying through by $\dfrac{\tau}{n}$, and
adding $F_\sigma(\beta^k)$ via~\eqref{eq:ESO-ineq-updated} yields
the one-step inequality
\begin{align}
\E_{\bar{\mathcal{S}}\times\xi}\!\left[F_\sigma(\beta^{k+1})
\mid \beta^k\right]
+\frac{\tau}{n\alpha_k}
\E_{\bar{\mathcal{S}}\times\xi}\!\left[D(\beta^*,\beta^{k+1})
\mid \beta^k\right]
&\le
\left(1-\frac{\tau^2}{n^2}\right)
F_\sigma(\beta^k)
+
\frac{\tau^2}{n^2}
F_\sigma(\beta^*)
\nonumber\\
&\quad
+
\frac{\tau}{n\alpha_k}
D(\beta^*,\beta^k)
+
\frac{\tau^2\alpha_k}{n^2}
\chi^2.
\label{eq:step-improve-updated}
\end{align}

Taking full expectations and summing~\eqref{eq:step-improve-updated}
for $k=1,\ldots,K$ with $\alpha_k=1/(2\sqrt{k+1})$ gives
\begin{align*}
&\sum_{k=2}^K
\vartheta_k
\E_{\bar{\mathcal{S}}\times\xi}[F_\sigma(\beta^k)-F_\sigma(\beta^*)]
+
n^2\alpha_K
\E_{\bar{\mathcal{S}}\times\xi}[F_\sigma(\beta^{K+1})-F_\sigma(\beta^*)]
+
n\tau
\E_{\bar{\mathcal{S}}\times\xi}[D(\beta^{K+1},\beta^*)]
\\
&\le
(n^2-\tau^2)\alpha_1
(F_\sigma(\beta^1)-F_\sigma(\beta^*))
+
n\tau D(\beta^1,\beta^*)
+
\tau^2
\left(
\sum_{k=1}^K\alpha_k^2
\right)
\chi^2,
\end{align*}
where
$\vartheta_k := n^2\alpha_{k-1} - (n^2-\tau^2)\alpha_k.$
Since $\alpha_k$ is nonincreasing, $\vartheta_k\ge\tau^2\alpha_k$.
Using
\[
\sum_{k=2}^K\vartheta_k
\ge
\tau^2(\sqrt{K+2}-2),
\qquad
\sum_{k=1}^K\alpha_k^2
\le
\frac{1}{4}\log\!\left(K+1\right),
\]
and dropping nonnegative terms on the left yields
\[
\E_{\bar{\mathcal{S}}\times\xi}[F_\sigma(\bar\beta^K)]-F_\sigma(\beta^*)
\le
\frac{
(n^2-\tau^2)\alpha_1(F_\sigma(\beta^1)-F_\sigma(\beta^*))
+
n\tau D(\beta^1,\beta^*)
+
\frac{1}{4}\tau^2\log(K+1)\chi^2
}{
\tau^2(\sqrt{K+2}-2)
}.
\]
Using $D(u,v)=\tfrac{1}{2}\|u-v\|_\nu^2$ and
$\alpha_1 = 1/(2\sqrt{2})$ yields the stated bound.\hfill\Halmos
\end{proof}

\begin{proof}{Proof of Theorem~\ref{thm:psbcd-convergencerate}.}
The result follows directly from Theorem~\ref{thm:psbcd-general}
by specializing to the stagewise partition with
$n=\tau=T$, under which all stage blocks are updated simultaneously
at every iteration.
\hfill\Halmos
\end{proof}

\section{Total Computational Complexity}
\label{sec:EC-complexity}

This section proves Proposition~\ref{prop:complexity-comparison} by
translating the convergence guarantee of
Theorem~\ref{thm:psbcd-convergencerate} into explicit total
computational complexity bounds.
Our analysis is based on the general PS-BCD framework described
in~\S\ref{sec:EC-PSBCD}, where the coefficient vector
$\beta\in\mathbb{R}^{BT}$ is partitioned into $n$ blocks and, at each
iteration, a random subset of $\tau$ blocks is updated.
The two algorithms compared in Proposition~\ref{prop:complexity-comparison}
correspond to two special cases of this general scheme.
When $n=\tau=1$, the entire coefficient vector is treated as a single block
and updated jointly, recovering the classical stochastic gradient descent
(SGD).
When $n=\tau=T$, the blocks correspond to stagewise coefficient vectors
$\beta_t$ and all $T$ blocks are updated simultaneously, yielding
Algorithm~\ref{alg:SBCD} (PS-BCD).
We therefore analyze the total computational complexity under these two
update regimes.

Before deriving the total computational costs, we establish a bound on
the stochastic term $\chi^2$ appearing in the convergence rate.
The following lemma provides a bound for both WTCA and PO that is
independent of the horizon length $T$, ensuring that any $T$-dependence
in the iteration complexity arises solely from temporal coupling in the
objective rather than from stochastic sampling variability.

\begin{lemma}
\label{lem:dual-noise-weighted}
Let $F_\sigma$ denote the smoothed function of the form
\[
F_\sigma(\beta)
=
\E_\xi\!\left[
\sum_{j=1}^m f_{j,\sigma}(\beta,\sx)
\right],
\]
where each component $f_{j,\sigma}$ is a log-sum-exp of affine functions:
\begin{equation}\label{eq:log-sum-exp-form}
f_{j,\sigma}(\beta,\sx)
=
c_j\,\sigma
\log\!\sum_{u\in\U_j}
\exp\!\left(
\frac{a_{j,\sx}(u)^\top\beta + b_{j,\sx}(u)}{\sigma}
\right).
\end{equation}
The exact forms for WTCA and PO are given
in~\S\ref{sec:seqBCD}.
Let $\beta=(\beta_{(1)},\ldots,\beta_{(n)})$ be a block partition.
Let $L_{ji}$ denote the block-$i$ Lipschitz constant of the deterministic
function $\beta \mapsto \E[f_{j,\sigma}(\beta,\sx)]$.
For each block $i$, define
$J_i
:=
\bigl\{j : \text{component } f_{j,\sigma} \text{ depends on block } i\bigr\},$
and 
$U_i := \sum_{j\in J_i} c_j.$
Let $\chi^2 := \sup_{\beta\in\Omega}\E_\xi\!\left[\|\delta\|_{\nu,*}^2\mid\beta\right]$,
where $\delta := G(\beta,\sx) - \nabla_\beta F_\sigma(\beta)$.
Then
\begin{equation}
\label{eq:dual-noise-weighted-bound}
\chi^2
\le
\sigma
\sum_{i=1}^n
\nu_i^{-1}\,U_i\sum_{j\in J_i} L_{ji}.
\end{equation}
\end{lemma}

\begin{proof}{Proof.}
Since the stochastic oracle is unbiased,
$\E_\xi[G(\beta,\sx)\mid\beta] = \nabla_\beta F_\sigma(\beta)$, and the
standard variance bound
$\E\|X-\E[X]\|^2 \le \E\|X\|^2$ applied in the $\|\cdot\|_{\nu,*}$ norm gives
\begin{equation}\label{eq:variance-decomp}
\E_\xi\!\left[\|\delta\|_{\nu,*}^2 \mid \beta\right]
\le
\E_\xi\!\left[\|G(\beta,\sx)\|_{\nu,*}^2 \mid \beta\right]
=
\sum_{i=1}^n
\nu_i^{-1}
\E_\xi\!\left[
\|G_{(i)}(\beta,\sx)\|_{i,*}^2
\mid \beta
\right],
\end{equation}
where the equality follows from the definition of the weighted dual norm.

Fix block $i$. Since $F_\sigma$ decomposes into components,
\begin{equation}\label{eq:block-gradient}
G_{(i)}(\beta,\sx)
=
\sum_{j\in J_i}
\nabla_{\beta_{(i)}} f_{j,\sigma}(\beta,\sx).
\end{equation}
Differentiating~\eqref{eq:log-sum-exp-form} and restricting to block $i$ gives
\[
\nabla_{\beta_{(i)}} f_{j,\sigma}(\beta,\sx)
=
c_j
\sum_{u\in\U_j}
p_{j,u}(\beta,\sx)\,a_{j,\sx}^{(i)}(u),
\]
where $a_{j,\sx}^{(i)}(u)$ is the block-$i$ component of the affine
coefficient vector and $p_{j,u}(\beta,\sx)$ is the weight
in~\eqref{eq:p}.
Since the weights $p_{j,u}$ sum to one, we get
\[
\|\nabla_{\beta_{(i)}} f_{j,\sigma}(\beta,\sx)\|_{i,*}
\le
c_j\max_{u\in\U_j}\|a_{j,\sx}^{(i)}(u)\|_{i,*}.
\]
Let $A_{j,\sx}^{(i)}$ denote the matrix whose rows are $a_{j,\sx}^{(i)}(u)^\top$,
$u\in\U_j$. Since for any matrix, the row norms are bounded by the spectral norm
(i.e.\ $\|a_{j,\sx}^{(i)}(u)\|_2\le\|A_{j,\sx}^{(i)}\|$ for each row),
\begin{equation}
\label{eq:block-grad-bound}
\|\nabla_{\beta_{(i)}} f_{j,\sigma}(\beta,\sx)\|_{i,*}^2
\le
c_j^2\|A_{j,\sx}^{(i)}\|^2.
\end{equation}
Applying the Cauchy--Schwarz inequality to the sum
in~\eqref{eq:block-gradient},
\begin{equation}\label{eq:Cauchy-Schwarz}
\|G_{(i)}(\beta,\sx)\|_{i,*}^2
=
\left\|
\sum_{j\in J_i}\sqrt{c_j}\cdot
\frac{\nabla_{\beta_{(i)}} f_{j,\sigma}(\beta,\sx)}{\sqrt{c_j}}
\right\|_{i,*}^2
\le
\left(\sum_{j\in J_i}c_j\right)
\sum_{j\in J_i}
\frac{1}{c_j}
\|\nabla_{\beta_{(i)}} f_{j,\sigma}(\beta,\sx)\|_{i,*}^2.
\end{equation}
Substituting~\eqref{eq:block-grad-bound} into~\eqref{eq:Cauchy-Schwarz}
and using $U_i = \sum_{j\in J_i}c_j$, we obtain
\[
\|G_{(i)}(\beta,\sx)\|_{i,*}^2
\le
U_i\sum_{j\in J_i}c_j\|A_{j,\sx}^{(i)}\|^2.
\]
Taking expectations over $\xi$ and using the identity
$L_{ji} = \dfrac{c_j}{\sigma}\E_\xi[\|A_{j,\sx}^{(i)}\|^2]$,
which follows from lemmas~\ref{lem:WTCA-smooth}
and~\ref{lem:PO-smooth}, gives
\[
\E_\xi\!\left[
\|G_{(i)}(\beta,\sx)\|_{i,*}^2 \mid\beta
\right]
\le
U_i\sum_{j\in J_i}c_j\E_\xi\!\left[\|A_{j,\sx}^{(i)}\|^2\right]
=
\sigma\,U_i\sum_{j\in J_i}L_{ji}.
\]
Substituting into~\eqref{eq:variance-decomp} and taking the supremum
over $\beta\in\Omega$ completes the proof.
\hfill\Halmos
\end{proof}

The following corollary specializes Lemma~\ref{lem:dual-noise-weighted} to
WTCA and PO under the two update regimes considered in
Proposition~\ref{prop:complexity-comparison}, specifically $n=\tau=1$ (SGD) and $n=\tau=T$ (PS-BCD).

\begin{corollary}
\label{cor:dual-noise-specializations}
Under the assumptions of Lemma~\ref{lem:dual-noise-weighted}, the
dual-weighted oracle variance $\chi^2$ satisfies
\[
\chi_{\mathrm{WTCA}}^2
\le
\frac{\sigma}{1-\gamma},
\qquad \text{ and }\qquad
\chi_{\mathrm{PO}}^2
\le
\sigma,
\]
under both joint updates ($n=\tau=1$) and stage-wise updates ($n=\tau=T$).
\end{corollary}

\begin{proof}{Proof.}
We specialize Lemma~\ref{lem:dual-noise-weighted} to each formulation
and regime.

\paragraph{(i) Joint updates ($n=\tau=1$).}
With a single block, $J_1=\{1,\ldots,m\}$, $U_1=\sum_{j=1}^mc_j$, and
$\nu_1 = \sum_{j=1}^m L_{j1}$ (the global Lipschitz constant of $F_\sigma$
under single-block ESO).
Lemma~\ref{lem:dual-noise-weighted} gives
\begin{equation}
\label{eq:cor-proof-n1-start}
\chi^2
\le
\sigma\,\nu_1^{-1}\,U_1\sum_{j=1}^mL_{j1}
=
\sigma\,U_1,
\end{equation}
where the equality uses $\nu_1 = \sum_{j=1}^m L_{j1}$.

For WTCA, components are indexed by stages $j=t\in\tset$
with weights $c_j=\gamma^t$, so
$U_1 = \sum_{t=0}^{T-1}\gamma^t = (1-\gamma^T)/(1-\gamma) \le 1/(1-\gamma),$ since $\gamma\in (0,1)$,
and hence $\chi_{\mathrm{WTCA}}^2\le\sigma/(1-\gamma)$.

For PO, there is a single component ($m=1$) with weight $c_1=1$,
so $U_1=1$ and $\chi_{\mathrm{PO}}^2\le\sigma$.

\paragraph{(ii) Stage-wise updates ($n=\tau=T$).}
For WTCA under the stage-wise partition, each block $i$ appears in at
most two adjacent stage components $t=i-1$ and $t=i$, so
$J_i \subseteq \{i-1,i\}$ and
$U_i = \sum_{t\in J_i}\gamma^t \le \gamma^{i-1}+\gamma^i$,
with the convention $\gamma^{-1}\equiv0$.
From~\eqref{eq:WTCA-nu} with $\kappa_t=2$ and $n=\tau=T$,
$\nu_i^{\mathrm{WTCA}} = 2\sum_{t\in J_i}L_{ti}^{\mathrm{WTCA}}.$
Substituting into Lemma~\ref{lem:dual-noise-weighted},
the sum $\sum_{t\in J_i}L_{ti}^{\mathrm{WTCA}}$ cancels and gives
\begin{equation}
\label{eq:cor-proof-wtca-nT-start}
\chi_{\mathrm{WTCA}}^2
\le
\sigma\sum_{i=0}^{T-1}
\frac{\gamma^{i-1}+\gamma^i}{2}.
\end{equation}
Since $\sum_{i=0}^{T-1}\gamma^i \le 1/(1-\gamma)$ and $\sum_{i=0}^{T-1}\gamma^{i-1} = \sum_{k=0}^{T-2}\gamma^k \le 1/(1-\gamma)$ (using the convention $\gamma^{-1}\equiv 0$ at $i=0$), we conclude $\chi_{\mathrm{WTCA}}^2 \le \sigma/(1-\gamma)$.

For PO under stage-wise partition, there is still a single component ($m=1$) with $c_1=1$, so $J_i=\{1\}$ and $U_i=1$ for all $i$. From~\eqref{eq:PO-nu} with $\kappa_1=T$ and $n=\tau=T$, $\nu_i^{\mathrm{PO}} = TL_{1i}^{\mathrm{PO}}$, so $(\nu_i^{\mathrm{PO}})^{-1}L_{1i}^{\mathrm{PO}} = 1/T$. Lemma~\ref{lem:dual-noise-weighted} therefore gives $\chi_{\mathrm{PO}}^2 \le \sigma\sum_{i=1}^T\frac{1}{T} = \sigma.$

Combining both regimes completes the proof.
\hfill\Halmos
\end{proof}

\subsection{Iteration Complexity in the Two Regimes}\label{sec:EC-K}

As discussed in the main paper, the total computational cost of a stochastic first-order method can be written as
\[
\mathrm{TotalCost} \;=\; K \times C,
\]
where $K$ is an upper bound on the number of iterations required to obtain an $\epsilon$-optimal solution (in expectation), and $C$ is the per-iteration cost. The latter depends on the temporal coupling width $\kappa(F_\sigma)$, the number of basis functions $B$, and the number of component functions $m$ in~\eqref{eq:general-formulation-smooth}.

To analyze the computational complexity of both PS-BCD and SGD in a unified manner, we consider their behavior under the general GPS-BCD algorithm described in~\S\ref{sec:EC-PSBCD}. In this framework, the method operates with $n$ blocks and $\tau$-nice sampling. The corresponding iteration complexity and per-iteration cost are denoted by
$K(n,\tau)$ and $C(n,\tau),$ respectively.

We will evaluate two specific regimes: (i) $n=\tau=1$, which corresponds to standard SGD, and (ii) $n=\tau=T$, which corresponds to full PS-BCD. 
This allows a direct comparison of the computational complexity of SGD and PS-BCD through the unified quantities $K(n,\tau)$ and $C(n,\tau)$.

Throughout this section, we set $\beta^1=0$ without loss of generality and suppress logarithmic factors using the $\tilde{\mathcal O}(\cdot)$
notation.

\medskip

To determine $K(n,\tau)$, we start from
Theorem~\ref{thm:psbcd-general}. When $n=\tau$, the bound in~\eqref{eq:psbcd-general-bound} simplifies to\looseness=-1
\begin{equation}
\label{eq:iteration-structure-EC}
\E_{\bar{\mathcal{S}}\times\xi}\!\left[F_\sigma(\bar\beta^K)\right]-F_\sigma(\beta^*)
=
\tilde{\mathcal O}\!\left(
\frac{1}{\sqrt{K}}
\Bigl(
\|\beta^*\|_\nu^2
+
\chi^2
\Bigr)
\right),
\end{equation}
where $\|\beta^*\|_\nu^2=\sum_{i=1}^n\nu_i\|\beta^*_{(i)}\|_i^2$ and
$\chi^2 := \sup_{\beta\in\Omega}\E_\xi[\|\delta\|_{\nu,*}^2\mid\beta]$
with $\delta = G(\beta,\sx)-\nabla_\beta F_\sigma(\beta)$.
Solving for $\E_{\bar{\mathcal{S}}\times\xi}[F_\sigma(\bar\beta^K)]-F_\sigma(\beta^*)\le\epsilon$ gives
\begin{equation}
\label{eq:K-general-EC}
K(n,n)
=
\tilde{\mathcal O}\!\left(
\frac{1}{\epsilon^2}
\Bigl(
\|\beta^*\|_\nu^2+\chi^2
\Bigr)^2
\right).
\end{equation}
The iteration complexity is thus governed by two quantities: a
deterministic term $\|\beta^*\|_\nu^2$, which encodes temporal coupling
through the curvature weights $\nu_i$, and a stochastic term $\chi^2$,
which captures the dual-weighted gradient variance.
We bound each quantity under both formulations and both update regimes.

Since the feasible set $\Omega$ is compact, there exist constants
$R_i>0$, independent of $T$ and $\sigma$, such that
$\|\beta_{(i)}^*\|_i^2\le R_i^2$ for all $i$. Therefore,
\begin{equation}
\label{eq:beta-nu-bound-EC}
\|\beta^*\|_\nu^2
=
\sum_{i=1}^n\nu_i\|\beta^*_{(i)}\|_i^2
\le
\Bigl(\max_{1\le i\le n}R_i^2\Bigr)\sum_{i=1}^n\nu_i
=
\mathcal O\!\left(\sum_{i=1}^n\nu_i\right).
\end{equation}
Thus, bounding $\|\beta^*\|_\nu^2$ reduces to bounding
$\sum_{i=1}^n\nu_i$. The stochastic term $\chi^2$ is controlled by
Corollary~\ref{cor:dual-noise-specializations}.

\paragraph{Regime (i): joint updates, i.e., SGD ($n=\tau=1$).}
In this regime, $\beta\in\mathbb{R}^{BT}$ is treated as a single block,
so $n=1$ and the ESO reduces to standard single-block smoothness:
$\nu_1 = \sum_{j=1}^m L_{j1}$.

For WTCA, the $m$ component functions correspond to stages $t=0,\ldots,T-1$.  
Thus, when $n=\tau=1$, the single-block curvature weight reduces to
\[
\nu_1^{\mathrm{WTCA}}
= 
\sum_{j=1}^m L_j^{\mathrm{WTCA}}
=
\sum_{t=0}^{T-1} L_t^{\mathrm{WTCA}}.
\]
Applying~\eqref{eq:WTCA-nu} with $n=\tau=1$ and $|\U_t|=1$ for all $t$ yields
\[
\nu_1^{\mathrm{WTCA}}
\le
\frac{1+\gamma^2}{\sigma}
\sum_{t=0}^{T-1}\gamma^t\,|\Endo_t||\action_t|
=
\mathcal{O}\!\left(\frac{1}{\sigma(1-\gamma)}\right),
\]
where the last step uses $\sum_{t=0}^{T-1}\gamma^t \le 1/(1-\gamma)$ and the assumption that $|\Endo_t||\action_t|$ is uniformly bounded in $t$.
Hence, by~\eqref{eq:beta-nu-bound-EC}, $\|\beta^*\|_\nu^2 = \mathcal O(1/(\sigma(1-\gamma)))$.
Combined with $\chi_{\mathrm{WTCA}}^2 = \mathcal O(\sigma/(1-\gamma))$
from Corollary~\ref{cor:dual-noise-specializations}, and substituting
into~\eqref{eq:K-general-EC}, we obtain
\[
K_{\mathrm{WTCA}}(1,1)
=
\tilde{\mathcal O}\!\left(
\frac{1}{\epsilon^2}
\left(
\frac{1}{\sigma(1-\gamma)}
+
\frac{\sigma}{1-\gamma}
\right)^2
\right).
\]

For PO with $n=\tau=1$, applying~\eqref{eq:PO-nu-general} gives
\[
\nu_1^{\mathrm{PO}}
\le
\frac{4\gamma^2|\action|}{\sigma(1-\gamma^2)}
=
\mathcal{O}\!\left(\frac{|\action|}{\sigma(1-\gamma)}\right).
\]
Combined with $\chi_{\mathrm{PO}}^2=\mathcal{O}(\sigma)$ from Corollary~\ref{cor:dual-noise-specializations}, we get
\[
K_{\mathrm{PO}}(1,1)
=
\tilde{\mathcal{O}}\!\left(
\frac{1}{\epsilon^2}
\left(
\frac{|\action|}{\sigma(1-\gamma)}
+\sigma
\right)^2
\right).
\]

\paragraph{Regime (ii): stage-wise updates, i.e., PS-BCD ($n=\tau=T$).}
In this regime, $\beta=(\beta_0,\ldots,\beta_{T-1})$ is partitioned into $T$ stage blocks updated in parallel.

For WTCA with $n=\tau=T$, each block index $i$ corresponds to a stage-wise block, so that $\beta_{(i)}=\beta_t$ with $t=i$.  
Applying~\eqref{eq:WTCA-nu} in this setting and using $\kappa_t=2$ yields
\[
\nu_i^{\mathrm{WTCA}}
\;\le\;
\frac{2(1+\gamma^2)}{\sigma}\,(\gamma^{i-1}+\gamma^i),
\qquad
\text{with the convention }\gamma^{-1}=0.
\]
Summing over all $T$ stage-aligned blocks then gives
\[
\sum_{i=0}^{T-1}\nu_i^{\mathrm{WTCA}}
\;\le\;
\frac{2(1+\gamma^2)}{\sigma}
\sum_{i=0}^{T-1}(\gamma^{i-1}+\gamma^i)
\;\le\;
\frac{4(1+\gamma^2)}{\sigma(1-\gamma)}
=
\mathcal{O}\!\left(\frac{1}{\sigma(1-\gamma)}\right).
\]
By~\eqref{eq:beta-nu-bound-EC}, $\|\beta^*\|_\nu^2=\mathcal{O}(1/(\sigma(1-\gamma)))$. Together with $\chi_{\mathrm{WTCA}}^2=\mathcal{O}(\sigma/(1-\gamma))$,
\[
K_{\mathrm{WTCA}}(T,T)
=
\tilde{\mathcal{O}}\!\left(
\frac{1}{\epsilon^2}
\left(
\frac{1}{\sigma(1-\gamma)}
+
\frac{\sigma}{1-\gamma}
\right)^2
\right).
\]

For PO, applying~\eqref{eq:PO-nu} with $n=\tau=T$ gives $\nu_i^{\mathrm{PO}}\le 4T\gamma^{2i}|\action|/\sigma$, so
\[
\sum_{i=0}^{T-1}\nu_i^{\mathrm{PO}}
\le
\frac{4T|\action|}{\sigma}
\sum_{i=0}^{T-1}\gamma^{2i}
\le
\frac{4T|\action|}{\sigma(1-\gamma^2)}
=
\mathcal{O}\!\left(
\frac{T|\action|}{\sigma(1-\gamma)}
\right).
\]
By~\eqref{eq:beta-nu-bound-EC}, $\|\beta^*\|_\nu^2=\mathcal{O}(T|\action|/(\sigma(1-\gamma)))$. Together with $\chi_{\mathrm{PO}}^2=\mathcal{O}(\sigma)$,
\[
K_{\mathrm{PO}}(T,T)
=
\tilde{\mathcal{O}}\!\left(
\frac{1}{\epsilon^2}
\left(
\frac{T|\action|}{\sigma(1-\gamma)}
+\sigma
\right)^2
\right).
\]

\subsection{Per-Iteration Computational Cost}
\label{sec:EC-cost}

We now derive the per-iteration cost $C(n,\tau)$ for the two special 
cases $n=\tau=1$ (SGD) and $n=\tau=T$ (PS-BCD).
Recall from~\eqref{eq:stochastic-gradient}--\eqref{eq:stochastic-gradient-block}
that the stochastic gradient and its block-$i$ component are \looseness=-1
\[
G(\beta,\sx)
=
\sum_{j=1}^m\sum_{u\in\U_j}p_{j,u}(\beta,\sx)\,a_{j,\sx}(u),
\qquad
G_{(i)}(\beta,\sx)
=
\sum_{j:\,i\in \mathcal{I}_j}\sum_{u\in\U_j}
p_{j,u}(\beta,\sx)\,a_{j,\sx}^{(i)}(u),
\]
where $p_{j,u}(\beta,\sx)$ is the weight~\eqref{eq:p} and
$a_{j,\sx}^{(i)}(u)$ is the block-$i$ portion of $a_{j,\sx}(u)$.
Since each affine vector $a_{j,\sx}(u)$ has nonzero entries in at most
$\kappa_j$ blocks of $\beta$, each of size $B$, computing the
stochastic gradient involves three steps per component $j$:

\begin{enumerate}

\item[(i)] \emph{Affine coefficient computation.}
Evaluating $a_{j,\sx}(u)^\top\beta$ for all $u\in\U_j$ requires
reading $\kappa_j$ blocks of size $B$, costing
$\mathcal O(|\U_j|\,\kappa_j B)$ in total for component~$j$.

\item[(ii)] \emph{$p_{j,u}$ weights computation.}
Given the affine coefficents, forming $p_{j,u}(\beta,\sx)$ via 
normalisation costs $\mathcal O(|\U_j|)$.

\item[(iii)] \emph{Gradient aggregation.}
Assembling the block-$i$ gradient contribution from component $j$
requires reading the block-$i$ portions $a_{j,\sx}^{(i)}(u)$ of
dimension $B$ for all $u\in\U_j$, costing $\mathcal O(|\U_j|\,B)$
per component $j$ per block $i\in \mathcal{I}_j$.

\end{enumerate}

Steps~(i)--(ii) are dominated by step~(i), which costs
$\mathcal O(|\U_j|\,\kappa_j B)$ per component $j$.
Note that the weights $p_{j,u}$ for component $j$ require the
full affine coefficient $a_{j,\sx}(u)^\top\beta$, which reads all $\kappa_j$
blocks of $\beta$; no block can form its contribution to $G_{(i)}$
until these weights are available. The total cost then depends on 
whether affine coefficients and weights are shared across blocks (SGD) 
or computed independently per block (PS-BCD).

\noindent \textbf{(i) $\mathbf{n=\tau=1}$: SGD.}
SGD computes the full gradient $G(\beta,\sx)$ by processing all $m$ 
components sequentially. For each component $j$, the affine coefficients 
and weights $p_{j,u}$ are computed once at cost 
$\mathcal{O}(|\U_j|\kappa_j B)$ and the result is aggregated 
across all $\kappa_j$ blocks simultaneously. The total cost is 
therefore
\begin{equation}\label{eq:cost-prox-general}
C(1,1) = \mathcal O\!\left(\sum_{j=1}^m |\U_j|\,\kappa_j B\right).
\end{equation}

\noindent \textbf{(ii) $\mathbf{n=\tau=T}$: PS-BCD.}
PS-BCD computes all $T$ block gradients $G_{(i)}(\beta,\sx)$ in 
parallel. Each block $i$ independently computes the affine coefficients 
and weights for every component $j\in\mathcal{I}_i$, incurring cost 
$\sum_{j:\,i\in \mathcal{I}_j}\mathcal O(|\U_j|\,\kappa_j B)$.
The effective per-iteration cost under parallel execution is therefore
\begin{equation}\label{eq:cost-psbcd-general}
C(T,T)
=
\mathcal O\!\left(
\max_{1\le i\le n}\sum_{j:\,i\in \mathcal{I}_j}|\U_j|\,\kappa_j B
\right).
\end{equation}

We now specialize~\eqref{eq:cost-prox-general}
and~\eqref{eq:cost-psbcd-general} to WTCA and PO.

\noindent \textbf{WTCA.}
WTCA has $m=T$ stagewise components each coupling two adjacent
blocks ($\kappa_t=2$, $\mathcal{I}_t=C_t=\{t,t+1\}$) over a finite 
index set $\U_t=\Endo_t\times\action_t$ of size $\mathcal O(1)$ 
uniformly in $t$. For SGD, each component $t$ contributes 
$|\U_t|\cdot\kappa_t\cdot B = \mathcal{O}(1)\cdot 2\cdot B = 
\mathcal{O}(B)$ to the sum in~\eqref{eq:cost-prox-general}. 
For PS-BCD, each block $i$ appears in at most two components 
($t=i-1$ and $t=i$), each contributing $\mathcal{O}(1)\cdot 2\cdot B
=\mathcal{O}(B)$, so the max in~\eqref{eq:cost-psbcd-general} is 
$\mathcal{O}(B)$ uniformly in $i$. Substituting:
\[
C_{\mathrm{WTCA}}(1,1)
=
\mathcal O\!\left(
\sum_{t=0}^{T-1}\mathcal{O}(B)
\right)
=
\mathcal O(TB),
\qquad
C_{\mathrm{WTCA}}(T,T)
=
\mathcal O\!\left(\max_i\,\mathcal{O}(B)
\right)
=
\mathcal O(B).
\]

\noindent \textbf{PO.}
PO has a single component ($m=1$) over the full action set
$\U_1=\action$, coupling all $T$ stage blocks ($\kappa_1=T$,
$\mathcal{I}_1=C_1=\{0,1,\ldots,T-1\}$).
For SGD, the single component contributes $|\action|\cdot T\cdot B$ 
to~\eqref{eq:cost-prox-general}. For PS-BCD, every block $i$ 
belongs to this single component, so the sum in the max 
of~\eqref{eq:cost-psbcd-general} equals $|\action|\cdot T\cdot B$ 
for every $i$. Substituting:
\[
C_{\mathrm{PO}}(1,1)
=
\mathcal O\!\left(|\action|\,TB\right),
\qquad
C_{\mathrm{PO}}(T,T)
=
\mathcal O\!\left(|\action|\,TB\right).
\]
Stage-block parallelism does not reduce the per-iteration cost for PO:
since all $T$ blocks appear in the single component,
every block $i$ must compute the full weights
$p_{1,u}(\beta,\sx)$ by reading all $T$ blocks,
and the $\max_i$ does not improve on the sequential sum.

\subsection{Total Complexity and Proof of
  Proposition~\ref{prop:complexity-comparison}}
\label{sec:EC-totalcost-proof}

We now combine the iteration complexity bounds from~\S\ref{sec:EC-K}
with the per-iteration costs from~\S\ref{sec:EC-cost} to obtain the
total complexity expressions in
Proposition~\ref{prop:complexity-comparison}.
The joint-update regime ($n=\tau=1$) corresponds to SGD and the
stage-wise regime ($n=\tau=T$) corresponds to PS-BCD.

\begin{proof}{Proof of Proposition~\ref{prop:complexity-comparison}.}

The total computational cost is
$\mathrm{TotalCost} = K(n,\tau)\cdot C(n,\tau)$,
where $K(n,\tau)$ is the iteration complexity from~\S\ref{sec:EC-K}
and $C(n,\tau)$ is the per-iteration cost from~\S\ref{sec:EC-cost}.

\medskip
\noindent\emph{(i) Joint updates, i.e., SGD ($n=\tau=1$).}
From~\S\ref{sec:EC-K},
\[
K_{\mathrm{WTCA}}(1,1)
=
\tilde{\mathcal O}\!\left(
\frac{1}{\epsilon^2}
\left(
\frac{1}{\sigma(1-\gamma)}
+
\frac{\sigma}{1-\gamma}
\right)^2
\right),
\qquad
K_{\mathrm{PO}}(1,1)
=
\tilde{\mathcal O}\!\left(
\frac{1}{\epsilon^2}
\left(
\frac{|\action|}{\sigma(1-\gamma)}
+
\sigma
\right)^2
\right).
\]
From~\S\ref{sec:EC-cost},
\[
C_{\mathrm{WTCA}}(1,1)
=\mathcal O(TB),
\qquad
C_{\mathrm{PO}}(1,1)
=\mathcal O(|\action|\,TB).
\]
Multiplying yields
\[
\mathrm{TotalCost}_{\mathrm{WTCA}}(1,1)
=
\tilde{\mathcal O}\!\left(
\frac{TB}{\epsilon^2}
\left(
\frac{1}{\sigma(1-\gamma)}
+
\frac{\sigma}{1-\gamma}
\right)^2
\right),
\quad
\mathrm{TotalCost}_{\mathrm{PO}}(1,1)
=
\tilde{\mathcal O}\!\left(
\frac{|\action|\,TB}{\epsilon^2}
\left(
\frac{|\action|}{\sigma(1-\gamma)}
+
\sigma
\right)^2
\right).
\]
For $\sigma\in(0,1)$, the term $1/(\sigma(1-\gamma))$ dominates
$\sigma/(1-\gamma)$ in the WTCA bound.
For PO, $|\action|/(\sigma(1-\gamma))$ dominates $\sigma$ in the regime
$|\action|\gg\sigma^2(1-\gamma)$, which holds whenever the action space
is not negligibly small relative to the smoothing parameter.
Retaining the leading-order terms gives
\[
\mathrm{TotalCost}_{\mathrm{WTCA}}(1,1)
=
\tilde{\mathcal O}\!\left(
\frac{BT}{\sigma^2\epsilon^2(1-\gamma)^2}
\right),
\qquad
\mathrm{TotalCost}_{\mathrm{PO}}(1,1)
=
\tilde{\mathcal O}\!\left(
\frac{|\action|^3\,BT}{\sigma^2\epsilon^2(1-\gamma)^2}
\right).
\]

\medskip
\noindent\emph{(ii) Stage-wise parallel updates, i.e., PS-BCD ($n=\tau=T$).}
From~\S\ref{sec:EC-K},
\[
K_{\mathrm{WTCA}}(T,T)
=
\tilde{\mathcal O}\!\left(
\frac{1}{\epsilon^2}
\left(
\frac{1}{\sigma(1-\gamma)}
+
\frac{\sigma}{1-\gamma}
\right)^2
\right),
\qquad
K_{\mathrm{PO}}(T,T)
=
\tilde{\mathcal O}\!\left(
\frac{1}{\epsilon^2}
\left(
\frac{T|\action|}{\sigma(1-\gamma)}
+
\sigma
\right)^2
\right).
\]
From~\S\ref{sec:EC-cost},
\[
C_{\mathrm{WTCA}}(T,T)
=\mathcal O(B),
\qquad
C_{\mathrm{PO}}(T,T)
=\mathcal O(|\action|\,TB).
\]
Multiplying yields
\[
\mathrm{TotalCost}_{\mathrm{WTCA}}(T,T)
=
\tilde{\mathcal O}\!\left(
\frac{B}{\epsilon^2}
\left(
\frac{1}{\sigma(1-\gamma)}
+
\frac{\sigma}{1-\gamma}
\right)^2
\right),
\quad
\mathrm{TotalCost}_{\mathrm{PO}}(T,T)
=
\tilde{\mathcal O}\!\left(
\frac{|\action|\,TB}{\epsilon^2}
\left(
\frac{T|\action|}{\sigma(1-\gamma)}
+
\sigma
\right)^2
\right).
\]
Retaining the leading-order terms under the same dominance conditions
as in part~(i) gives
\[
\mathrm{TotalCost}_{\mathrm{WTCA}}(T,T)
=
\tilde{\mathcal O}\!\left(
\frac{B}{\sigma^2\epsilon^2(1-\gamma)^2}
\right),
\qquad
\mathrm{TotalCost}_{\mathrm{PO}}(T,T)
=
\tilde{\mathcal O}\!\left(
\frac{|\action|^3\,BT^3}{\sigma^2\epsilon^2(1-\gamma)^2}
\right).
\]
These four expressions coincide with the entries in the table of
Proposition~\ref{prop:complexity-comparison} and complete the proof.
\hfill\Halmos
\end{proof}

\section{Asymptotic Approximation under Random Basis Constructions}
\label{sec:EC-asymptotics}

This section proves Proposition~\ref{prop:asymptotic-consistency-random}. We work under Assumption~\ref{assump:random-basis functions} in the paper and use the random basis functions described in \S\ref{sec:model-selection}.

\begin{proof}{Proof of  Proposition~\ref{prop:asymptotic-consistency-random}.}

Under Assumption~\ref{assump:random-basis functions}, Proposition EC.1 of \citet{pakiman2024} implies that for the ALP formulation with $B$ random basis functions, the resulting upper bound $U_{\mathrm{ALP},B}$ satisfies the high-probability error bound:
\begin{equation}
\label{eq:alp-rate-appendix}
U_{\mathrm{ALP},B} - V^*_0(s_0)
=
\mathcal{O}\!\left(
\frac{1}{(1-\gamma)\sqrt{B}}
\sqrt{\log\!\frac{1}{\delta}}
\right)
\quad
\text{with probability at least } 1-\delta,
\end{equation}
for any $\delta \in (0,1)$, where the probability is with respect to the randomness in the sampled basis functions.
We extend this rate to WTCA and PO using the ordering established in Proposition~\ref{prop:WTCA-comparison}:
\[
0
\le
U_{\mathrm{PO},B} - V^*_0(s_0)
\le
U_{\mathrm{WTCA},B} - V^*_0(s_0)
\le
U_{\mathrm{ALP},B} - V^*_0(s_0).
\]
On the high-probability event in which~\eqref{eq:alp-rate-appendix} holds, it follows immediately that
\[
0
\le
U_{\mathrm{PO},B} - V^*_0(s_0)
\le
U_{\mathrm{WTCA},B} - V^*_0(s_0)
\le
\mathcal{O}\!\left(
\frac{1}{(1-\gamma)\sqrt{B}}
\sqrt{\log\!\frac{1}{\delta}}
\right).
\]
Thus, for any fixed initial state $s_0$, WTCA and PO achieve the same $\mathcal{O}(1/\sqrt{B})$ approximation rate as ALP under random basis functions, with probability at least $1-\delta$.
\hfill\Halmos

\end{proof}

%%%%%%%%%%%%%%%%%%%%%%%%%%
\section{Merchant Ethanol Production}
\label{sec:EC-numerical}
This section provides the full problem formulation, implementation details, and additional results for the merchant ethanol production experiments summarized in \S\ref{scc:ethanol-main}. The main paper describes the operating modes, action space, and exogenous state structure at a high level; here we give the complete mathematical specification.

\subsection{Problem setup}
\label{sc:MEP_Instance}
Consider an ethanol production plant that converts corn and natural gas into ethanol. The operator purchases inputs in wholesale markets and sells ethanol into energy markets. Over a finite planning horizon, the operator maximizes expected total profit by dynamically switching the facility among operating modes in response to uncertain commodity prices. 

The problem is naturally modeled as a finite-horizon MDP using the real options approach~\citep{dixit1994investment,trigeorgis1996real,smith1998valuing,smith1999options,nadarajah2015relaxations,nadarajah2023review}.
Let $T$ denote the planning horizon with stages $t\in\tset$.
At each stage, the plant can be operational ($\mathsf{O}$), mothballed ($\mathsf{M}$), or abandoned ($\mathsf{A}$). The endogenous state is therefore $\se_t\in\Endo_t:=\{\mathsf{O},\mathsf{M},\mathsf{A}\}$.

Let $\mathcal{C}:=\{\mathrm{C}, \mathrm{NG}, \mathrm{E}\}$ denote the set of commodities, where $\mathrm{C}$, $\mathrm{NG}$, and $\mathrm{E}$ represent corn, natural gas, and ethanol, respectively. For each $c\in\mathcal{C}$ and $t,s\in\tset$ with $s\geq t$, we denote as $\sx_{t,s}^{c}$ the futures price observed at stage $t$ for delivery of commodity $c$ at maturity $s$. The spot price at stage $t$ is $\sx_{t,t}^c$. The forward curve for commodity $c$ at stage $t$ is the vector $\mathbf{\sx}_{t}^{c}:=(\sx_{t,t}^{c},\sx_{t,t+1}^{c},\ldots,\sx_{t,T-1}^{c})$. The exogenous state concatenates all three forward curves, i.e., $\mathbf{\sx}_{t}:= (\mathbf{\sx}_{t}^{\mathrm{C}}, \mathbf{\sx}_{t}^{\mathrm{NG}}, \mathbf{\sx}_{t}^{\mathrm{E}})\in \mathcal{W}_t:=\R_{+}^{3\times(T-t)}$.
Unlike the Bermudan setting, the dimension of $\mathcal{W}_t$ increases with $T$, i.e., the forward curve at each stage becomes longer as the horizon $T$ increases.
The full state space is $\mathcal{S}=\bigcup_{t=0}^{T-1}\mathcal{X}_t\times\mathcal{W}_t$.

The feasible action set depends on the plant's current state. When operational ($\mathsf{O}$), the plant can produce at full capacity ($\mathsf{P}$), temporarily suspend production ($\mathsf{S}$), transit to the mothballed state ($\mathsf{M}$), or permanently shut down ($\mathsf{A}$). We intentionally reuse $\mathsf{M}$ and $\mathsf{A}$ for both states and actions, as these actions deterministically transition the plant to the corresponding states. From the mothballed state ($\mathsf{M}$), the feasible actions are to remain mothballed ($\mathsf{M}$), reactivate the plant ($\mathsf{R}$), or abandon it ($\mathsf{A}$). Suspension incurs a higher fixed cost than mothballing but allows faster resumption of production. Once the plant is abandoned, the feasible action set is a singleton that only maintains the abandoned state. By the final stage $T-1$, abandonment is mandatory.
The feasible action sets are thus $\mathcal{A}(\mathsf{O}):=\{\mathsf{A},\mathsf{P},\mathsf{S},\mathsf{M}\}$, $\mathcal{A}(\mathsf{M}):=\{\mathsf{A},\mathsf{M},\mathsf{R}\}$, $\mathcal{A}(\mathsf{A}):=\{\mathsf{A}\}$, and $\mathcal{A}(x_{T-1}):=\{\mathsf{A}\}$ for any $x_{T-1} \in \{\mathsf{O},\mathsf{M}, \mathsf{A}\}$. 

The endogenous state evolves based on the endogenous transition function $h(\se_t,a_t):\mathcal{X}_t\times\mathcal{A}(x_t)\rightarrow\mathcal{X}_{t+1}$, which maps the current operational state $\se_t$ and action $a_t$ to the next operational state $x_{t+1}$.
\[
h(\se_t,a_t)=
\begin{cases}
\mathsf{O}, & \text{if } (\se_t,a_t)\in\{(\mathsf{O},\mathsf{P}),(\mathsf{O},\mathsf{S}),(\mathsf{M},\mathsf{R})\},\\[4pt]
\mathsf{M}, & \text{if } (\se_t,a_t)\in\{(\mathsf{O},\mathsf{M}),(\mathsf{M},\mathsf{M})\},\\[4pt]
\mathsf{A}, & \text{if } a_t=\mathsf{A}.
\end{cases}
\]
Actions $\mathsf{P}$ and $\mathsf{S}$ both preserve the operational state $\mathsf{O}$. Action $\mathsf{M}$ at state $\mathsf{O}$ moves the plant to the mothballed state $\mathsf{M}$. From $x_t=\mathsf{M}$, action $\mathsf{M}$ maintains the mothballed state, while action $\mathsf{R}$ returns the plant to the operational state $\mathsf{O}$.
Action $\mathsf{A}$ transits any other states to the abandoned state.
The endogenous state transitions are summarized in Figure \ref{fig:DecisionTree}.

\begin{figure}[ht!]
\centering
\includegraphics[scale=0.7]{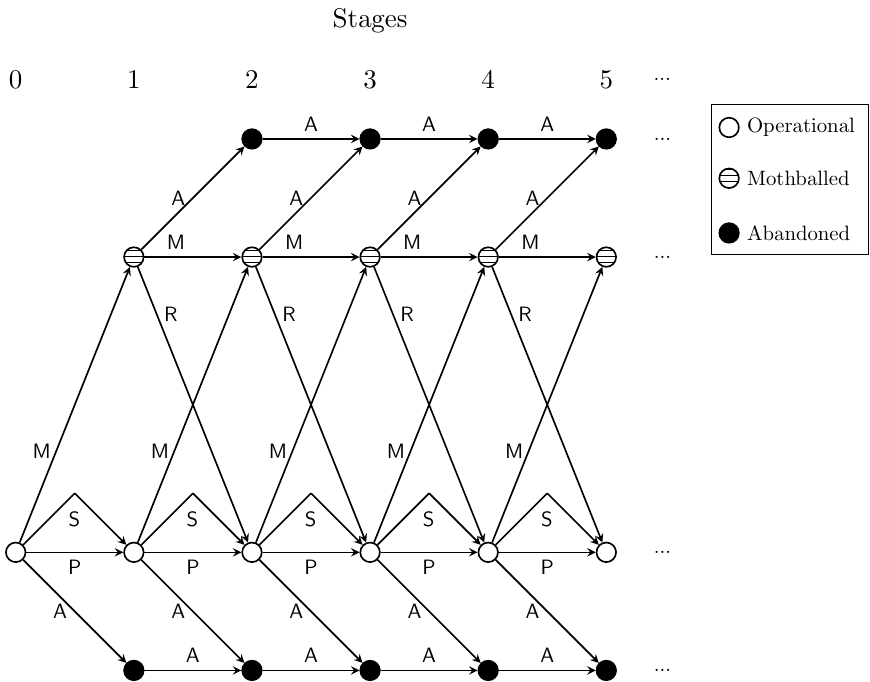}
\caption{Endogenous state transitions in ethanol production~\citep{guthrie2009real,yang2024least,yang2025improved}.}
\label{fig:DecisionTree}
\end{figure}
The dynamics of commodity prices are exogenous to the actions. For commodity $c\in\mathcal{C}$, stage $t$, and maturity $s\geq t$, let $\eta_{t,s,l}^{c}$ be the loading coefficient associated with risk factor $l \in \{1,\ldots,L\}$. The futures price process for $W^{c}_{t,s}$ follows the term structure model
\begin{equation}
\label{eq:MEP_PriceModel}
\frac{dW^{c}_{t,s}}{W^{c}_{t,s}}=\sum_{l=1}^{L}\eta_{t, s,l}^{c}dZ_{l},
\end{equation}
where the $\{Z_{l}\}_{l=1}^{L}$ are independent standard Brownian motions under the risk neutral measure. The term structure model captures $L$ sources of uncertainty to drive price movements and correlated shocks across commodities and delivery maturities.
 
Let $\mu_{\mathrm{C}}$ and $\mu_{\mathrm{NG}}$ denote the quantities of corn and natural gas required to produce one unit of ethanol, respectively.  
At stage~$t$, the instantaneous marginal profit of production is the spread\looseness = -1
\[
\sx_{t,t}^{\mathrm{E}} - \mu_{\mathrm{C}} \sx_{t,t}^{\mathrm{C}} - \mu_{\mathrm{NG}} \sx_{t,t}^{\mathrm{NG}},
\]
which represents the ethanol spot price net of input costs.  
The plant produces at capacity~$Q$ and has a fixed production cost~$\mathsf{C}_{\mathsf{P}}$.  
Suspending production incurs a maintenance cost~$\mathsf{C}_{\mathsf{S}}$, whereas keeping the plant mothballed incurs the smaller cost~$\mathsf{C}_{\mathsf{M}}$, with $\mathsf{C}_{\mathsf{M}} < \mathsf{C}_{\mathsf{S}}$.  
Mode changes also entail costs. Mothballing requires~$\mathsf{I}_{\mathsf{M}}$, reactivation requires~$\mathsf{I}_{\mathsf{R}}$, and permanent abandonment yields a salvage value~$\mathsf{S}_{\mathsf{A}}$. The immediate reward function $r(\se_t,\mathbf{\sx}_t,a_t)$ summarizes these payoffs for each $(\se_t,\mathbf{\sx}_t,a_t)\in\Endo_t\times\Exo_t\times\action(\se_t)$:
\[
r(\se_t,\mathbf{\sx}_t,a_t)=
\begin{cases}
(\sx_{t,t}^{\mathrm{E}}-\mu_{\mathrm{C}}\sx_{t,t}^{\mathrm{C}}-\mu_{\mathrm{NG}}\sx_{t,t}^{\mathrm{NG}})Q-\mathsf{C}_{\mathsf{P}}, & (\se_t,a_t)=(\mathsf{O},\mathsf{P}),\\[3pt]
-\mathsf{C}_{\mathsf{S}}, & (\se_t,a_t)=(\mathsf{O},\mathsf{S}),\\[3pt]
-\mathsf{I}_{\mathsf{M}}, & (\se_t,a_t)=(\mathsf{O},\mathsf{M}),\\[3pt]
-\mathsf{C}_{\mathsf{M}}, & (\se_t,a_t)=(\mathsf{M},\mathsf{M}),\\[3pt]
-\mathsf{I}_{\mathsf{R}}, & (\se_t,a_t)=(\mathsf{M},\mathsf{R}),\\[3pt]
\mathsf{S}_{\mathsf{A}}, & (\se_t,a_t)\in\{(\mathsf{O},\mathsf{A}),(\mathsf{M},\mathsf{A})\},\\[3pt]
0, & (\se_t,a_t)=(\mathsf{A},\mathsf{A}).
\end{cases}
\]

\subsection{Implementation Details}
\label{sc:MEP_Details}
We evaluate WTCA and benchmark methods on two sets of merchant ethanol production instances that differ in planning horizon.  
The first set contains 12 instances, one for each month of 2011, with a 24-month horizon ($T=24$).
For each instance, the initial forward curve $\mathbf{\sx}_0$ consists of corn, natural gas, and ethanol futures prices observed on the first trading day of that month.  
The price model~\eqref{eq:MEP_PriceModel} is calibrated using daily futures data from the Chicago Mercantile Exchange for 2008--2011.  
We use $L=8$ factors, explaining approximately 95\% of the observed price variance.  
Monthly risk-free discount factors $\gamma$ are derived from the one-year U.S. Treasury rate on the initial date of each instance.  
Operational parameters follow \citet{yang2024least,yang2025improved} and are summarized in Table~\ref{Tab:ops_parameters}.

\begin{table}[ht!]
\centering
\caption{Operational parameter settings for merchant ethanol production}
\renewcommand{\arraystretch}{1.3}
\begin{tabular}{cc|cc}
\hline
Parameter & Value (\$\ MM) & Parameter & Value \\ \hline
$\mathsf{I}_{\mathsf{M}}$ & 0.50 & $T$ & 24 months \\
$\mathsf{I}_{\mathsf{R}}$ & 2.50 & $\mu_{\mathrm{C}}$ & 0.36 bushel/gallon \\
$\mathsf{C}_{\mathsf{P}}$ & 2.25 & $\mu_{\mathrm{NG}}$ & 0.035 MMBtu/gallon \\
$\mathsf{C}_{\mathsf{S}}$ & 0.5208 & $Q$ & 8.33 million gallons \\
$\mathsf{C}_{\mathsf{M}}$ & 0.02917 & $\mathsf{S}_{\mathsf{A}}$ & 0 \\ \hline
\end{tabular}
\label{Tab:ops_parameters}
\end{table}

The second set uses the same starting months but extends the horizon to 36 months ($T=36$).  
For the first 24 maturities, both the initial forward curve and the loading coefficients replicate the values from the 24-month calibration. We extrapolate the remaining 12 maturities by repeating the data corresponding to maturities 12--23. All other operational parameters remain identical to those in Table~\ref{Tab:ops_parameters}.

\paragraph{Basis function architecture.}
At stage~$t$, the value function approximation includes a constant term, all spot and futures prices across the three commodities, and $B - 1 - 3(T-t)$ random Fourier basis functions.  
We define a vector of random coefficients as
\[
\theta := (\theta_0,\theta_0^{\mathrm{E}},\ldots,\theta_{T-1}^{\mathrm{E}},\theta_0^{\mathrm{C}},\ldots,\theta_{T-1}^{\mathrm{C}},\theta_0^{\mathrm{NG}},\ldots,\theta_{T-1}^{\mathrm{NG}})\in\mathbb{R}^{3T+1}.
\] 
The intercept $\theta_0$ is drawn from $\text{Unif}[-\pi,\pi]$, and each $\theta_q^{c}$ for commodity $c\in\mathcal{C}$ and $q\in\tset$ is sampled independently from $N(0,\varrho)$, where $\varrho$ is a bandwidth parameter.  
The resulting random Fourier basis function at stage~$t$ is
\[
\phi(\mathbf{\sx}_t,\theta) = 
\cos\!\left(\theta_0 + \sum_{c\in\mathcal{C}}\sum_{q=0}^{T-t-1}\theta_q^c\,\sx_{t,t+q}^c\right).
\]
The value function approximation is then given by
\begin{equation}\label{eq:MEP_VFA}
\hat{V}_t(\se_t,\mathbf{\sx}_t)
= \beta_{\se_t,0}
+ \sum_{b=1}^{B-3(T-t)-1}\beta_{\se_t,b}\,\phi_b(\mathbf{\sx}_t;\theta^b)
+ \sum_{c\in\mathcal{C}}\sum_{s=t}^{T-1}\beta_{\se_t,s-t}^c\,\sx_{t,s}^c,
\end{equation}
where $\theta^b$ denotes the $b$th independently drawn coefficient vector and the $\beta$ terms are the associated weights.
The first term represents the constant; the second term aggregates the random Fourier bases, and the last term captures the linear dependence on spot and futures prices across all maturities.
We include 15 random Fourier basis functions to achieve stable bound performance; hence, the total number of basis terms at stage $t$ is $3(T-t)+16$.

\paragraph{Training and benchmarks.}
We smooth the WTCA and PO models via the log-sum-exp transformation described in \S\ref{sec:EC-log-sum-exp-smoothing-curvature} and solve the resulting smoothed models using Algorithm \ref{alg:SBCD} and SGD, respectively. We select the smoothing parameter $\sigma$ from $\{1,10,20,\ldots,5000\}$ and tune the bandwidth parameter $\varrho$ over $\{10^5,10^4,\ldots,10^{-5}\}$. We rescale the payoff terms by $\lambda\in\{1,10,20,\ldots,100\}$ to match the scale of the Fourier bases. Hyperparameters are chosen by cross-validation to balance accuracy and computational time.

For each instance, PO uses the same basis function architecture and smoothing procedure as WTCA, and the two methods are compared under the same time budget. The time limits are five hours for the 24-month instances and nine hours for the 36-month instances. As an additional benchmark, we include LSM, using the same basis functions and 50,000 simulated trajectories.

We compute the greedy lower bounds and the information relaxation upper bounds as in \S\ref{scc:Details-BermudanOptions}. All bounds are evaluated on the same 100{,}000 sample paths across the three methods. The conditional expectation in the dual penalty is computed with an one-step inner sampling with 500 draws.

\subsection{Results}
\label{sc:MEP_Results}
The $T=24$ results are reported in Table~\ref{T:Results3-main} of \S\ref{scc:ethanol-main} and discussed there. Table~\ref{T:Results4} below reports the 36-stage results. The table format is identical to that of the main paper tables.

\begingroup
\renewcommand{\arraystretch}{1.2}
\begin{table}[ht!]
\centering
\caption{Upper and lower bound estimates (with standard errors) and relative percentage differences for LSM, PO, and WTCA in 36-stage ($\mathbf{T=36}$) ethanol production instances}
\begin{tabular}{ccccccc}
\hline\hline
         & \multicolumn{3}{c}{Upper Bound Estimate}         & \multicolumn{3}{c}{Lower Bound Estimate}          \\ \cline{2-7} 
Instance & LSM              & PO                & WTCA            & LSM               & PO               & WTCA             \\ \hline
Jan      & 57.27 (0.10)     & 60.38 (0.05)      & 54.32 (0.07)    & 50.21 (0.04)      & 50.70 (0.05)     & 50.78 (0.05)     \\
Feb      & 52.76 (0.10)     & 55.38 (0.05)      & 49.75 (0.07)    & 46.30 (0.04)      & 46.54 (0.05)     & 46.61 (0.05)     \\
Mar      & 84.47 (0.10)     & 85.73 (0.05)      & 80.28 (0.07)    & 78.63 (0.04)      & 78.75 (0.05)     & 78.76 (0.05)     \\
Apr      & 64.04 (0.10)     & 65.93 (0.05)      & 60.23 (0.07)    & 56.99 (0.04)      & 57.05 (0.05)     & 57.05 (0.05)     \\
May      & 54.85 (0.10)     & 55.39 (0.05)      & 51.99 (0.07)    & 47.92 (0.04)      & 48.02 (0.05)     & 48.04 (0.05)     \\
Jun      & 52.09 (0.10)     & 55.03 (0.05)      & 49.41 (0.07)    & 44.93 (0.04)      & 45.06 (0.05)     & 45.08 (0.05)     \\
Jul      & 41.99 (0.10)     & 45.78 (0.05)      & 39.96 (0.07)    & 36.33 (0.04)      & 36.36 (0.05)     & 36.76 (0.05)     \\
Aug      & 58.56 (0.10)     & 60.21 (0.05)      & 55.41 (0.07)    & 51.51 (0.04)      & 51.56 (0.05)     & 52.01 (0.05)     \\
Sep      & 61.69 (0.10)     & 63.83 (0.05)      & 58.21 (0.07)    & 54.61 (0.04)      & 54.79 (0.05)     & 55.21 (0.05)     \\
Oct      & 53.27 (0.10)     & 55.86 (0.05)      & 50.10 (0.07)    & 47.38 (0.04)      & 47.35 (0.05)     & 47.41 (0.05)     \\
Nov      & 49.92 (0.10)     & 50.01 (0.05)      & 47.44 (0.07)    & 44.00 (0.04)      & 44.08 (0.05)     & 44.12 (0.05)     \\
Dec      & 47.83 (0.10)     & 48.76 (0.05)      & 45.11 (0.07)    & 42.11 (0.04)      & 42.19 (0.05)     & 42.20 (0.05)     \\ 
\hline
         & \multicolumn{3}{c}{\%(UB - Best UB)/ (Best UB)} & \multicolumn{3}{c}{\%(Best UB - LB) / (Best UB)} \\ \cline{2-7} 
Instance & LSM            & PO              & WTCA          & LSM             & PO             & WTCA           \\ \hline
Jan      & 5.43           & 11.16           & 0.00          & 7.57            & 6.66           & 6.52           \\
Feb      & 6.05           & 11.32           & 0.00          & 6.93            & 6.45           & 6.31           \\
Mar      & 5.22           & 6.79            & 0.00          & 2.06            & 1.91           & 1.89           \\
Apr      & 6.33           & 9.46            & 0.00          & 5.38            & 5.28           & 5.28           \\
May      & 5.50           & 6.54            & 0.00          & 7.83            & 7.64           & 7.60           \\
Jun      & 5.42           & 11.37           & 0.00          & 9.07            & 8.80           & 8.76           \\
Jul      & 5.08           & 14.56           & 0.00          & 9.08            & 9.01           & 8.01           \\
Aug      & 5.68           & 8.66            & 0.00          & 7.04            & 6.95           & 6.14           \\
Sep      & 5.98           & 9.65            & 0.00          & 6.18            & 5.88           & 5.15           \\
Oct      & 6.33           & 11.50           & 0.00          & 5.43            & 5.49           & 5.37           \\
Nov      & 5.23           & 5.42            & 0.00          & 7.25            & 7.08           & 7.00           \\
Dec      & 6.03           & 8.09            & 0.00          & 6.65            & 6.47           & 6.45           \\ \hline\hline
\end{tabular}
\label{T:Results4}
\end{table}
\endgroup

\paragraph{Scaling to $T=36$.}
Table~\ref{T:Results4} summarizes results for the 36-stage instances.
WTCA continues to deliver the tightest upper bound in every month, and its advantage widens with the horizon.
PO's average upper-bound deviation increases from 3.1\% at $T=24$ to 9.5\% at $T=36$, while LSM's rises from 4.3\% to 5.7\%. The near tripling of PO's deviation under a 50\% horizon increase is consistent with the $\mathcal{O}(T)$ complexity of PO established in Proposition~\ref{prop:complexity-comparison}.

\paragraph{Policy quality at $T=36$.}
Policy performance is similar across methods at $T=36$. The average optimality gaps are 6.2\% for WTCA, 6.5\% for PO, and 6.7\% for LSM. WTCA and PO differ by less than 0.5 percentage points in 9 of 12 instances, and all three methods are within 1 percentage point of each other in the same 9 instances. Relative to $T=24$, where WTCA had a clearer advantage, the gaps compress at $T=36$.
A likely explanation is that long horizon instances tend to have similar early-stage decisions across methods. In particular, the initial months often favor suspending production, which narrow the performance differences.
The remaining differences occur in months when commodity spreads lie near switching thresholds, where the timing of mothballing and reactivation is most sensitive.
\looseness = -1

%%%%%%%%%%%%%%%%%
\end{document}